\documentclass[letterpaper,twoside]{article}

\usepackage[letterpaper,left=1in,right=1in,top=1.5in,bottom=1.5in]{geometry}

\newcommand{\myauthor}{David Gepner and Jeremiah Heller}

\newcommand{\mytitle}{The tom Dieck splitting theorem in equivariant motivic homotopy theory}

\title{\mytitle}
\author{ David Gepner\footnote{David
Gepner was supported
by DFG award GE 2504/1-1 and NSF award DMS-1714273.}~ and Jeremiah Heller\footnote{Jeremiah Heller was supported by DFG award HE6740/1-1 and NSF award DMS-1710966.}}
\date{}

\usepackage{amsmath, amscd, amsthm} 
\usepackage{amssymb, amsfonts}

\usepackage{fancyhdr}
\usepackage[shortlabels]{enumitem}
\setlist[enumerate, 1]{font=\upshape, nosep,  wide=0.5em}
\usepackage{eucal}
\usepackage{microtype}
\usepackage{mathrsfs}
\usepackage{dsfont}
\usepackage{marginnote} 
\usepackage{comment}

\usepackage[dvipsnames]{xcolor} 

\definecolor{refkey}{gray}{0.5}
\definecolor{labelkey}{gray}{0.5}
\definecolor{note}{rgb}{0.94, 0.99, 1.00}
\colorlet{myurlcolor}{Aquamarine}
\colorlet{mylinkcolor}{violet}
\colorlet{mycitecolor}{YellowOrange}
\definecolor{reference}{rgb}{.93,.51,.93}
\definecolor{citation}{rgb}{1,.68,.26}
\definecolor{mrnumber}{rgb}{.80,.40,0}

\newcommand\myshade{85}

\usepackage[pdfstartview=FitH,
            pdfauthor={\myauthor},
            pdftitle={\mytitle},
            colorlinks,
            linkcolor=mylinkcolor!\myshade!black,
            citecolor=citation!\myshade!black,
            urlcolor=myurlcolor!\myshade!black,
            backref=page,
			bookmarks, 
			bookmarksopen=true, 
			bookmarksopenlevel=3, 
			bookmarksdepth=3]{hyperref}

\usepackage{tikz}
\usetikzlibrary{cd}
\usepackage[all]{xy} 

\usepackage{booktabs}

\usepackage{aliascnt} 
\usepackage{marginfix}

\usepackage[textsize=scriptsize, color=note, linecolor=lightgray]{todonotes} 

\newcommand{\A}{\mathds{A}}
\newcommand{\V}{\mathds{V}}
\renewcommand{\P}{\mathds{P}}
\newcommand{\Z}{\mathds{Z}}
\newcommand{\N}{\mathds{N}}

\renewcommand{\L}{\mathrm{L}}
\newcommand{\G}{\mathds{G}}

\newcommand{\MM}{\mathcal{M}}
\newcommand{\NNN}{\mathcal{N}}
\newcommand{\NN}{\mathrm{N}}
\newcommand{\W}{\mathrm{W}}

\newcommand{\TT}{\mcal{T}}

\newcommand{\VV}{\mcal{V}}
\newcommand{\WW}{\mcal{W}}

\newcommand{\OO}{\mathcal{O}}

\renewcommand{\AA}{\mathcal{A}}
\newcommand{\BB}{\mathcal{B}}
\newcommand{\CC}{\mathcal{C}}
\newcommand{\DD}{\mathcal{D}}

\renewcommand{\S}{\mathcal{S}}

\renewcommand{\SS}{\mathds{1}}

\newcommand{\oX}{\overline{X}}

\newcommand{\iso}{\cong}
\newcommand{\wkeq}{\simeq}

\newcommand{\Sch}{\mathrm{Sch}}
\newcommand{\Sm}{\mathrm{Sm}}

\newcommand{\coev}{\mathrm{coev}}

\renewcommand{\i}{\infty}

\newcommand{\SH}{\mathcal{S}\mathrm{pt}}
\newcommand{\HH}{\mathcal{S}\mathrm{pc}}
\newcommand{\bb}{\bullet}
\newcommand{\spt}{\SH}

\newcommand{\set}{\mathrm{Set}}

\newcommand{\pt}{\mathrm{pt}}
\newcommand{\ev}{\mathrm{ev}}

\newcommand{\hh}{\mathbf{h}}
\newcommand{\tte}{\mathbf{t}}
\newcommand{\mgf}{{\bf \Phi}}

\newcommand{\Fil}{\mathrm{Fil}}
\newcommand{\op}{\mathrm{op}}

\newcommand{\Pre}{\mathcal{P}}
\newcommand{\Shv}{\mathcal{S}\mathrm{hv}}

\newcommand{\FF}{\mathcal{F}}
\newcommand{\E}{\mathcal{E}}

\newcommand{\Eg}{\mathbf{E}}
\newcommand{\EE}{\mathbf{E}}

\newcommand{\Es}{\mathrm{E}_{\bullet} }

\newcommand{\stab}{\mathrm{Stab}}
\newcommand{\Stab}{\mathrm{Stab}}
\newcommand{\Mod}{\mathrm{Mod}}

\newcommand{\Th}{\mathrm{Th}}
\newcommand{\id}{\mathrm{id}}

\newcommand{\Rep}{\mathrm{Rep}}
\newcommand{\Sph}{\mathrm{Sph}}

\newcommand{\mcal}[1]{\mathcal{#1}}
\newcommand{\ul}[1]{\underline{\smash{#1}}}
\newcommand{\cd}{\smash \cdot}

\newcommand{\co}{\mathrm{co}}

\newcommand{\Orb}{\mathrm{Orb}}

\newcommand{\wt}[1]{\widetilde{#1}}

\newcommand{\et}{\acute{e}t}

\renewcommand{\emptyset}{\varnothing}

\newcommand{\Nis}{\mathrm{Nis}}
\newcommand{\mot}{\mathrm{mot}}

\newcommand{\all}{\mathrm{all}}

\newcommand{\pr}{\mathrm{pr}}
\newcommand{\fpr}{\mathrm{fpr}}
\newcommand{\red}{\mathrm{red}}

\newcommand{\fppf}{\mathrm{fppf}}

\newcommand{\Ntriv}{N\textrm{-}\mathrm{triv}}

\newcommand{\Nfree}{N\textrm{-}\mathrm{free}}
\newcommand{\Hfree}{H\textrm{-}\mathrm{free}}

\newcommand{\minus}{\smallsetminus}

\newcommand{\can}{\mathrm{can}}

\DeclareMathOperator*{\colim}{\mathrm{colim}}

\DeclareMathOperator{\spec}{\mathrm{Spec}}
\DeclareMathOperator{\proj}{\mathrm{Proj}}

\DeclareMathOperator{\Hom}{Hom}

\DeclareMathOperator{\Sym}{Sym}

\DeclareMathOperator{\Sing}{\mathrm{Sing}}

\DeclareMathOperator{\map}{Map}

\DeclareMathOperator{\Fun}{Fun}
\newcommand{\stable}{\ul{\mathrm{Map}}}

\def\sCirclearrowleft{\ensuremath{%
		\rotatebox[origin=c]{90}{$\scriptstyle \circlearrowleft$}}}
\newcommand{\eSch}{\Sch^{\sCirclearrowleft}}
\newcommand{\CAlg}{{\rm CAlg}}
\newcommand{\Cat}{{\rm Cat}}
\newcommand{\LPr}{\mathcal{P}{\mathrm{r}}^{\mathrm{L}}}
\newcommand{\LPrx}{\mathcal{P}{\mathrm{r}}^{\mathrm{L},\otimes}}
\newcommand{\Cati}{\mathrm{Cat}_{\infty}}

\newcommand*\circled[1]{\tikz[baseline=(char.base)]{
		\node[shape=circle,draw,inner sep=2pt] (char) {#1};}}
\newcommand{\ctext}[1]{\text{\makebox[0pt]{#1}}}

 \numberwithin{equation}{section} 

\theoremstyle{plain}
\newtheorem*{theorem*}{Theorem}

\newaliascnt{theorem}{equation}  
\newtheorem{theorem}[theorem]{Theorem}  
\aliascntresetthe{theorem}

\newaliascnt{proposition}{equation}  
\newtheorem{proposition}[proposition]{Proposition}
\aliascntresetthe{proposition}

\newaliascnt{lemma}{equation}  
\newtheorem{lemma}[lemma]{Lemma}
\aliascntresetthe{lemma}

\newaliascnt{corollary}{equation}  
\newtheorem{corollary}[corollary]{Corollary}
\aliascntresetthe{corollary}

\newaliascnt{claim}{equation}  

\aliascntresetthe{claim}

\newaliascnt{conjecture}{equation}  

\aliascntresetthe{conjecture}

 \theoremstyle{definition}

\newaliascnt{definition}{equation}  
\newtheorem{definition}[definition]{Definition}
\aliascntresetthe{definition}

\newaliascnt{example}{equation}  
\newtheorem{example}[example]{Example}
\aliascntresetthe{example}

\newaliascnt{remark}{equation}  
\newtheorem{remark}[remark]{Remark}
\aliascntresetthe{remark}

\newaliascnt{assumption}{equation}  

\aliascntresetthe{assumption}

\newaliascnt{condition}{equation}  
\newtheorem{condition}[condition]{Condition}
\aliascntresetthe{condition}

\newaliascnt{notation}{equation}  
\newtheorem{notation}[notation]{Notation}
\aliascntresetthe{notation}

\newcommand{\aref}[1]{\autoref{#1}}

\begin{document}

\maketitle  

\begin{abstract}
 We establish,  in the setting of equivariant motivic homotopy theory for a finite group, a version of tom Dieck's splitting theorem for the fixed points of a suspension spectrum. Along the way we establish structural results and constructions for equivariant motivic homotopy theory of independent interest. This includes geometric fixed point functors and the motivic Adams isomorphism.

    \paragraph{Mathematics Subject Classification 2010.}
    Primary:
    \href{https://mathscinet.ams.org/msc/msc2010.html?t=14Fxx&btn=Current}{14F42},
    \href{https://mathscinet.ams.org/msc/msc2010.html?t=55Pxx&btn=Current}{55P91}.
    Secondary:
    \href{https://mathscinet.ams.org/msc/msc2010.html?t=55Pxx&btn=Current}{55P42},
    \href{https://mathscinet.ams.org/msc/msc2010.html?t=55Pxx&btn=Current}{55P92}.
 
 \end{abstract}

\tableofcontents

\section{Introduction}

In his 1970 ICM talk \cite{Seg70}, Segal sketched a computation of the endomorphism ring of the equivariant sphere spectrum for a finite group $G$, identifying this endomorphism ring with the Burnside ring of finite $G$-sets. Using other methods, this computation  was recovered and massively generalized by tom Dieck's splitting theorem \cite{tD75}.
These results form a crucial layer of the foundations  on which the successes of equivariant homotopy in the ensuing decades were built, from early foundations \cite{LMS}, to Carlsson's resolution of the Segal completion conjecture \cite{Car}, to the Hill-Hopkins-Ravenel solution of the Kervaire invariant one problem \cite{HHR}.

 An equivariant version of motivic homotopy theory 
 was introduced by Voevodsky
 \cite{Deligne} to study quotients of motives by group actions, which played a role in his work on the Bloch-Kato conjecture. A number of authors have subsequently further developed this theory, and variants,  including
 Hu-Kriz-Ormsby \cite{HKO}, Heller-Krishna-{\O}stv{\ae}r \cite{HKO:EMHT}, Herrmann \cite{Herrmann}, and Carlsson-Joshua \cite{CJ}. The state-of-the-art is  Hoyois's   \cite{Hoyois:6}, where he develops the formalism of   Grothendieck's six operations in this theory.

 In this paper we establish an analogue of tom Dieck's splitting in the context of stable motivic homotopy theory for finite group actions. Throughout, we assume that $G$ is a finite group whose order is coprime to the characteristics of the residue fields of the base-scheme $B$; in other words, the group scheme associated to $G$ is linearly reductive over $B$.
Our splitting theorem, proved in \aref{thm:mtd} below, computes the $N$-fixed points of suspension spectra (more generally of ``split spectra") as a motivic $G/N$-spectrum, where $N\unlhd G$ is a normal subgroup. In case $N=G$, this takes the following form, where $(H)$ denotes the conjugacy class of a subgroup.
 
\begin{theorem}[Motivic tom Dieck splitting]\label{thm:imtd}
Let $G$ be a finite group whose order is invertible on $B$. Let $X$ be a based motivic $G$-space over $B$. There is an equivalence of motivic spectra
$$
\Theta_X:\bigoplus_{(H)} 
 \left(\Sigma^\i( X^{H}) \right)_{\hh\W H}
\xrightarrow{\sim} (\Sigma^\i X)^G
$$ 
\end{theorem}

The reader familiar with tom Dieck's theorem \cite{tD75} will recognize this result as taking a very similar  form as the classical result. The key difference here is that 
the functor $(-)_{\hh G}$, is an algebro-geometric, or motivic, version of the homotopy orbits functor rather than the familiar categorical construction. Recall that the ordinary homotopy orbits functor is defined as follows. A $G$-spectrum $Y$ determines a diagram on the category $BG\simeq B_{{\rm Nis}}G$ and the  homotopy orbits is the colimit of this diagram: $Y_{hG}\simeq \colim_{BG} Y$. The motivic version should then be thought of as a motivic, or parameterized, colimit of $Y$ over the category $B_{\et}G$ of \'etale $G$-torsors. We don't make the version of the definition, as just stated, precise here, but instead provide a direct construction of the functor $(-)_{\hh G}$.  
 First, recall that Morel-Voevodsky \cite{MV:A1} introduce a geometric model for the classifying space of \'etale $G$-torsors. This construction is distinct from the usual simplicial construction of the classifying space; rather, the simplicial construction is a model for the classifying space of Nisnevich $G$-torsors.  The equivariant manifestation of this fact is that the universal free motivic $G$-space $\EE G$ is not equivalent to the usual simplicial construction $\mathrm{E}_\bb G$.
The motivic homotopy orbits of a $G$-spectrum $Y$ is defined here by a variant on the standard formula $Y_{hG}\simeq (\mathrm{E}_\bb G_+\otimes Y)/G$, obtained by replacing the use of $\mathrm{E}_\bb G$ by $\EE G$. That is, we take $Y_{\hh G} \simeq (Y\otimes \EE G_+)/G$ as the definition of the motivic homotopy orbits of $Y$.

  Before explaining the intermediate results leading to the splitting theorem,   
 we pause to point out an obvious, but important, difference between ordinary equivariant and motivic equivariant homotopy theory.  
  In the topological case, equivalences are detected by the fixed point functors for subgroups $H\leq G$. This corresponds to the fact that  a set of generators is given by the orbits $G/H$, or that the homotopy theory of $G$-spaces can be presented as presheaves of spaces on the category of $G$-orbits.
On the other hand,  generators for equivariant motivic homotopy theory  are smooth schemes over $B$ with a $G$-action. Orbits $G/H$ are examples
 of smooth $G$-schemes over $B$, but of course there are many more.  Equivalences between motivic $G$-spaces or $G$-spectra are not detected by fixed points,\footnote{Of course one can define a homotopy theory which has this property, but as pointed out by Herrmann \cite{Herrmann}, equivariant algebraic $K$-theory is not representable in the resulting homotopy category.}  because smooth $G$-schemes cannot in any meaningful way be decomposed into pieces of the form $G/H\times X$ (where $X$ has trivial action). 
However, 
 by analyzing filtrations of equivariant motivic homotopy theory arising from localizations and colocalizations determined by families of subgroups, as in \aref{sec:universalF},
one can see that  
 equivalences can be detected using only (desuspensions of) smooth $G$-schemes of very special form, namely those of the form $G\times_{K}X$ such that there is a normal subgroup $N\unlhd K$ which acts trivially on $X$ and the quotient $K/N$ acts freely on $X$. These are the $G$-schemes whose stabilizers are concentrated at a single conjugacy class.

In addition to the  six functor formalism established in \cite{Hoyois:6}, the proof of \aref{thm:imtd} relies on several new results for equivariant motivic homotopy theory, which should be of independent interest. 
As with the splitting theorem, there are versions for all of these results established relative to a normal subgroup $N\trianglelefteq G$; for simplicity we discuss here in the introduction only the results for $N=G$.

The first key ingredient which we need is the geometric fixed points functor, constructed in \aref{sub:gfp}. 
The
 geometric fixed points $X^{\mgf G}$ of a motivic $G$-spectrum $X$ may be obtained as the $G$-fixed points of a suitable localization of $X$, namely one determined by smooth $G$-schemes with trivial action. This functor satisfies analogues of the  main features of the  geometric fixed points functor from ordinary equivariant motivic homotopy theory, as follows. 
\begin{enumerate}
\item It is a symmetric monoidal left adjoint and  $(\Sigma^\i Y_+)^{\mgf G}\simeq \Sigma^\i(Y^G_+)$ for any $Y\in \Sm_B$. 
\item $X^{\mgf G}\simeq (X\otimes \wt \EE \mcal{P})^G$ where $\wt\EE\mcal{P}$ is the 
 unreduced suspension of the universal motivic $G$-space associated to the family $\mcal{P}$ of proper subgroups of $G$.
\item  $X^{\mgf G}\simeq (X[a^{-1}])^G$, where $a$ is the euler class $a:S^0\to T^{\overline{\rho}_G}$ and $\overline{\rho}_G$ is the reduced regular representation.
\end{enumerate}
 Here, given a representation $V$, we write $T^V$ for the associate motivic sphere (i.e, its Thom space). The connection between the second two items is provided by  a geometric presentation  for universal motivic $G$-spaces for families, established in \aref{sec:universalF}, analogous to Morel-Voevodsky's geometric description of classifying spaces. In particular, for the family of proper subgroups, what we find is that $\wt\EE\mcal{P}$ may be described by the  formula 
$\wt\EE\mcal{P}\simeq T^{\i \overline{\rho}_G}=\colim_n T^{n\overline{\rho}_G}$, analogous to the familiar formula from topology.

Also of interest is that the construction of the geometric fixed points functor here permits a motivic version of the Tate square for $C_p$-equivariant motivic spectra in \aref{sub:Tate}.  This is a homotopy pushout square of motivic spectra
\[
\begin{tikzcd}
	X^{C_{p}} \ar[r] \ar[d] & X^{\mgf C_{p}} \ar[d] \\
	X^{\hh C_{p}} \ar[r] & X^{\tte C_{p}},
	\end{tikzcd}
\]
where $X^{\hh C_{p}}$ is a motivic version of the homotopy fixed points functor, defined using $\EE C_p$.

A second key ingredient entering into the splitting theorem is the motivic Adams isomorphism,  proved in \aref{thm:adams} below. This fundamental result identifies the quotient of a free $G$-spectrum with its fixed points. 
There is a natural transformation from the former to the latter and the bulk of the work in the section \aref{sec:adams} is devoted to verifying that this morphism is an equivalence.
Our strategy is to first check that this transformation is an equivalence on the full subcategory of dualizable free $G$-spectra. 
Of course, unless the base is a field of characteristic zero, this doesn't suffice to conclude the result in general. But since $\EE G_+$ is a colimit of dualizable spectra,
it does imply that the  fixed points of $\EE G_+$ coincides with $\mathbf{B}G_+$. Using that $\mathbf{B}G_+$ contains $\SS_B$ as a summand, this lets us define an inverse to the Adams transformation to obtain the general result.  
It is worth pointing out that if $f:T\to B$ is an \'etale torsor, then the Adams isomorphism for $f_\#\SS_T$  is a straightforward consequence of ambidexterity, proved in \cite{Hoyois:6}, for the finite \'etale map $f$.  An obvious strategy presents itself. If $q:X\to B$ is a smooth $G$-scheme over $B$ with free action, then $g:X\to X/G$ is an \'etale torsor and $p:X/G\to B$ is smooth. Since $(g_\#\SS_X)^G\simeq (g_\#\SS_X)/G$ and $p_\#$ commutes with the quotient functor, to 
verify the Adams isomorphism for $p_\#(g_\#\SS_X)$, it would suffice to check that the fixed points functor commutes with $p_\#$. 
Establishing this change-of-base formula directly appears to be as difficult as  the Adams isomorphism itself and we actually obtain this base change formula as \emph{a consequence} of the Adams isomorphism. 
It is interesting to note that from the viewpoint of motivic homotopy theory of stacks, this is an instance of smooth-proper base change formula, along the non-representable map $B G\to B$.

Once all of the foundational results are in place, the proof of the splitting theorem is actually fairly straightforward. It is not hard to write down the map $\Theta_X$ of the statement of the theorem and to check it is an equivalence it suffices to check it is an equivalence when $X$ is concentrated at a single conjugacy class, a case which follows from the analysis in \aref{sec:universalF} of localizations and colocalizations of equivariant motivic homotopy.

\paragraph{Outline:}
We utilize the language of $\infty$-categories throughout this paper.
We begin in \aref{sec:EMHT} by recalling the construction of the $\infty$-categories of motivic $G$-spaces and motivic $G$-spectra from \cite{Hoyois:6}, as well as a few extensions used in this paper. In \aref{sec:universalF} we  study the colocalizations and localizations of equivariant motivic homotopy theory which are determined by a family. 
In \aref{sec:fp} we define fixed point functors and the geometric fixed points. In \aref{sec:quotient} we define the quotient functor on $N$-free spectra and in  \aref{sec:adams} we prove the motivic Adams isomorphism. Finally in \aref{sec:mtds}, we prove the motivic tom Dieck splitting theorem.

\paragraph{Acknowledgements:}
This project was begun  awhile ago. During its   long gestation period we have had the  pleasure and benefit of many interesting and helpful conversations on the material in this paper.  
We especially thank Tom Bachmann,  Elden Elmanto, 
 Christian Haesemeyer, Marc Hoyois, Niko Naumann, Markus Spitzweck, and Paul Arne {\O}stv{\ae}r, and 
the 2016 WiT team:  Agnes Beaudry, Kathryn Hess, Magda Kedziorek, Mona Merling, Vesna Stojanoska.

\paragraph{Notation:} Throughout, $B$ is a quasi-compact, quasi-separated base scheme and $G$ is a finite group whose order is invertible in $\mcal{O}_{B}$.  
We write $\Sch^{G}_B$ for the category of $G$-schemes which are finitely presented and $G$-quasi-projective over $B$. 
For $S\in \Sch^G_B$, write $\Sch^G_S$ for the slice category over $S$ and $\Sm^G_S\subseteq \Sch^G_S$ for the full subcategory whose objects are smooth over $S$.

If $\mcal{E}$ is a locally free  $\mcal{O}_S$-module, we write 
\[
 \V_S(\mcal{E}) := \ul{\spec} (\Sym (\mcal{E}^{\vee})) \;\;\text{and}\;\;
 \P_S(\mcal{E}) := \ul{\proj}(\Sym (\mcal{E}^{\vee}))
\]
respectively for the associated vector bundle scheme and the associated projective bundle on $S$. A representation of $G$ over $B$ is a locally free  $G$-module on $B$. 
If $M$ is a $G$-set we let 
\[
\rho_{M} = \mcal{O}_{B}[M]:=\mcal{O}_{B}\otimes\Z[M]
\]
denote the associated permutation representation. In particular, $\rho_{G}$ is the regular representation.

We use the language of $\i$-categories in this paper and  mostly follow the terminology in \cite{HTT, HA} with the exception that we write $\Cati$ for the $\i$-category of not necessarily small $\i$-categories.

\section{Equivariant motivic homotopy theory}\label{sec:EMHT}

We recall definitions and  basic properties of equivariant motivic homotopy theory.  
We will use the $\infty$-categorical approach to equivariant motivic spectra  introduced
by Hoyois \cite{Hoyois:6}. 
Since we are working with finite groups, the unstable homotopy category agrees with those constructed by Voevodsky \cite{Deligne} and Heller-Krishna-{\O}stv{\ae}r \cite{HKO:EMHT} and the stable homotopy category agrees with the one  from \cite[Appendix A.4]{HVO:BL}.

\subsection{Equivariant geometry}

Let $S\in \Sch^G_B$. If $\phi:G\to K$ is a group homomorphism, we write
\[
\phi^{-1}:\Sch^K_S\to \Sch^G_{\phi^{-1}S}
\]
for the restriction functor, which regards a $K$-scheme over $S$ as a $G$-scheme over $S$ via $\phi$. When no confusion should arise, we write $\Sch^G_{S}$ instead of $\Sch^G_{\phi^{-1}S}$.

Let $X\in \Sch^G_B$. 
 Say that $G$ {\em acts freely} on $X$ if the action of $G(T)$ on $X(T)$ is free for any $T\in \Sch_B$. 
If $G$ acts freely on $X$, then the fppf-quotient $(X/G)_{\fppf}$ is representable by an object $X/G\in \Sch_B$ (since all of our schemes are quasiprojective), see \cite[\href{https://stacks.math.columbia.edu/tag/07S7}{Tag 07S7}]{stacks-project}. Since $G$ is smooth, the fppf-torsor $X\to X/G$ is an \'etale torsor (as it is a smooth map and so \'etale locally admits sections).

\begin{definition}
The stabilizer
 of a point $x\in X$ is the subgroup $\stab(x)\leq G$ defined by
\[
\stab(x) = \{g\in G \mid g\cd x = x\textrm{ and $g$ acts as $\id$ on $k(x)$} \}.
\]
\end{definition}
Then, $G$ acts freely on $X$ provided $\stab(x)=\{e\}$ for all $x\in X$.
More generally, if $H\leq G$ is a subgroup which acts freely on $X$ then the quotient $X/H$ inheirits an action of  the Weyl group 
$\W(H)=\W_G(H):=\NN_G(H)/H$, so defines a functor 
$(-)/H:\Sch^{G,\Hfree}_B \to \Sch^{\W H}_{B}.$
This is the composite of the restriction  functor $\Sch^{G,\Hfree}_B\to \Sch^{\NN H,\Hfree}_B$ and followed by the quotient $\Sch^{\NN H,\Hfree}_B \to \Sch^{\W H}_{B}$.

Let  $N\trianglelefteq G$ be a normal subgroup. 
Suppose that  either
\begin{enumerate}[(i)]
	\item $N$ acts trivially on $S$, or 
	\item $N$ acts freely on $S$. 
\end{enumerate}
In either case,  the 
quotient functor yields a functor
\[
\Sm^{G,\Nfree}_S\to \Sm^{G/N}_{S/N}.
\]

Write $q:S\to S/N$ for the quotient map of schemes and  $q^{-1}:\Sm^{G/N}_{S/N} \to \Sm^{G/N}_S$ for the functor defined by $q^{-1}(Y) =  Y\times_{S/N}S$.

\begin{proposition}\label{prop:qe}
	Suppose that $N$ acts freely on $S$, then $(-)/N$ and $\pi^{-1}q^{-1}$ are inverse equivalences
	\[
	(-)/N:\Sm^G_S \rightleftarrows \Sm^{G/N}_{S/N}:\pi^{-1}q^{-1}.
	\]
\end{proposition}
\begin{proof}
	Let $f:X\to S$ be in $\Sm^G_S$. By descent we have a cartesian square in $\Sm_S$ (hence in $\Sm^G_S$)
	\[
	\begin{tikzcd}
	X\ar[r]\ar[d] & S\ar[d] \\
	X/N\ar[r] & S/N.
	\end{tikzcd}
	\] 
It follows that $(-)/N$ is fully faithful. It is also essentially surjective since if $Y\in \Sm^{G/N}_{S/N}$ then $Y\cong (Y\times_{S/N}S)/N$.
\end{proof}

Let $\iota:H\hookrightarrow G$ be a monomorphism of groups. The induction-restriction adjunction
\[
\iota_!:\Sch^H_S\rightleftarrows \Sch^G_S:\iota^{-1}
\]
restricts to 
an adjunction
\[
\iota_!: \Sm^H_S\rightleftarrows \Sm^G_S: \iota^{-1}.
\]
When $H\leq G$ is a subgroup and $\iota$ is the inclusion, we often write 
$
\iota_!(X) = G\times_{H} X
$ 
and this scheme is described concretely as follows. 
The scheme $G\times X$ becomes an $H$-scheme under the action $h(g,x) = (gh^{-1}, hx)$ and we define  
$$
G\times_{H} X = (G\times X)/H.
$$
The scheme $G\times_H X$ has a left $G$-action through the action of $G$ on itself. We can describe
$G\times_{H}X$ in slightly more concrete terms as follows.  Choose a complete set of left coset representatives $g_i$, then  $G\times_{H}X= \coprod_{g_{i}}X_i$, each $X_i$ is a copy of $X$ and $g\in G$ acts as $k:X_{i}\to X_{j}$ where $k\in H$ satisfies $gg_i=g_jk$.

Let $X\in \Sch^G_B$. The presheaf of sets $X^G$ is the presheaf of sets on $\Sch_B$ defined by 
 \[
 X^G(Y) = \{ y\in X(Y)\mid \textrm{$y$ is fixed by $G$}\}.
 \] 
If $X\to B$ is seperated, the presheaf $X^G$ is represented by a closed subscheme of $X$ which is finitely presented over $B$, which is moreover smooth over $B$ if $X$ is, \cite[Proposition XII.9.2, Corollaire XII.9.8]{SGA3} or \cite[Proposition A.8.10]{CGP:pseudo}.
Note that if $H\leq G$ is a subgroup, the fixed point subscheme $X^{H}$ comes equipped with an action of the Weyl group $\W(H)$.

Now, suppose that  $N\unlhd G$ is a normal subgroup which acts trivially on $S$.  Write $\pi:G\to G/N$ for the quotient map. Restricting action along $\pi$ defines a functor 
$\pi^{-1}:\Sch_S^{G/N} \to \Sch^{G}_S$, which is left adjoint to fixed points. 
We will usually simply write again $X$ instead of $\pi^{-1}X$, whenever context makes the meaning clear. 
Now, restricting attention to smooth $S$-schemes we obtain the adjunction
\[
\pi^{-1}:\Sm_S^{G/N}\rightleftarrows \Sm^{G}_{S}:(-)^N.
\]

 \subsection{Families of subgroups}
Families of subgroups provide a convenient way to filter  equivariant motivic homotopy theory. 

\begin{definition}\label{defn:family}
A \emph{family} $\FF$ of subgroups of $G$ is a set of subgroups which is 
closed under taking subgroups and conjugation. 
\end{definition}

\begin{example}
The following families play an important role. 
\begin{enumerate}
\item The trivial family $\FF_\mathrm{triv} := \{e\}$. 
\item The family of all subgroups $\FF_\mathrm{all} := \{H\leq G \mid \textrm{$H$ is a subgroup}\}$.
\item The family of proper subgroups $\mcal{P}:=\{H \lneq G \mid \textrm{$H$ is a proper subgroup}\}$.
\item For a normal subgroup $N \trianglelefteq G$, define $\FF[N] :=\{H\leq G \mid N \not\subseteq H \}$.  Note that $\mcal{P} = \FF[G]$.
\item For a normal subgroup $N\trianglelefteq G$, define $\FF(N):=
\{H\leq G \mid H\cap N = \{e\}\}$. Note that $\FF({G}) = \FF_\mathrm{triv}$.
\end{enumerate}
\end{example}
 
If $\FF$ is a family, we write $\co(\FF):=\FF_\mathrm{all} - \FF$ for its complement. Note  that $\co(\FF)$ is not a family.\footnote{Rather, it is a {\em cofamily}, meaning it is closed under conjugation and $K\in \co(\FF)$ whenever $K$ contains a subgroup in $\co(\FF)$}
  
\begin{remark}
	A family of subgroups can be equivalently viewed as a sieve on the orbit category $\Orb^G$. The sieve corresponding to $\FF$ is the full subcategory $\Orb^{G}[\FF]\subseteq \Orb^G$ of orbits such that  $G/H\in\Orb^{G}[\FF]$ if and only if $H\in \FF$. 
\end{remark}

A family $\FF$ determines a sieve on $\Sm^G_S$ by letting $\Sm^G_S[\FF]\subseteq \Sm^G_S$ be the full subcategory whose objects are smooth $G$-schemes over $S$ such that all stabilizers are contained in $\FF$. It is useful to make the following more general definition.
\begin{definition}\label{def:Ecat}
 Let $\mathcal{E}$ be a set of subgroups of $G$ which is closed under conjugacy. Write $\Sm_S^G[\mathcal{E}]\subseteq \Sm_S^G$ for the full subcategory whose objects are those smooth $G$-schemes $X$ over $S$ such that $\stab(x)\in \mathcal{E}$ for every point $x\in X$.
\end{definition}

\begin{notation}\label{not:useful}
 Let $X\in \Sch^G_S$. Write 
 \[
 X^{\FF} := \bigcup_{H\in\co(\FF)} X^{H}
 \]
 and
 \[
 X(\FF) := X - X^{\FF}.
 \]
\end{notation}
The subset $X(\FF)\subseteq X$ is the set of points  whose stabilizers are  in $\FF$. Observe that $X(\FF)\subseteq X$ an open invariant subscheme and $X^{\FF}\subseteq X$ is a closed invariant subscheme since  $X^{\FF}$ is a finite union of closed subschemes.

\subsection{Motivic \texorpdfstring{$G$}{G}-spaces}

Recall that the Nisnevich topology is defined via a $cd$-structure.

\begin{definition}[{\cite{V:cd}}]
Let $\CC$ be a small category which has an initial object $\emptyset$.

\begin{enumerate}
\item A $cd$-structure on $\CC$ is a collection $\mathcal{A}$ of commutative squares in $\CC$ such that if $Q\in \mathcal{A}$ and if $Q'$ is isomorphic to $Q$, then $Q'\in \mathcal{A}$. 

\item Given a $cd$-structure $\mathcal{A}$ in $\CC$, the Grothendieck topology $t_{\mathcal{A}}$ generated by $\mathcal{A}$ is the smallest topology on $\CC$ such that 
\begin{enumerate}
\item the empty sieve is a covering sieve of $\emptyset$, and 
\item given any square in $\mathcal{A}$ 
\begin{equation*}
\begin{tikzcd}
V \ar[r] \ar[d] & Y \ar[d,"p"] \\
U \ar[r, "j"] & X,
\end{tikzcd} 
\end{equation*}
 the sieve generated by $\{U\to X, Y\to X\}$ is a covering sieve.
\end{enumerate}
\end{enumerate}

\end{definition}

\begin{definition}
An equivariant map $f:Y\to X$ is said to be \emph{fixed point reflecting} at $y\in Y$ if $f$ induces an isomorphism $\stab(y) \iso \stab(f(y))$. If this condition holds at every $y\in Y$, then $f$ is simply said to be \emph{fixed point reflecting}.
\end{definition}

Let $\CC_S\subseteq \Sch^G_S$ be a full subcategory containing $\emptyset$. 
We often require $\CC_S$ to satisfy one or both of the following properties:
\begin{enumerate}
 \item[(P)] If $Y\to X$ is  fixed point reflecting and \'etale,
 then $Y\in \CC_S$ whenever $X\in \CC_S$.
 \item[(H)] If $X\in \CC_S$ then so is $X\times_{S}\A^1$. 
\end{enumerate}

Our primary examples of interest are the categories $\Sm_S^G[\E]$, where $\E$ is a set of subgroups closed under conjugacy. More generally, we could also consider the following property. 
\begin{enumerate}
\item[(P$^\prime$)] If $Y\to X$ is an equivariant \'etale map, then $Y\in \CC_S$ whenever $X\in \CC_S$. 
\end{enumerate}

The condition (P) on $\CC_S$ guarantees that the fixed point Nisnevich $cd$-structure (defined below) is complete while the condition (P$^\prime$) guarantees that the Nisnevich $cd$-structure is complete. Some categories of interest in this paper, e.g., $\Sm_S^G[\co(\FF)]$ for a family $\FF$, do not satisfy (P$^\prime$) but do satisfy the  weaker property (P).  We will see below in \aref{prop:fpNis=N} that when $\CC_S$ satisfies (P$^\prime$) then the topology associated to the fixed point Nisnevich $cd$-structure coincides with the Nisnevich topology.

\begin{definition}
Let $\CC_S\subseteq \Sch^G_S$ be a full subcategory containing $\emptyset$. 
\begin{enumerate}
\item The \emph{Nisnevich} $cd$-structure ${\rm Nis}$ on $\CC_S$ consists of 
 cartesian squares
\begin{equation}\label{eqn:EDS}
\begin{tikzcd}
V \ar[r, hookrightarrow] \ar[d] & Y \ar[d, "p"] \\
U \ar[r, hookrightarrow, "j"] & X,
\end{tikzcd} 
\end{equation}
where $j$ is open immersion, $p$ is \'etale, and the map 
$(Y- V)_{\rm red} \to (X- U)_{\rm red}$ is an isomorphism in $\Sch_S$.

\item The \emph{fixed point Nisnevich} $cd$-structure ${\rm fpNis}$ on $\CC_S$ consists of cartesian squares as in the Nisnevich $cd$-structure, but with the added condition that $p$ is a fixed point reflecting \'etale map.
\end{enumerate}
\end{definition}
\begin{remark}\label{rem:red}
In general, if $G$ is a flat group scheme over $B$, one may choose a scheme structure on $Z:=X\minus U$ so that $Z$ is invariant under the $G$-action \cite[Lemma 2.1]{Hoyois:6}. Since $G$ is a finite discrete group, $Z_{\red}$ is invariant and the map $p^{-1}(Z)_\red\to Z_\red$ is equivariant.
\end{remark}

\begin{proposition}\label{prop:fpNis=N}
Suppose that $\CC_S$ satisfies 	{\rm (P$^\prime$)}. The
 topology $t_{{\rm fpNis}}$ coincides with the Nisnevich topology on $\CC_S$.
\end{proposition}
\begin{proof}
Every $t_{\rm fpNis}$-cover is a Nisnevich cover. We show that the reverse implication holds. It suffices to show that any Nisnevich square 
$$
\begin{tikzcd}
V \ar[r, hookrightarrow] \ar[d] & Y \ar[d, "p"] \\
U \ar[r, hookrightarrow, "j"] & X.
\end{tikzcd} 
$$
admits a 
$t_{\rm fpNis}$-refinement. 

Write $\fpr(Y)\subseteq Y$ for the set of points where $p$ is fixed point reflecting. Since $Y\to X$ is unramified and $\stab(X)\to X$ is universally closed, \cite[Proposition 3.5]{Rydh} applies to show that the set $\fpr(Y)\subseteq Y$ is an invariant open subset. We have that 
$Y\minus V\subseteq \fpr(Y)$. It follows that the outer square
\begin{equation}\label{eqn:Nisrefine}
\begin{tikzcd}
 \fpr(V) \ar[r, hookrightarrow]\ar[d, hookrightarrow] & \fpr(Y) \ar[d, hookrightarrow, "i"] \\
V \ar[r, hookrightarrow] \ar[d] & Y \ar[d, "p"] \\
U \ar[r, hookrightarrow, "j"] & X
 \end{tikzcd}
\end{equation}
is a fixed point Nisnevich square (as is the top square). In particular, $\{j,pi\}$ is a  $t_{\rm fpNis}$-cover which refines 
$\{j, p\}$.
\end{proof}

Write $\Pre(\CC_S)$ for the $\infty$-category of presheaves of spaces on $\CC_S$.
Note that  $\CC_S$ does not necessarily contain a terminal object, in particular a terminal object $\Pre(\CC_S)$ is in general not representable. We write $\pt\in\Pre(\CC_S)$
 for this terminal object. Of course if $S\in\CC_S$ then $\pt$ is representable; it is the presheaf represented by $S$.

\begin{definition}
Say that $F\in \Pre(\CC_S)$ is \emph{Nisnevich excisive} if
  \begin{enumerate}
  \item $F(\emptyset)$ is contractible, and 
  \item  for any Nisnevich  square  \eqref{eqn:EDS} in  $\CC_S$, the square 
  \begin{equation*}
\begin{tikzcd}
F(X) \ar[r] \ar[d] & F(Y) \ar[d] \\
F(V) \ar[r] & F(U)
\end{tikzcd} 
\end{equation*}
is cartesian.
  \end{enumerate}
 
Write $\Shv_{\Nis}(\CC_S)\subseteq\Pre(\CC_S)$ for the subcategory of Nisnevich excisive presheaves of spaces on $\CC_S$.

\end{definition}

Temporarily say that $F$ is ``fixed point Nisnevich excisive" if $F(\emptyset)\simeq \pt$
and $F(\mathcal{Q})$ is cartesian for any fixed point Nisnevich square $\mathcal{Q}$. There is no real need for this extra terminology by the following proposition.

\begin{proposition}\label{prop:fp&Nis}
Let $F\in \Pre(\CC_S)$ and suppose $\CC_S$ satisfies {\rm (P$^\prime$)}. Then the following are equivalent.
\begin{enumerate}
\item $F$ is fixed point Nisnevich excisive.
\item $F$ is Nisnevich excisive.
\item $F$ is a sheaf for the Nisnevich topology.
\end{enumerate}
\end{proposition}
 \begin{proof}
If $\CC_S$ satisfies (P$^\prime$), then the $cd$-structures ${\rm fpNis}$ and ${\rm Nis}$ both satisfy the conditions of 
\cite[Theorem 3.2.5]{AHW1}, i.e. they are complete and regular, in the terminology of \cite{V:cd}. Together with \aref{prop:fpNis=N}, it follows that $(1)$, $(2)$ and $(3)$ are equivalent. 

 \end{proof}

\begin{definition}
Let $\CC_S\subseteq \Sch_S$ be a full subcategory which contains $\emptyset$ and satisfies properties (P) and (H).
Say that $F\in\Pre(\CC_S)$ is \emph{$\A^1$-homotopy invariant} if for any $Y\in \CC_S$, the projection map $\pi:Y\times\A^1\to Y$  induces an equivalence $F(Y)\wkeq F(Y\times \A^1)$. We write 
    $\Pre_{\A^1}(\CC_S)\subseteq\Pre(\CC_S)$ for the full subcategory consisting of the $\A^1$-homotopy invariant presheaves.
\end{definition}

The property that a presheaf is Nisnevich excisive is defined by a small set of conditions. It follows from \cite[Section 5.5.4]{HTT} that the inclusion $\Shv_{\Nis}(\CC_S)\subseteq \Pre(\CC_S)$ is an accessible localization. 
Write $\L_{\Nis}$ for the resulting localization endofunctor on $\Pre(\CC_S)$.
Moreover, this localization is \emph{left-exact} in the sense of \cite[Section 6.2.2]{HTT}.

Similarly, the property that a presheaf is $\A^1$-homotopy invariant is defined by a small set of conditions, so that the inclusion $\Pre_{\A^1}(\CC_S)\subseteq \Pre(\CC_S)$ is also an accessible localization. Write $\L_{\A^1}$ for the resulting localization endofunctor. It can be described explicitly by the formula 
$\L_{\A^1}\simeq \Sing_{\A^1}$, where 

\begin{equation}\label{eqn:LA1}
\Sing_{\A^1}(F)(U) := \colim_{\Delta}(n\mapsto F(U\times\Delta^{n}_{S})).
\end{equation}

A map $f:F_1\to F_2$ in $\Shv_{\Nis}(\CC_S)$ is a \emph{Nisnevich equivalence} provided $\L_{\Nis}(f)$ is an equivalence and 
an \emph{${\A^1}$-equivalence} provided $\L_{\A^1}(f)$ is an equivalence.
 \begin{remark}
 	Say that $F\in \Pre(\CC_S)$ is {\em strongly $\A^1$-homotopy invariant} if for any projection $E\to X$ of a $G$-affine bundle in $\CC_S$, the induced map is an equivalence  $F(X)\wkeq F(E)$. 
 	Any $X\in \Sm^G_S$ is Nisnevich locally affine. This implies that  if $F$ is Nisnevich excisive, then $F$ is strongly $\A^1$-invariant if and only if it $\A^1$-homotopy invariant (as $E\to X$ always has local sections in this case), see \cite[Remark 3.13]{Hoyois:6}.

 In particular, the motivic localization considered here agrees with the one in \cite{Hoyois:6}.  
 \end{remark}

\begin{definition}
\begin{enumerate}
\item  A \emph{motivic $G$-space over $S$} is a 
 Nisnevich excisive and $\A^1$-homotopy invariant presheaf $F\in\Pre(\Sm^G_S)$.
\item Write $\HH^G(S)$ for the $\infty$-category of motivic $G$-spaces. The category of based motivic $G$-spaces is $\HH_\bb^G(S)=\HH^G(S)_{\pt/}$.
\item 
More generally,  write $\HH(\CC_S)$  for the $\infty$-categories of Nisnevich excisive $\A^1$-homotopy invariant presheaves on $\CC_S$
and $\HH_\bb(\CC_S)=\HH(\CC_S)_{\pt/}$.
 \end{enumerate}
 \end{definition}
When $\E$ is a set of subgroups, closed under conjugacy,  we use the notation
\begin{align*}
 \HH^{G,\E}(S)&:= \HH(\Sm^{G}_S[\E] )  \\ \HH_\bb^{G,\E}(S)&:=\HH_\bb(\Sm^{G}_S[\E] ). 
\end{align*}

The inclusion $\HH(\CC_S)\subseteq \Pre(\CC_S)$ is an accessible localization and we write 
\[
\L_{\mot}:\Pre(\CC_S)\to \Pre(\CC_S)
\] 
for the corresponding localization endofunctor.
The motivic localization functor may be computed by the formula 
\cite[Lemma 3.2.6]{MV:A1}
\begin{equation}\label{eqn:motformula}
\L_{\mot}(F) = \L_{\Nis}\colim_{n\to \infty} (\L_{\A^1}\circ \L_{\Nis})^{n}(F).
\end{equation}

\begin{proposition}
The motivic localization functor 
$\L_\mot:\Pre(\CC_S)\to\Pre(\CC_S)$  is locally cartesian\footnote{A localization endofunctor $\L:\DD\to \DD$ is {\em locally cartesian} if $\L(A\times_{B}X)\to A \times_{B}\L(X)$ is an equivalence for any maps $A\to B$, $X\to B$ in $\DD$ where  $A,B\in \L(\DD)$ \cite[\textsection 1]{GK15}. } and preserves finite products.
 In particular, colimits in $\HH^G(\CC_S)$ are universal. 
\end{proposition}

\begin{proof}
The localization functors $\L_{\Nis}$, $\L_{\A^1}$ satisfy these properties and therefore so does $\L_{\mot}$ using \eqref{eqn:motformula}.	

\end{proof}

We will make use of the notion of the tensor product of presentable $\infty$-categories as defined and studied in \cite[Section 4.8.1]{HA}, see especially \cite[Proposition 4.8.1.17]{HA}.
The category $\HH(\CC_S)$ is cartesian monoidal (i.e., it is symmetric monoidal with respect to the cartesian product). The symmetric monoidal product on $\HH(\CC_S)$
 extends to one on
\[
\HH_\bb(\CC_S)\simeq\HH(\CC_S)\otimes\S_\bullet,
\]
 where  $\S_\bullet:=\S_{\pt/}$ denotes the $\infty$-category of pointed spaces, see \cite[Lemma 3.6]{GGN} and \cite[Proposition 4.8.2.11]{HA}. 
We write $\wedge$ for the symmetric monoidal product on $\HH_\bb(\CC_S)$.
Sometimes we will need to use the symmetric monoidal product on $\Pre_{\bullet}(\CC_S):=\Pre(\CC_S)_{\pt/}$, which we denote $\wedge^{\Pre}$ to avoid confusion. 
The symmetric monoidal localization functor $\L_{\mot}:\Pre(\CC_S) \to \HH(\CC_S)$ induces a unique symmetric monoidal localization functor $\L_{\mot}:\Pre_{\bb}(\CC_S)\to \HH_{\bb}(\CC_S)$. In particular, given $X,Y\in\Pre_{\bullet}(\CC_S)$, then   
\[
\L_{\mot}(X)\wedge \L_{\mot}(Y)\simeq \L_{\mot}(X\wedge^{\Pre} Y).
\]

Given a functor $u:\CC \to \DD$,
the functor 
$u^*:\Pre(\DD) \to \Pre(\CC)$ defined by precomposition with $u$ has both a left adjoint $u_!$ as well as a right adjoint $u_*$,
\[
\begin{tikzcd}
\Pre(\CC)\ar[rr, "u_!", bend left] \ar[rr, "u_*"' , bend right]&  & \Pre(\DD) \ar[ll, "u^*"'].
\end{tikzcd}
\]

A functor $u:\CC\to \DD$ between small $\i$-categories equipped with Grothendieck topologies is called \emph{topologically cocontinuous}\footnote{This is called  ``cocontinuous" in \cite[III.2.1]{SGA4}. We follow the terminology in \cite{Khan:MV} to avoid confusion with the category theorist's terminology ``cocontinuous functor".}
 if for every $Y\in \CC$ and every covering sieve $R$ on $u(Y)$, the sieve $u^*R\times_{u^*u_!Y}Y\hookrightarrow Y$, consisting of arrows $Z\to Y$ such that $uZ\to uY$ factors through $R$, is a covering sieve on $Y$.

\begin{lemma}\label{lem:cocoeq}
Let $u:\CC\to \DD$ be a topologically cocontinuous functor between small $\i$-categories equipped with Grothendieck topologies $\tau_\CC$ and $\tau_\DD$ respectively. 
Then $u^*:\Pre(\DD)\to \Pre(\CC)$ preserves all $\tau_\DD$-local equivalences.
\end{lemma}
\begin{proof}
The set of $\tau_D$-local equivalences in $\Pre(\DD)$ is the 
closure under pushouts, small colimits, and $2$-out-of-$3$, of the set of maps $R\hookrightarrow X$, where $R$ is a covering sieve on $X\in \DD$. Since $u^*$ preserves colimits, it suffices to show that $u^*R\to u^*X$ is a $\tau_\CC$-local equivalence. Since colimits are universal in $\Pre(\CC)$, it suffices to show that $u^*R\times_{u^*X}Y\to Y$ is an equivalence for any map $Y\to u^*X$ where $Y\in \CC$. To see this, note that it is in fact a covering sieve. Indeed,   $R\times_{X}uY\to uY$ is a covering sieve and therefore 
$u^*( R\times_{X}uY)\times_{u^*u_!Y} Y\simeq u^*R\times_{u^*X}Y \to Y$
is covering since $u$ is cocontinuous. 
\end{proof}

For the remainder of this section 
\begin{equation}\label{eqn:u}
u:\CC_S\subseteq\DD_S
\end{equation}
is the inclusion between full subcategories of $\Sch_S^G$, both containing $\emptyset$ and both satisfying properties (P) and (H).
   Main examples of interest to keep in mind are the inclusions $\Sm^G_S[\E]\subseteq \Sm^G_S[\E']$.

\begin{lemma}\label{lem:ucoco}
Let  $u:\CC_S\subseteq \DD_S$ be as in \eqref{eqn:u}. Then $u$ is topologically cocontinuous.
\end{lemma}
\begin{proof}
	Let $X\in \CC_S$. Since the $cd$-structure ${\rm fpNis}$ on each of these categories is complete, the covering sieves on $X$ and on $u(X)$, in both cases, are exactly those which contain the sieve generated by a simple covering, see \cite[Section 2]{V:cd}. Property (P) implies that any simple covering of  $X$ in $\DD_S$ is a simple covering in $\CC_S$. It follows that the pullback $u^*R$ of any covering sieve $R$ is again a covering sieve.
\end{proof}

\begin{proposition}\label{prop:peq}
Let   $u:\CC_S\subseteq \DD_S$ be as in \eqref{eqn:u}.  
\begin{enumerate}
\item The functor $u_!:\Pre(\CC_S)\to \Pre(\DD_S)$ preserves all Nisnevich  and all motivic equivalences.
\item The functor $u^*:\Pre(\DD_S)\to \Pre(\CC_S)$ preserves all Nisnevich  and all motivic equivalences.
\end{enumerate}
\end{proposition}
\begin{proof}        
Recall that if $\mathcal{A}$ is a presentable $\i$-category and $S$ is a set of morphisms in $\mathcal{A}$, then the class of $S$-local equivalences is the closure of $S$ under pushouts, small colimits, and $2$-out-of-$3$, see e.g., \cite[Proposition 5.5.4.15]{HTT}.

The first statement then follows from the fact that $u_!$ preserves colimits, fixed point Nisnevich squares, and that $u_!(X\times \A^1) \simeq u_!(X)\times \A^1$.
The second statement follows from \aref{lem:cocoeq}, \aref{lem:ucoco}, and that $u^*(X\times \A^1)\simeq u^*(X)\times \A^1$. 
\end{proof}

\begin{corollary}\label{prop:iphi}
 Let $u:\CC_S\subseteq \DD_S$ be as above. There are natural equivalences 
 $ u^*\circ \L_{\Nis}\wkeq \L_{\Nis}\circ u^*$ and $u^*\circ \L_{\mot} \wkeq \L_{\mot}\circ u^*$ of functors $\Pre(\DD_S)\to \Pre(\CC_S)$.
\end{corollary}
\hfill \qed

The adjoint pairs $(u_!,u^*)$ and $(u^*,u_*)$ extend to adjoint pairs on $\A^1$-invariant Nisnevich sheaves. Overloading notation, we continue to write $u_!,u^*,u_*$ for the induced functors on categories of  $\A^1$-invariant Nisnevich sheaves.

\begin{proposition}\label{cor:sev}
	Let $u:\CC_S\subseteq \DD_S$ be as in \eqref{eqn:u}.  
\begin{enumerate}
\item The restriction functor $u^*:\Pre(\DD_S)\to \Pre(\CC_S)$ preserves $\A^1$-invariant Nisnevich sheaves.
\item The induced functor 
$u^*:\HH(\DD_S)\to \HH(\CC_S)$ is symmetric monoidal and has a left adjoint $u_!$ and a right adjoint $u_*$.
\item Similarly $u^*:\HH_\bb(\DD_S)\to \HH_\bb(\CC_S)$ is symmetric monoidal and has a left and a right adjoint, which we again denote respectively by $u_!$ and $u_*$. 
\end{enumerate}

 Moreover, these functors fit into commutative diagrams

\[
\begin{tikzcd}
\Pre(\CC_S) \ar[r,shift left, "u_!"] \ar[d, shift right, "\L_{\mot}"']\ar[d, shift left, hookleftarrow] &  \Pre(\DD_S) \ar[l, shift left, "u^*"] \ar[d, shift right, "\L_{\mot}"'] \ar[d, shift left, hookleftarrow]\\
\HH(\CC_S)\ar[r, shift left, "u_!"] \ar[d, shift right, "(-)_+"'] \ar[d, shift left, leftarrow] & \HH(\DD_S) \ar[l, shift left, "u^*"]  \ar[d, shift right, "(-)_+"'] \ar[d, shift left, leftarrow]\\
\HH_\bb(\CC_S) \ar[r, shift left, "u_!"] & \ar[l, shift left, "u^*"] \HH_\bb(\DD_S)
\end{tikzcd}
\,\,\textrm{ and } \,\,\,
\begin{tikzcd}
\Pre(\DD_S) \ar[r,shift left, "u^*"] \ar[d, shift right, "\L_{\mot}"']\ar[d, shift left, hookleftarrow] &  \Pre(\CC_S) \ar[l, shift left, "u_*"] \ar[d, shift right, "\L_{\mot}"'] \ar[d, shift left, hookleftarrow]\\
\HH(\DD_S)\ar[r, shift left, "u^*"] \ar[d, shift right, "(-)_+"'] \ar[d, shift left, leftarrow] & \HH(\CC_S) \ar[l, shift left, "u_*"]  \ar[d, shift right, "(-)_+"'] \ar[d, shift left, leftarrow]\\
\HH_\bb(\DD_S) \ar[r, shift left, "u^*"] & \ar[l, shift left, "u_*"] \HH_\bb(\CC_S).
\end{tikzcd}
\]

\end{proposition}
\begin{proof}
Since $u_!$ preserves fixed point Nisnevich squares and $u_!(X\times \A^1)\simeq u_!(X)\times \A^1$,  it follows that $u^*$ preserves $\A^1$-invariant Nisnevich sheaves on $\CC_S$ and thus restricts to a limit preserving functor $u^*:\HH(\DD_S)\to \HH(\CC_S)$. As these are categories are presentable, $u^*$ has a left adjoint $u_!$. It follows from \aref{prop:peq} that $u^*:\HH(\DD_S)\to \HH(\CC_S)$ preserves colimits. 
In particular it has a right adjoint $u_*$. 
Note that $u^*$ is symmetric monoidal since it preserves limits and  $\HH(\DD_S)$ is cartesian monoidal.

The functors $u^*$ and $u_*$ preserve final objects so induce adjoint pairs on based spaces. Since $u^*$ preserves limits, it has a left adjoint $u_!$. 
It is straightforward to verify that these fit into the displayed commutative diagrams.

\end{proof}

\begin{remark}
Suppose that $\CC_S$ and $\DD_S$ are also closed under binary products. Then $u_!:\HH(\CC_S)\to\HH(\DD_S)$ preserves binary products (the monoidal product on these categories). However, $u_!$ is not in general a symmetric monoidal functor since it does not always preserve the unit object $\pt$ (because it is not in general representable) but there is always a canonical map $u_!(\pt)\to\pt$ adjoint to the equivalence $\pt\simeq u^*(\pt)$.
Similarly, $u_!:\HH_\bb(\CC_S) \to \HH_\bb(\DD_S)$ preserves the smash product, but need not preserve the unit object $S^0\simeq\pt\coprod\pt$, though again there is a canonical map $u_!(S^0)\to S^0$.
\end{remark}

\begin{proposition}\label{prop:i*ff}
Let $u:\CC_S\subseteq \DD_S$ be as above. 
 The functors
\begin{align*}
u_!, u_*:\HH(\CC_S) &  \to \HH(\DD_S) \\
u_!, u_*:\HH_\bb(\CC_S) & \to \HH_\bb(\DD_S)
\end{align*}
 are all full and faithful. In particular if $\E$ is a family of subgroups closed under conjugacy and $u:\Sm^G_S[\E]\subseteq \Sm^G_S$ is the inclusion, 
$u_!, u_*:\HH^{G,\E}(S) \to \HH^G(S)$ and
$u_!, u_*:\HH_\bb^{G,\E}(S) \to \HH_\bb^G(S)$ are full and faithful.
\end{proposition}
\begin{proof}
 First we note that the unit of the adjunction $u_!^{\Pre}:\Pre(\CC_S)\rightleftarrows  \Pre(\DD_S):u^*$  is an equivalence 
 $\eta:\id\wkeq u^*u^{\Pre}_!$ and the counit of $u^*:\Pre(\DD_S)\rightleftarrows  \Pre(\CC_S):u_*$ is an equivalence $\epsilon:u^*u_*\wkeq \id$. Indeed $u^{\Pre}_!F$ and $u_*F$ are respectively computed by the left and right Kan extensions of $F$ along $u$ and so 
$u^{\Pre}_!F(W) \wkeq \colim_{ W\to X}F(X)$ and 
$u_*F(W) \wkeq \lim_{X\to W}F(X)$. Here the indexing categories are respectively the categories of morphisms  $W\to X$ and $X\to W$ where $X\in \Sm^G[\E]$. In particular, if $W\in \Sm^G_S[\E]$ these have an initial, respectively terminal, object and so  $\id\wkeq u^*u^{\Pre}_!$ and $u^*u_*\wkeq \id$, as claimed.

 The counit of  $u^*:\HH(\DD_S)\rightleftarrows  \HH(\CC_S):u_*$ is thus an equivalence as $u^*u_*\simeq \id$. To show that
 the unit of $u_!:\HH(\CC_S)\rightleftarrows  \HH(\DD_S):u^*$ is an equivalence
  $\id\simeq u^*u_!$, we note that by
  \aref{prop:iphi} (and writing $\iota:\HH^G(\CC_S)\subseteq \Pre(\CC_S)$ for the inclusion) we have $u^*u_! \wkeq u^*\L_{\mot}u_!^{\Pre}\iota\wkeq \L_{\mot}u^*u_!^{\Pre}\iota \wkeq \L_{\mot}\iota\simeq \id$. Thus both of the functors $u_!,u_*:\HH(\CC_S)\to \HH(\DD_S)$ are full and faithful.

The pointed cases follow
 immediately from the above considerations. 

\end{proof}

Let $\iota:H\hookrightarrow G$ be a group monomorphism. Let $\CC_S^H\subseteq\Sch^H_S$ and $\CC^G\subseteq \Sch^G_S$ be full subcategories satisfying (P) and (H), as above. Suppose further that the induction-restriction adjunction 
$\iota_!:\Sch^{H}_S\rightleftarrows \Sch^{G}_S:\iota^*$ 
restricts to the adjunction
\[\iota_!:\CC^{H}_S\rightleftarrows \CC^{G}_S:\iota^*.\]

\begin{lemma}\label{lem:indres}
Let $\iota$, $\CC_S^G$, $\CC^G_S$ be as above. 
\begin{enumerate}
	\item The restriction functor $\iota^*:\Pre(\CC^G_S)\to \Pre(\CC^H_S)$ preserves $\A^1$-invariant Nisnevich sheaves.
	\item The induced functor 
	$\iota^*:\HH(\CC^G_S)\to \HH(\CC^H_S)$ is symmetric monoidal and has a left adjoint $\iota_!$ and a right adjoint $\iota_*$.
	\item Similarly $\iota^*:\HH_\bb(\CC^G_S)\to \HH_\bb(\CC^H_S)$ is symmetric monoidal and has a left and a right adjoint, which we again denote respectively by $\iota_!$ and $\iota_*$. 
\end{enumerate}

Moreover, these functors fit into commutative diagrams

\[
\begin{tikzcd}
\CC^H_S \ar[r, "\iota_!"] \ar[d] &  \CC^G_S  \ar[d] \\
\HH(\CC^H_S)\ar[r, shift left, "\iota_!"] \ar[d, shift right, "(-)_+"'] \ar[d, shift left, leftarrow] & \HH(\CC^G_S) \ar[l, shift left, "\iota^*"]  \ar[d, shift right, "(-)_+"'] \ar[d, shift left, leftarrow]\\
\HH_\bb(\CC^H_S) \ar[r, shift left, "\iota_!"] & \ar[l, shift left, "\iota^*"] \HH_\bb(\CC^G_S)
\end{tikzcd}
\,\,\textrm{ and } \,\,\,
\begin{tikzcd}
\CC^G_S \ar[r, "\iota^*"] \ar[d] &  \CC^H_S \ar[d] \\
\HH(\CC^G_S)\ar[r, shift left, "\iota^*"] \ar[d, shift right, "(-)_+"'] \ar[d, shift left, leftarrow] & \HH(\CC^H_S) \ar[l, shift left, "\iota_*"]  \ar[d, shift right, "(-)_+"'] \ar[d, shift left, leftarrow]\\
\HH_\bb(\CC^G_S) \ar[r, shift left, "\iota^*"] & \ar[l, shift left, "\iota_*"] \HH_\bb(\CC^H_S).
\end{tikzcd}
\]
\end{lemma}
\begin{proof}
The functors $i_!:\CC^H_S\to \CC^G_S$, $i^*:\CC^G_S\to \CC^H_S$ both send distinguished squares to distinguished squares and
$\iota_!(X)\times \A^1\cong \iota_!(X\times \A^1)$ and  
$\iota^*(X)\times \A^1\cong \iota^*(X\times \A^1)$.
It follows that $\iota_!:\Pre(\CC^H_S)\to \Pre(\CC^G_S)$ and $\iota^*:\Pre(\CC^G_S)\to \Pre(\CC^G_S)$ preserve all motivic equivalences. It follows that these induce functors on the category of motivic spaces as displayed above. Since $\iota^*:\HH(\CC^G_S)\to \HH(\CC^G_S)$ preserves limits and these categories are cartesian monoidal, it follows that $\iota^*$ is symmetric monoidal. 

Since $\iota^*$, $\iota_*$ preserve final objects they induce an adjoint pair on based spaces. Since $\iota^*$ preserves limits, it has a left adjoint $\iota_!$. 
It is straightforward to verify that these fit into commutative diagrams as displayed.

\end{proof}

\begin{remark}\label{rem:pw}
The adjunctions of the previous lemma admit the following alternate description in terms change of base functors. 
There is an equivalence of categories 
\begin{equation}\label{eqn:scheq}
\Sch^G_{G\times_HS} \xrightarrow{\sim}\Sch^H_S
\end{equation}
induced by taking the fiber over $\{e\}\times S\subseteq G\times_HS$. 
Write $\CC_{G\times_HS}\subseteq \Sch^G_{G\times_HS}$ for the  full subcategory corresponding to $\CC^H_S$ under \eqref{eqn:scheq}. 
The restriction functor $\iota^*$ corresponds to pullback along $f:G\times_HS\to S$. 
Moreover, the equivalence \eqref{eqn:scheq} induces an equivalence of motivic spaces
\[
\HH(\CC_{G\times_HS})\simeq \HH(\CC_S^H)
\]
and under this equivalence the functors 
$i_!, i^*, i_*$ are respectively identified with the functors $f_\#, f^*, f_*$.

\end{remark}

\subsection{Motivic  \texorpdfstring{$G$}{G}-spectra}
We recall the construction of categories of motivic $G$-spectra from \cite{Hoyois:6} and some variants.

Let $\CC^{\otimes}$ be a presentably symmetric monoidal $\infty$-category and $X$ a set of objects in $\CC$.
 If $I = \{x_1,\ldots, x_n\}$ is a finite subset of $X$, write $\bigotimes I = x_1\otimes \cdots \otimes x_n$. Write
$$
\CC[X^{-1}] := \colim_{\substack{I\subseteq X \\ I\textrm{ finite}}}\CC[(\bigotimes I)^{-1}],
$$
where $\CC[x^{-1}]$ denotes the symmetric monoidal inversion of an object $x\in \CC$ in a presentable symmetric monoidal $\infty$-category $\CC$, see \cite[Section 2.1]{Robalo}.

Alternatively, one may consider the stabilization in $\Mod_{\CC}$, the $\i$-category of $\CC$-modules in $\LPrx$. Recall that if $\MM\in \Mod_{\CC}$ and $x\in \CC$ then $\Stab_x(\MM)$ is the colimit, in $\Mod_{\CC}$, 
of the sequence $\MM\xrightarrow {-\otimes x} \MM \xrightarrow{-\otimes \MM} \MM \to\cdots$. More generally, for a set of objects $X$ in $\CC$, define
\[
\Stab_X(\MM) :=  \colim_{\substack{I\subseteq X \\  I\textrm{ finite}}}\Stab_{\otimes I}(\MM).
\]
An object $x\in \CC$ is \emph{$n$-symmetric} if the cyclic permutation on $x^{\otimes n}$ is homotopic to the identity. If each $x\in X$ is $n$-symmetric for some $n\geq 2$, then the canonical map of $\CC[X^{-1}]$-modules is an equivalence
\[
\MM\otimes_{\CC}\CC[X^{-1}]\xrightarrow{\sim}\Stab_{X}(\MM),
\]
see \cite[Corollary 2.22]{Robalo}, \cite[Section 6.1]{Hoyois:6}.

Let $\E$ be a finite rank locally free $G$-module on $S$. Write $T^\E\in \HH_\bb(S)$ for the associated motivic sphere, defined as the Thom space
\[
T^\E = \V(\E)/\V(\E)- z(S),
\]
where $z:S\to \V(\E)$ is the zero section. We will also write $\Sigma^{\E}$ for the associated endofunctor 
\[
\Sigma^\E\wkeq T^{\E}\wedge - .
\]

We will also be interested in stabilizing the categories $\HH^{G,\FF}_\bb(S)$, for a family $\FF$. Observe that
$\HH^{G,\FF}_\bb(S)$ is an $\HH^{G}_\bb(S)$-module and 
$u_!:\HH_\bb^{G,\FF}(S)\to \HH_\bb^G(S)$ is a map of $\HH^{G}_\bb(S)$-modules 
since if $X\in \Sm^G_S[\FF]$ and $Y\in \Sm^G_S$  then  $X\times Y\in \Sm^G_S[\FF]$.  
In particular, even though spheres $T^\E\in \HH^{G}_\bb(S)$ are generally not objects of $\HH^{G,\FF}_\bb(S)$, they still determine endofunctors
$\Sigma^{\E}:\HH^{G,\FF}_\bb(S)\to\HH^{G,\FF}_\bb(S)$.

Write  $\Sph_B^G:=\{T^{\E}\mid \E\in \Rep_B^G \}$ where $\Rep_B^G$ is the set 
of finite rank $G$-vector bundles over $B$. 

\begin{definition}
A subset $\TT\subseteq \Sph_B^G$ is \emph{stabilizing} if there is some $T^\E\in \TT$ such that $T^\E\wkeq T\wedge T^{\E'}$, for some locally free $G$-module $\E'$.
\end{definition}
\begin{definition}	\hspace{1mm}
\begin{enumerate}
\item Let $p:S\to B$ be a $G$-scheme over $B$ and $\TT\subseteq \Sph^G_B$ a stabilizing subset.
Write 
 $$
 \SH^G_{\TT}(S):= \HH_\bb^G(S)[(p^*\TT)^{-1}].
 $$
If $\TT = \{T^\E\}$ consists of a single sphere, we write 
$\SH^G_{T^\E}(S)$ in place of $\SH^G_\TT(S)$. 
When $\TT =\Sph^G_B$, we simply write
$$
\SH^G(S):=\SH^G_{\Sph^G_B}(S).
$$
\item Let $\FF$ be a family of subgroups closed under conjugation. Define
\[
\SH^{G,\FF}_{\TT}(S):= \HH_\bb^{G,\FF}(S)\bigotimes_{\HH_\bb^{G}(S)} \SH^G_{\TT}(S).
\]

\end{enumerate}
\end{definition}

Write 
\[
\Sigma^{\infty}_{\TT}:\HH^{G,\FF}_{\bb}(S) \to \SH^{G,\FF}_{\TT}(S)
\]
for the stabilization functor. In case $\TT = \Sph_B^G$, we simply write $\Sigma^{\infty}$. When no confusion should arise, given $X\in \HH^{G,\FF}_{\bb}(S)$  we will  write again $X$ for its image in 
$\SH^{G,\FF}_{\TT}(S)$ instead of $\Sigma^{\infty}_{\TT}X$.

\begin{proposition}\label{prop:summary}
	Let $S\in \Sch^G_B$. Then $\SH^{G,\FF}_{\TT}(S)$ is a symmetric monoidal stable $\i$-category satisfying the following properties.
	\begin{enumerate} 
		\item There is a canonical equivalence of $\HH^G_\bb(S)$-modules	
		\[
		\SH^{G,\FF}_{\TT}(S)\simeq \stab_{\TT}(\HH_{\bb}^{G,\FF}(S)).
		\]
		\item It is generated under sifted colimits by the compact objects  $\Sigma^{-k\VV}\Sigma^{\i}_{\TT}X_+$ where $k\geq 0$, $T^\VV\in\TT$, $p:X\to S$ is in $\Sm^{G}_S[\FF]$, and $X$ is affine. 
\item The family of functors 
\[
\{p^*:\SH^{G,\FF}_{\TT}(S) \to \SH^{G,\FF}_{\TT}(X)\,|\,\, p:X\to S\textrm{ is in } \Sm^{G}_S[\FF] \}
\]
 is conservative. 
	\end{enumerate}
\end{proposition}
\begin{proof}
	That $\SH^{G,\FF}_{\TT}(S)$ is stable is a consequence of the fact that $\TT$ is stabilizing and that
	$T\wkeq S^1\wedge \G_m$. It is symmetric monoidal by construction. The arguments for the remaining points of (1) and (2) are the same as \cite[Proposition 6.4]{Hoyois:6}.
	Lastly,  (3) follows from (2). 
\end{proof}

\begin{remark}\label{rem:single}
Over an affine base, every representation is the quotient of a finite sum of copies of the regular representation $\rho_G$. This implies that for any $S$,
\[
\SH^G(S) \wkeq \SH^G_{T^{\rho_G}}(S).
\] 
\end{remark}

Let $N\trianglelefteq G$ be a normal subgroup and $\pi:G\to G/N$ the quotient homomorphism. This induces a function 
$\pi^*: \Rep_B^{G/N} \to \Rep_B^{G}$, and we write		
\[
\Ntriv= \{T^{\E} \mid \E\in \pi^*( \Rep_B^{G/N}) \}\subseteq \Sph_B^G
\]
for the associated set of ``$N$-trivial $G$-spheres''. 
This stabilizing set of spheres plays an important role in later sections.

\begin{lemma}\label{lem:ideal}
Let $\FF$ be a family. The adjunction
$u_!\colon\HH_\bb^{G,\FF}(S)\rightleftarrows \HH_\bb^G(S):u^*$ of $\HH^{G}_\bb(S)$-modules induces an adjoint pair
\[
u_!:\SH^{G,\FF}_{\TT}(S)\rightleftarrows \SH_\TT^G(S):u^*.
\]
Moreover, $u^*$ is symmetric monoidal and
 $$
u_!:\SH^{G,\FF}_{\TT}(S)\hookrightarrow \SH^G_{\TT}(S)
$$
is full and faithful with  essential image  the localizing tensor ideal
generated by $\Sigma^{-nV}X_+$, where $T^V\in p^*\TT$ and $X\in \Sm^{G}_S[\FF]$.  

\end{lemma}
\begin{proof}
 
That the adjunction $(u_!,u^*)$ of $\HH^G_\bb(S)$-modules induces an adjoint pair on categories of motivic spectra follows from the 
description of $\SH^G_\TT(S)$ and $\SH^{G,\FF}(S)$ respectively as $\stab_\TT(\HH^G_\bb(S))$ and $\stab_\TT(\HH^{G,\FF}_\bb(S))$, see
 the discussion preceeding Definition 6.1 in \cite{Hoyois:6}. 
This also implies that $u_!$ is full and faithful, since $u_!\colon\HH_\bb^{G,\FF}(S)\hookrightarrow \HH_\bb^G(S)$ is,  by \aref{prop:i*ff}.

\end{proof}

Let $\CC^{\otimes}$ be a presentably symmetric monoidal $\i$-category and $E\in \CC$ is an idempotent object. Recall \cite[Definition 4.8.2.1]{HA} that this means there is a map $e:\SS \to E$ such that  $\id\otimes e:E \simeq E\otimes \SS \to E\otimes E$ and $e\otimes \id:E\simeq \SS\otimes E \to E\otimes E$ are equivalences. 
Tensoring with $E$ is a localization functor $\L=E\otimes -:\MM\to \MM$ on any $\MM\in \Mod_\CC$, see \cite[Proposition 4.8.2.4]{HA}. 
\begin{lemma}
With notation as above,
\[
\L\MM \simeq \L\CC\otimes_{\CC} \MM.
\]
\end{lemma}
\begin{proof}

An argument similar to \cite[Proposition 4.8.2.10]{HA} shows that the forgetful functor $\Mod_{E}(\MM)\to \MM$ determines an equivalence of $\CC$-modules 
$\Mod_{E}(\MM)\simeq \L\MM$. Since  $\Mod_{E}(\MM)\simeq \Mod_{E}(\CC)\otimes_{\CC} \MM$ by \cite[Theorem 4.8.4.6]{HA}, the lemma follows.
\end{proof}

 We record the following result which we'll use a few times. A similar statement can be found in \cite[Lemma 26]{Bachmann:etale}.

\begin{lemma}\label{lem:Lcomm}
	Let $\CC^{\otimes}$ be a presentably symmetric monoidal $\i$-category, 
	$X$ a set of objects. Suppose that $E\in \CC$ is an idempotent object.
	Write $\L=E\otimes-$ for the associated symmetric monoidal localization endofunctor. Then there is an equivalence in $\CAlg(\LPrx)$
	\[
	\L(\CC[X^{-1}])\simeq (\L \CC)[X^{-1}].
	\]
\end{lemma}
\begin{proof}
Both of these categories can be identified with $\L\CC\otimes_{\CC}\CC[X^{-1}]$. 
	 
\end{proof}

Next we record the motivic version of the Wirthmueller isomorphism, which is a special case of  the ambidexterity equivalence proved in \cite{Hoyois:6}. If $H\leq G$ is a subgroup and $\iota$ is the inclusion we sometimes write
\[
G_+\ltimes_{H} X:=\iota_!X.
\]

\begin{proposition}[Wirthmueller isomorphism, \cite{Hoyois:6}]\label{prop:wirthmueller}
	Let $\iota:H\hookrightarrow G$ be a group monomorphism. Let $\FF$,$\FF'$ be families of subgroups respectively of $H$ and of $G$ such that the induction-restriction adjunction
	restricts to $\iota_!:\Sm^H_S[\FF]\rightleftarrows\Sm^G_S[\FF']:\iota^{-1}$
	Then there is an induced  adjunction
\[
\iota_!:\SH^{H,\FF}(S)\rightleftarrows \SH^{G,\FF'}(S):\iota^*,
\]
such that $\iota_!(\Sigma^\i X)\simeq \Sigma^\i(\iota_!X)$ and $\iota^*(\Sigma^\i X)\simeq \Sigma^\i(\iota^{-1}X)$. Moreover $\iota^*$ admits a right adjoint $\iota_*$ and there is an equivalence 
	\[
	\iota_!\xrightarrow{\sim} i_*.
	\]
\end{proposition}
\begin{proof}
The first statements are straightforward. We explain the last statement. Consider the $G$-equivariant map $f:G\times_HS\to S$. Then $\iota_!,\iota^*,\iota_*$ are identified with $f_\#,f^*,f_*$ via the equivalence $\SH^{G,\overline{\FF}}(G\times_HS)\simeq \SH^{H,\FF}(S)$, see \aref{rem:pw}. Here $\overline{\FF}$ is the family of subgroups of $G$ generated by $\FF$.

We have that
$u^*\iota_!'v_!\simeq \iota_!$ and $u^*\iota_*'v_!\simeq \iota_*$, where we write $u:\Sm^{G,\FF}_S\subseteq \Sm^G_S$ and $v:\Sm^{H,\FF}_S\subseteq \Sm^H_S$ for the inclusions and  $\iota'_!,\iota'_*:\SH^{H}(S)\to \SH^G(S)$ are the corresponding functors for the family $\FF_\all=\FF'$. In particular,  the general case follows from the case $\FF=\FF_\all$.
But in this case, the Wirthmueller isomorphism is the  ambidexterity equivalence \cite[Theorem 1.5]{Hoyois:6} for the finite \'etale morphism $f$. 
\end{proof}

\subsection{Functoriality}
We record basic functoriality of the categories of motivic spaces and spectra, as the group $G$, family $\FF$, and base scheme $S$ vary. 

\begin{definition}
	\begin{enumerate}
		\item The category $\eSch_B$ of {\em equivariant $B$-schemes} has objects pairs $(G, S)$ consisting of a finite group $G$ (whose order is invertible in $\OO_B$)  and $S\in \Sch_B^G$.   
		A morphism $(G', S') \to (G,S)$ is a pair $(\phi, f)$ where 
		$\phi:G'\to G$ is a homomorphism of groups and $f:S'\to S$ is 
		a $\phi$-equivariant map of $B$-schemes. 
		
		\item The category $\eSch_B[\cd]$ {\em equivariant $B$-schemes and families} has objects consisting of triples $(G,\FF, S)$ where $(G,S)\in \eSch_B$ and $\FF$ is a family of subgroups of $G$. 
		A morphism $(G',\FF',S')\to (G,\FF,S)$ is a triple $(\phi,i,f)$ where $(\phi, f)$ is a morphism in $\eSch_B$ and $i:\phi^{-1}\FF\subset\FF'$ is an inclusion of posets. 
	\end{enumerate}
	
\end{definition}

The inclusion $i:\phi^{-1}\FF\subset\FF'$ is  unique if it exists. When no confusion should arise, we write $(\phi, f)$ instead of  $(\phi, i, f)$ for a morphism in $\eSch_B[\cd]$.

We identify $\eSch_B$ with the full subcategory of $\eSch_B[\cd]$ whose objects are triples of the form $(G,\FF_{\all}, S)$. The inclusion 
$\eSch_B\subseteq \eSch_B[\cd]$ is left adjoint to the forgetful functor.

In this subsection we extend the construction of the previous sections  extend to {\em functors}
\[
\HH^{\times}, \HH_{\bb}^{\wedge}, \SH^{\otimes}:(\eSch_B[\cd])^{\op}\to \mathrm{\CAlg(\LPr)}
\]
whose respective values on $(G,\FF,S)$ are  the symmetric monoidal $\i$-categories $\HH^{G,\FF}(S)$, $\HH^{G,\FF}_{\bb}(S)$, and $\SH^{G,\FF}(S)$, and   on morphisms, 
$(\phi, i, f)^*\simeq i_!\phi^*f^*$.

\begin{lemma}
	The categories $\eSch_B$ and $\eSch_B[\cd]$  
	admit finite products.
	In particular, they are cartesian symmetric monoidal. 
\end{lemma}

\begin{proof}
	It is straightforward to check that 
	$(G, S)\times (G',S') = (G\times G', S\times_B S' )$  and 
	$(G,\FF,S)\times (G',\FF',S')= (G\times G',\FF\times \FF',S\times_{B}S')$ define a cartesian product. 	

\end{proof}

\begin{corollary}
The assignment $(G,\FF,S)\mapsto\Sm_S^G[\FF]$, $(\phi,i,f)\mapsto i\phi^{-1}f^{-1}$ 
extends to a functor
\[
\eSch_B[\cd]^{\op}\to\CAlg(\Cat_\infty).
\]
\end{corollary}
\hfill $\qed$

Composing with the symmetric monoidal presheaves functor, we obtain a functor

\begin{align*}
(\eSch_B[\cd])^{\op}
& \to\CAlg(\LPr), \\
 (G,\FF, S) & \mapsto \Pre(\Sm^G_S[\FF]) \\
(\phi, i, f) & \mapsto i_!\phi^*f^*.
\end{align*}

To obtain the desired functoriality of equivariant motivic spaces and spectra, we follow the techniques of \cite[Section 6.1]{BH}. 
Recall from loc. cit. that we have a commutative diagram of  $\i$-categories
\[
\begin{tikzcd}
\mathcal{M}\Cat_\i\ar[r]\ar[d] & \mathcal{O}\Cat_\i\ar[r]\ar[d] & \mathcal{E}\ar[d]\\
\Cat_\i\ar[r, "{\Fun(\Delta^1,-)}"] & \Cat_\i\ar[r] & \mathrm{Pos},
\end{tikzcd}
\]
in which all squares are cartesian. 
Here $\mathrm{Pos}$ denotes the $\i$-category of (not necessarily small) posets, the lower right hand horizontal arrow sends an $\i$-category to the poset of subsets of the set of equivalence classes of objects,  and $\mathcal{E}\to\mathrm{Pos}$ is the universal cocartesian fibration, restricted to posets.  The $\i$-categories $\mathcal{O}\Cat_\i$ and $\mathcal{M}\Cat_\i$ are respectively the $\i$-categories of
$\i$-categories equipped with a collection of equivalence classes of objects respectively a collection of equivalence classes of arrows.

\begin{lemma}\label{lem:loc}
Let $(\CC,W)\in \mathcal{M}\Cat_\i$ such that $\CC$ is presentable and $W$ is of small generation.
\begin{enumerate}
\item The partial adjoint to 
\[
\Cat_\i\to \mathcal{M}\Cat_\i, \;\;\;\CC\mapsto (\CC,{\rm equivalences}) 
\]
 is defined at $(\CC,W)$ and the localization $\CC[W^{-1}]$ is again presentable.

\item Suppose that $\CC$ admits a symmetric monoidal structure $\CC^{\otimes}\in \CAlg(\Cati)$ 
and 
$W$ is stable under the monoidal product. Then $(\CC,W)$ lifts to $(\CC,W)^{\otimes}\in \CAlg(\mathcal{M}\Cati)$ and the partial left adjoint to \[\CAlg(\Cat_\i)\to \CAlg(\mathcal{M}\Cat_\i), \;\;\;\CC\mapsto (\CC,{\rm equivalences}) \] is defined at $(\CC,W)^{\otimes}$.
\end{enumerate}
\end{lemma}
\begin{proof}
	The first item follows from \cite[Proposition 5.5.4.15, Proposition 5.5.4.20]{HTT}. It
	follows from \cite[Proposition 2.2.1.9]{HA}
	that $\CC[W^{-1}]$ inheirits a monoidal structure such that $\CC\to \CC[W^{-1}]$ is monoidal. This implies the second item.
\end{proof}

Consider the subfunctor of the composition
\[
(\eSch_B[\cd])^{\op}\to\Cat_\infty\overset{\mathcal{P}}{\to}\LPr\to\Fun(\Delta^1,\Cat_\i)
\]
whose value on the object $(G,\FF,S)$ is the full subcategory
\[
W_{(G,\FF,S)}\subset\Fun(\Delta^1,\mathcal{P}(\Sm_S^G[\FF]))
\]
consisting of the motivic equivalences. Since motivic equivalences are stable under smash product and are preserved by the functors $f^*$, $\phi^*$, and $i_!$,
the assignments
\[
(G,\FF,S)\mapsto(\Pre(\Sm_S^G[\FF]),W_{(G,\FF,S)}),\,(\Pre_\bb(\Sm_S^G[\FF]), W_{(G,\FF,S)})
\]
induce functors
\[
(\eSch_B[\cd])^{\op}\to\mathcal{M}\Cat_\infty.
\]
The images of 
$(\mathcal{P}(\Sm_S^G[\FF]), W_{(G,\FF,S)})$ and $(\Pre_\bb(\Sm_S^G[\FF]), W_{(G,\FF,S)})$
under the partial left adjoint to $\CAlg(\Cat_\i)\to \CAlg(\mathcal{M}\Cat_\i)$ are 
respectively 
$\HH^{G,\FF}(S)$ and $\HH^{G,\FF}_\bb(S)$.

Write $\Sph^G_S$ for the set of spheres $\{T^\E\}$ where $\E$ is an equivariant vector bundle over $S$. The assignment $(G,\FF,S)\mapsto \Sph^G_S$ is a presheaf of sets on $\eSch_B[\cd]^{\op}$, which we write as $\Sph$. Let
\[
\TT:\eSch_B[\cd]^{\op}\to\set
\] 
be a subpresheaf of $\Sph$, which is closed under smash product and takes values in stabilizing sets of spheres, i.e. there is some $T^\E\in\TT_{(G,\FF,S)}$ such that $T^\E\wkeq T\wedge T^{\E'}$. 
We obtain a functor
\[
\eSch_B[\cd]^{\op}\to\CAlg(\mathcal{O}\Cat_\i)
\]
which on objects is the assignment $(G,\FF,S)\mapsto(\HH_\bullet^{G,\FF}(S),\TT_{(G,\FF,S)})$.
By the following lemma, we obtain $\SH^{G,\FF}_{\TT_{(G,\FF,S)}}(S)$ as the image of
$(\HH_\bullet^{G,\FF}(S),\TT_{(G,\FF,S)})$ under
  the partial left adjoint of $\CAlg(\Cati)\to \CAlg(\mathcal{O}\Cati)$. $\CC\mapsto (\CC,\pi_0{\rm Pic}(\CC))$. 
\begin{lemma}
Let $(\CC^\otimes,U)$ be an object of $\CAlg(\mathcal{O}\Cat_\i)$ such that $\CC$ is presentable symmetric monoidal and $U$ is small.
Then the partial adjoint of
\[
\CAlg(\Cati)\to \CAlg(\mathcal{O}\Cati), \;\;\CC\mapsto (\CC,\pi_0{\rm Pic}(\CC))
\]
is defined at $(\CC^\otimes,U)$. 
\end{lemma}
\begin{proof}
	This follows from \cite[Section 6.1]{Hoyois:6}.
\end{proof}

\section{Filtering by isotropy}\label{sec:universalF}
In this section we develop techniques to define and analyze filtrations of motivic $G$-spaces and spectra by families of isotropy.

\subsection{Universal Motivic \texorpdfstring{$\mathcal{F}$}{F}-spaces}

Let $\FF$ be a family of subgroups. 
In classical equivariant homotopy theory there is a $G$-space $\Es \FF$ characterized by the property that 
a $G$-space $X$ admits a unique map  $X\to \Es\FF$ if all of the stabilizers of $X$ are in $\FF$, and no maps from $X$ to $\Es\FF$ otherwise. The $G$-space $\Es\FF$ formally exists as a presheaf on $\Sm^G_B$ and hence as a motivic $G$-space over $B$, but it doesn't have the correct universality property. 

\begin{example}
 Let $G\neq \{e\}$, $B=\spec(k)$ a field, and $L/k$ be a finite Galois extension such that 
 $G\subseteq {\rm Gal}(L/k)$ and consider $\spec(L)$ as a smooth $G$-scheme over $k$ via the Galois action. Then $\spec(L)^{H} = \emptyset$ for all $e\neq H\subseteq G$, i.e. it has a free $G$-action. However, 
 we claim that
 $$
 \map_{\HH^G(k)}(\spec(L), \Es G) = \emptyset.
 $$
 Indeed,
 $$
(\Es G)_{0}(\spec(L))= \Hom_{\Sm_k^G}(\spec(L), \coprod_{G}\spec(k)) = \emptyset,
 $$ 
 and so the claim follows, since $ (\Es G)_{0}(\spec(L))$ surjects onto $\pi_{0}(\map_{\HH^G(k)}(\spec(L), \Es G))$ (see e.g., \cite[Corollary 2.3.22, Remark 3.2.5]{MV:A1})).
 
\end{example}

\begin{definition}
Let $\FF$ be a family of subgroups in $G$. The \emph{universal motivic 
$\FF$-space over $S$} is the object $\EE\FF_S\in \Pre(\Sm^G_S)$ whose value on $X\in\Sm^G_S$ is
$$
%
\EE\FF_S(X) = \begin{cases}
                  \pt &  X\in \Sm^G_S[\FF] \\
                           \emptyset &   \text{else}.
                 \end{cases}
$$
When the base $S$ is understood, we simply write $\EE \FF$.
\end{definition}

\begin{proposition}
Let $\FF$ be a family of subgroups in $G$. The presheaf
$\EE\FF$ is a motivic $G$-space.
\end{proposition}
\begin{proof}
We need to check that  $\EE\FF$ is Nisnevich excisive and $\A^1$-homotopy invariant. 
From the definition we have that $\Eg\FF(\emptyset)=\pt$.
Given a Nisnevich square (\ref{eqn:EDS}), the possible values of 
\begin{equation*}
\begin{tikzcd}
\Eg\FF(X) \ar[r] \ar[d] & \Eg\FF(U) \ar[d] \\
\Eg\FF(Y) \ar[r] & \Eg\FF(V)
\end{tikzcd}
\end{equation*}
are the squares
\begin{equation*}
\begin{tikzcd}
\emptyset\ar[d] \ar[r] & \emptyset\ar[d] \\
\emptyset \ar[r]& \emptyset,
\end{tikzcd}
\;\;
\begin{tikzcd}
\emptyset\ar[d] \ar[r] & \emptyset\ar[d] \\
\emptyset \ar[r]& \pt,
\end{tikzcd}
\;\;
\begin{tikzcd}
\emptyset\ar[d] \ar[r] & \emptyset\ar[d] \\
\pt \ar[r]& \pt,
\end{tikzcd}
\;\;
\begin{tikzcd}
\emptyset\ar[d] \ar[r] & \pt\ar[d] \\
\emptyset \ar[r]& \pt,
\end{tikzcd}
\;\; 
\begin{tikzcd}
\pt\ar[d] \ar[r] & \pt\ar[d] \\
\pt \ar[r]& \pt,
\end{tikzcd}
\end{equation*}
which are all cartesian squares. It follows that $\Eg\FF$ is Nisnevich excisive. Also $\EE\FF$ is $\A^1$-homotopy invariant since  
$(X\times_S\A^1_S)^{H} = X^{H}\times_{S^H} \A^1_{S^H}$.
\end{proof}

 Totaro \cite{Totaro:chow} and Morel-Voevodsky
 \cite[Section 4.2]{MV:A1} construct a geometric model for the classifying space of an algebraic group. This is generalized by Hoyois in \cite[Section 2]{Hoyois:cdh},
 to construct a geometric model for certain equivariant classifying spaces. 
Similar considerations lead to geometric models for the spaces $\EE \FF$.

\begin{definition}\label{def:EFapprox}
A \emph{system of approximations to $\EE\FF_S$} is a diagram $(U_i)_{i\in I}$ which is a subdiagram of a diagram $(V_i)_{i\in I}$ of inclusions of $G$-equivariant vector bundles over $S$, where $I$ is a filtered poset and subject to the following conditions:

\begin{enumerate}
\item Each $U_i$ is in $\Sm_S[\FF]$ and 
$U_i\subseteq V_i$ is an open subscheme. 
\item For $i\in I$, there exists an element $2i\in I$ with the property that $2i\geq i$  and such that there is an isomorphism $V_{2i} \iso V_{i}\oplus V_{i}$ of $G$-vector bundles which identifies
the inclusion 
$V_i\hookrightarrow V_{2i}$ with the inclusion $(\id,0):V_{i}\hookrightarrow  V_i\oplus V_i$. 
\item Under the isomorphism $V_{2i}\iso V_i\oplus V_i$, $(U_i\times V_i)\cup (V_i\times U_i)\subseteq U_{2i}$.
\item 
There is a Nisnevich cover $\{T_j\to S\}$ such that for any affine $X$ in $\Sm_{T_j}^G[\FF]$, there is an $i\in I$ such that  $(U_i)_{X}\to X$ admits a section. 
\end{enumerate}  
\end{definition}

\begin{example}\label{ex:bc}
Write $\rho=\rho_G$. Let  $U_{n} \subseteq \V_B({n\rho})$ be the open invariant subscheme 
$$
U_{n}:=\V(n\rho)\minus \bigcup_{H\in\co(\FF)}\V({n\rho})^{H}.
$$ 
The  inclusions 
$\V({n\rho})\subseteq \V({(n+1)\rho})$ induce maps 
$U_{n}\to U_{n+1}$.

Let $f:S\to B$ in $\Sch_B^G$. Then $(f^*U_n)_{n\in \N}$ is a system of approximations to $\EE\FF_S$. The conditions (1)-(3) are clear. To check the last condition we may assume that $B$ and $S$ are affine (since $S$ is equivariant locally affine). Let $X\in \Sm_S^G[\FF]$ be affine. Then
for $n$ sufficiently large, there is an equivariant closed immersion of $B$-schemes, 
$X\hookrightarrow \V(n\rho)$. For any $H\not\in \FF$, we have $X\cap (\V(n\rho)^{H} = \emptyset$ which means that $X\to \V(n\rho)$ factors to give a map $X\to U_n$ over $B$. This defines the desired section $X\to f^*U_n$. 
\end{example}

\begin{example}\label{ex:key}
Let $N\trianglelefteq G$ be a normal subgroup. 
The family $\FF[N]$ consists of all subgroups not containing $N$. 
Write $W:=\rho_{G}/\rho_{{G/N}}$ for the quotient representation (where $\rho_{G/N}$ is viewed as a $G$-representation via the quotient homomorphism $G\to {G/N}$). 
 Let $U_n = \V_S(nW)\minus \{0\}$. This defines a system of approximations to $\EE\FF[N]_S$. The conditions (1)-(3) of the definition are clear. As in the previous example, to check the last condition, it suffices to assume that $B$ and $S$ are affine. 
Let $X\in \Sm_S^G[\FF[N]]$ be affine. Then
for $n$ large enough, there is an equivariant closed immersion of $B$-schemes 
$X\hookrightarrow \V_B(n\rho_G)$. The preimage of $0$ under the projection $p:\V(n\rho_G)\to \V(nW)$ is $\V(n\rho_{G/N})$. Since $N$ is not contained in any stablizer of $X$, $X\cap \V_B(n\rho_{{G/N}}) = \emptyset$, which implies that restriction of $p$ to $X$ factors through $U_n$. 
This defines the desired section $X\to U_n$. 
\end{example}

If $(U_i)_{i\in I}$ is a system of approximations to $\EE\FF_S$, define 
\[
U_{\infty} := \colim_IU_i \in \HH^G(S). 
\]

\begin{proposition}\label{prop:key}
Let $\FF$ be a family of subgroups in $G$ and $(U_i)_{i\in I}$ a system of approximations to $\EE\FF_S$.  Then there is an equivalence 
$$
U_{\infty} \xrightarrow{\wkeq} \EE\FF_S
$$  
in $\HH^{G}(S)$.
\end{proposition}
\begin{proof}

To prove the result, it suffices to work  Nisnevich locally on $S$, and so we may assume that $S=T_j$ in the last condition of \aref{def:EFapprox}. 
For each $i$, there is a unique map $U_{i}\to \EE\FF$ which induces the unique map $U_{\infty}\to \EE\FF$. It suffices to show that 
$\Sing_{\A^1}(U_{\infty})(X) \to \Sing_{\A^1}(\EE\FF)(X)$ is an equivalence for any affine $X$ in $\Sm^G_S$. Both sides are empty if $X\not\in\Sm^G_S[\FF]$, so we just need to show that 
 $\Sing_{\A^1}(U_{\infty})(X)$
 is contractible for any affine $X$ in $\Sm^G_S[\FF]$,
By assumption there is a section of $U_\infty\times_S X \to X$ and so the result follows from 
\cite[Lemma 2.6]{Hoyois:cdh}. 

\end{proof}

\begin{proposition}\label{prop:bcE}
Let $\FF$ be a family of subgroups. Let $f:S'\to S$ 
be a morphism
in $\Sch_B^G$. Then

\[
 f^*(\EE\FF_{S})\wkeq \EE\FF_{S'}.
\] 
\end{proposition}

\begin{proof}
Follows from \aref{prop:key} together with \aref{ex:bc}.
\end{proof}

Write  $i:\Sm^{G}_S[\FF] \subseteq\Sm^G_S$ for the inclusion of categories. 
\begin{proposition}\label{prop:EEFFcoloc}
	There is a canonical equivalence of endofunctors 
	\[
	i_!i^*\wkeq \EE\FF\times -:\HH^G(S)\to \HH^G(S)
	\] 
	and 
	\[
	i_!i^*\wkeq \EE\FF_+\wedge -:\HH_\bb^G(S) \to \HH_\bb^G(S).
	\]  
\end{proposition}
\begin{proof}
	We treat the unbased case,  the based case then follows.
Let $X\in \HH^G(S)$. We have $(i^*X)(W) \simeq X(W)$ for $W\in \Sm^G_S[\FF]$. In particular, the projection $\EE\FF\to \pt$ induces the equivalences $i^*(\EE\FF\times X) \simeq i^*(X)$ and thus 
	$$
	i_!i^*(\EE\FF\times -) \xrightarrow{\sim} i_!i^*(-).
	$$

Now $\EE\FF$ is equivalent to  $\colim_i U_i$ where $U_i\in\Sm^{G}_S[\FF] $. If  $W\in \Sm^G_S$ then each $U_i\times W$ is in $\Sm^{G}_S[\FF] $. 
	It follows that if $X$ is any object of $\HH^G(S)$ then, writing $X$ as a colimit of objects of $\Sm_S^G$,  we see that
	$\EE\FF\times X \simeq i_!(E)$ for some $E\in \HH^{G,\FF}(S)$. In particular, since $\eta:\id\wkeq i^*i_! $ is an equivalence by \aref{prop:i*ff}, we have 
	$i_!\eta:i_!(E)\wkeq i_!i^*i_!(E)$. From the triangle identity for the unit and counit, we have $\epsilon i_!:i_!i^*(i_!(E)) \wkeq i_!(E)$ is also an equivalence. It follows that  we have equivalences
	\[
	\EE\FF\times - \wkeq i_!i^*(\EE\FF\times -) \wkeq i_!i^*.
	\]
\end{proof}

\begin{corollary}\label{cor:EEFFequiv}
	Let $\FF$ be a family of subgroups.
	\begin{enumerate}
		\item The essential images of $\EE\FF\times -$ and  $\EE\FF_+\wedge -$ are respectively the subcategories  
		$\HH^{G,\FF}(S)\subseteq \HH^G(S)$ and $\HH_\bb^{G,\FF}(S)\subseteq \HH_\bb^G(S)$. 
		\item The projection $\EE\FF\times X \to X$ 
		is an equivalence  for $X\in \HH^G(S)$ if and only if $X \simeq i_!(X')$ for some $X'\in \HH^{G,\FF}(S)$.
		\item The projection 
		$\EE\FF_+\wedge Y \to Y$ is an equivalence for $Y\in \HH_\bb^G(S)$ if and only if $Y \simeq i_!(\tilde{Y})$ for some $\tilde{Y}\in \HH_\bb^{G,\FF}(S)$.
\item The canonical maps are equivalences 
\begin{align*}
\map_{\HH^G(S)}(\EE\FF\times X, X')& \wkeq \map_{\HH^G(S)}(\EE\FF\times X,\EE\FF\times X') \\
\map_{\HH_\bb^G(S)}(\EE\FF_+\wedge Y, Y')& \wkeq \map_{\HH_\bb^G(S)}(\EE\FF_+\wedge Y,\EE\FF_+\wedge Y').
\end{align*}
\end{enumerate}
\end{corollary}
\hfill \qed

Recall that $i_!:\SH^{G,\FF}_\TT(S)\to \SH^G_\TT(S)$ is full and faithful with essential image the localizing tensor ideal generated by  $T^{-\E}\otimes X_+$ where  $X\in \Sm_S^{G}[\FF] $ and $T^\E\in\TT$, see \aref{lem:ideal}.

\begin{proposition}\label{prop:EEimage}
	
There is an equivalence of colocalization endofunctors
\[
 \EE\FF_+\otimes -\simeq i_!i^*:\SH^G_{\TT}(S)\to \SH^G_{\TT}(S).
\]
In particular there is a natural equivalence
\[
i_!i^*\simeq i_!i^*(\SS_S)\otimes \id.
\]
\end{proposition}
\begin{proof}
This follows from \aref{prop:EEFFcoloc}.

\end{proof}

\subsection{Filtrations by adjacent families}\label{sec:filtadj}

We recall the definition of adjacent families. 
\begin{definition}
Let $\FF\subseteq\FF'$ be an inclusion of families of subgroups of $G$.
	\begin{enumerate}
		\item We say that $\FF$ and $\FF'$ are \emph{adjacent} if there is a subgroup $H\leq G$ such that 
		$\FF'\minus\FF = \{ (H) \}$.
		\item If $N\trianglelefteq G$ is a normal subgroup, say tha  $\FF$ and $\FF'$ are \emph{$N$-adjacent} if there is a subgroup $H\leq N$ such that $\FF'\minus \FF = \{K\leq G \mid (K\cap N) = (H) \}$ (where as before, $(A)$ denotes $G$-conjugacy class of a subgroup $A$).
		\item Say that $\FF$ and $\FF'$ are \emph{$N$-adjacent at $H\leq N$} if the families are $N$-adjacent and  $\FF'\minus \FF$    is the set of subgroups $K\leq G$ such that  $(K\cap N) = (H)$
	\end{enumerate}
\end{definition}

Of course, if $N=G$ then $N$-adjacent families are exactly adjacent families. 
Since $G$ is finite, one can always find a filtration
\begin{equation*}
\emptyset = \FF_{-1} \subseteq \FF_{0}\subseteq \FF_{1}\subseteq \cdots \subseteq \FF_{n}=\FF_{\all},
\end{equation*}
such that each pair $\FF_i\subseteq \FF_{i+1}$ is $N$-adjacent.
 For example, one can be produced as follows. 
Define the sequence of families
\[
\{e\}=\Fil_0^G\subset \Fil_1^G\subset \cdots \subset \Fil_i^G\subset \cdots \subset \Fil_{N}^G=\FF_{{\all}}
\]
by setting $\Fil_0 = \{e\}$ and inductively defining $\Fil_i^G$ by
\[
\Fil_{i}^G := \{ H\leq G \mid \textrm{each proper subgroup $K<H$ is in $\Fil_{i-1}^G$}  \}.
\]
(Since $G$ is finite, this sequence terminates at a finite stage.)
Each $\Fil_i^G$ is a family. More generally, define
\[
\Fil_{i}^{N \trianglelefteq G}:= \{ K\leq G \mid K\cap N\in \Fil_{i}^{N}  \}.
\]
The families just defined aren't adjacent, but 
$\Fil_{i}^{N \trianglelefteq G}\minus \Fil_{i-1}^{N \trianglelefteq G}$ is a finite union of conjugacy classes. Let $\{(H_i)\}$ be the set of these conjugacy classes. Then the families
\[
\Fil_{i-1}^{N \trianglelefteq G}\subseteq \Fil_{i-1}^{N \trianglelefteq G}\cup \{(H_1)\}\subseteq \Fil_{i-1}^{N \trianglelefteq G}\cup\{(H_1), (H_2)\} \subseteq \cdots \subseteq \Fil_{i}^{N \trianglelefteq G}
\] 
are all $N$-adjacent.

In anycase, a filtration 
$\emptyset  \subseteq \FF_{0}\subseteq \FF_{1}\subseteq \cdots \subseteq \FF_{all}$
gives rise to a filtration
\begin{equation}\label{eqn:fbyf}
\ast \to \EE\FF_{0+}\wedge X\to \EE\FF_{1+}\wedge X\to \cdots \to \EE\FF_{n+}\wedge X\simeq X
\end{equation}
of an object $X\in \CC$ whenever $\CC$ is an $\HH^G_{\bb}(S)$-module (e.g., $\SH^G(S)$).

To utilize this filtration, we need to analyze the filtration quotients, which we will do below in 
\aref{prop:EENadj} and \aref{prop:Phiadj}.

\subsection{Universal spaces for pairs }

\begin{definition}
 Let $\FF\subseteq \FF'$ be a subfamily. Define the based motivic $G$-space 
$\EE(\FF', \FF)$ so that it sits in the cofiber sequence
$$
\EE\FF_{+}\to \EE\FF'_{+} \to \EE(\FF', \FF).
$$	
If $\FF'=\FF_{\all}$ define
 $\widetilde{\EE}\FF := \EE(\FF_\mathrm{all},\FF)$
\end{definition}

Note that
$\EE(\FF',\emptyset)\wkeq\EE(\FF')_+$. At the other extreme, since $\EE\FF_\mathrm{all} \simeq \pt$, the space  $\widetilde{\EE}\FF$ sits in the cofiber sequence in $\HH_\bb^G(S)$
\begin{equation}\label{eqn:cof}
\EE\FF_{+} \to S^{0} \to \widetilde{\EE }\FF.
\end{equation}

\begin{proposition}\label{lem:EErew}
	Let $\FF\subseteq \FF'$ be a subfamily. 
There is a canonical equivalence 
$\EE(\FF', \FF)\wkeq \wt\EE\FF\wedge\EE\FF'_+$. In particular, $\wt\EE\FF\wedge\EE\FF_+\wkeq \pt$.
\end{proposition}
\begin{proof}

This follows from the commutative diagram
$$
\begin{tikzcd}
\EE\FF_+ \ar[r]\ar[d, "\sim"'] & \EE\FF'_+\ar[r]\ar[d, "\sim"] & \EE(\FF',\FF)\ar[d] \\
 (\EE\FF\times\EE\FF')_+ \ar[r] & (\EE\FF_\mathrm{all}\times\EE\FF')_+ \ar[r] & \wt\EE\FF\wedge \EE\FF'_+
\end{tikzcd}
$$
induced by inclusions of families, in which the rows are each cofiber sequences and the left and middle vertical arrows are equivalences.
\end{proof}

\begin{corollary}\label{cor:wtEEloc}
	Let $X\in \HH^G_\bb(S)$. 
The map $X\to \wt\EE\FF\wedge X$ induced by $S^0\to \wt\EE\FF$, induces an equivalence
$$
\map_{\HH_\bb^G(S)}(\wt\EE\FF\wedge X, \wt\EE\FF\wedge Y) 
\wkeq \map_{\HH_\bb^G(S)}( X,\wt\EE\FF\wedge Y).
$$
\end{corollary}
\begin{proof}

This follows from the previous proposition, \aref{cor:EEFFequiv}, and the fiber sequence
$\map(\wt\EE\FF\wedge X, \wt\EE\FF\wedge Y)\to  
\map( X,\wt\EE\FF\wedge Y) \to \map( \EE\FF_+\wedge X,\wt\EE\FF\wedge Y)$. 
Alternatively, simply note that $\EE\FF_+\otimes - $ is a colocalization, $\wt\EE\FF\otimes -$ is a localization endofunctor.
\end{proof}

 \begin{lemma}\label{prop:usefulF}
 	Let $\FF\subseteq\FF'$ be an inclusion of families and $\mcal{E}$  a family such that $\mcal{E}\cap\FF = \mcal{E}\cap\FF'$. Then the map $S^{0}\to \wt{\Eg } \mcal{E}$ induces an equivalence of based motivic $G$-spaces
 	$$
 	\Eg (\FF',\FF)\xrightarrow{\sim}\Eg (\FF',\FF)\wedge \wt{\Eg } \mcal{E}.
 	$$ 
 \end{lemma}
 \begin{proof}
 	Smashing $\Eg\mcal{E}_+$ with the defining cofiber sequence for $\Eg(\FF', \FF)$ yields the cofiber sequence
 	$$
 	\Eg \FF_+\wedge\Eg \mcal{E}_+\xrightarrow{f} 
 	\Eg \FF'_{+}\wedge \Eg \mcal{E}_{+}
 	\to \Eg (\FF',\FF)\wedge\Eg \mcal{E}_+.
 	$$
 	Smashing $\Eg(\FF',\FF)$ with the defining cofiber sequence 
 	for $\wt{\Eg}\mcal{E}$ yields the cofiber sequence
 	$$
 	\Eg (\FF',\FF)\wedge\Eg \mcal{E}_+
 	\to \Eg (\FF',\FF) 
 	\xrightarrow{i}
 	\Eg (\FF',\FF)\wedge \wt{\Eg } \mcal{E}.
 	$$
 	By the hypothesis and \aref{prop:wtEEFF1FF2property}, $f$ is an equivalence, and so  $\Eg (\FF',\FF)\wedge\Eg \mcal{E}_+$ is contractible. This implies that $i$ is an  equivalence.
 \end{proof}

To continue the analysis, we pass to the $\i$-categorical stabilization (i.e., $S^1$-spectra) of the various categories of motivic $G$-spaces. We write 
\[
\spt_{S^1}^{G,\E}(S):=\stab(\HH^{G,\E}_\bb(S)).
\]
Recall that $\spt_{S^1}^{G,\E}(S)$ can  be identified with the 
category of $\A^1$-invariant Nisnevich sheaves of spectra.
We write $\stable_{\CC}(X,Y)$ for the spectrum of maps in a stable $\i$-category $\CC$.

Let $u:\CC_S\subseteq \DD_S$ be as in \eqref{eqn:u}. Then we have induced functors 
\[
\begin{tikzcd}
\spt_{S^1}(\CC_S)\ar[rr, "u_!", bend left] \ar[rr, "u_*"' , bend right]&  & \spt_{S^1}(\DD_S) \ar[ll, "u^*"'].
\end{tikzcd}
\]
Since $u^*: \Pre(\DD_S)\to \Pre(\CC_S)$ is a symmetric monoidal left adjoint, it follows that $\spt_{S^1}(\DD_S)\to \spt_{S^1}(\CC_S)$ is as well. Since $u_!,u_*: \Pre(\CC_S)\to \Pre(\DD_S)$ are full and faithful by \aref{prop:i*ff}, it follows that 
$u_!,u_*:\spt_{S^1}(\CC_S)\to \spt_{S^1}(\DD_S)$ are as well.

We introduce a minor technical condition on pairs $\FF\subseteq \FF'$ over $S$, which we sometimes require. 
\begin{condition}\label{a:pair}
Let $\FF\subseteq \FF'$. Suppose there is a normal subgroup $N\unlhd G$ such that
\begin{enumerate}
\item  $N$ as well as the elements of 
$(\FF'\cap\FF[N])\minus \FF$ act trivially on $S$, and

\item  $\FF\subseteq \FF'\cap\FF[N]$.
\end{enumerate}
\end{condition}

We will say that $\FF$ satisfies this condition if the pair $\FF\subseteq \FF_{\all}$ satisfies the condition.

\begin{remark}
The pair $\FF\subseteq \FF'$ satisfies \aref{a:pair} in the following two important cases, which cover all of the cases relevant to this paper.
\begin{enumerate}
	\item The base $S$ has trivial action (in which case we take $G=N$).
	\item The normal subgroup $N$ acts trivially on $S$ and $\FF=\FF'\cap\FF[N]$. 
\end{enumerate} 
\end{remark}

\begin{proposition}\label{prop:wtEEFF1FF2property}
Let $\FF\subseteq \FF'$ be a subfamily satisfying \aref{a:pair}, $\E=\FF'\minus\FF$, and $u:\Sm^G_S[\E]\subseteq \Sm^G_S$ the inclusion. 
Let
$f:X_1\to X_2$ be a map in $\spt_{S^1}^G(S)$. Suppose that $S$ has trivial action. Then the following are equivalent.
\begin{enumerate}
\item The map $f_*:X_1(W) \to X_2(W)$ is an equivalence, for any $W\in \Sm_S^G[\E]$.
	\item The  $u^*(f):u^*(X_1)\to u^*(X_2)$ is an equivalence in $\spt_{S^1}^{G,\E}(S)$.
	\item The map 
	\[
	\EE(\FF',\FF)\otimes X_1\to \EE(\FF',\FF)\otimes X_2
	\]
	is an equivalence in $\spt_{S^1}^G(S)$. 
\end{enumerate}

\end{proposition}
\begin{proof}

The first two items are immediately equivalent.

To see the (3) implies (2), we have 
$ u^*(\EE(\FF',\FF)\otimes X)\simeq  u^*(\EE(\FF',\FF))\otimes u^*(X)$ 
and it is straightforward that $u^*(\EE(\FF',\FF))\simeq S^{0}$ in $\spt_{S^1}^{G,\E}(S)$.

Now we show that (2) implies (3).  
First, we assume that all elements of $\E$ act trivially on $S$.
Filter the inclusion $\FF\subseteq \FF'$ by adjacent families
and consider the resulting sequence
\[
\EE\FF_+=\EE\FF_{0+}\to\EE\FF_{1+}\to \EE\FF_{2+}\to \cdots 
\EE\FF_{n-1+}\to \EE\FF_{n+}=\EE\FF'_+.
\]
An inductive argument shows that it suffices to establish that
\[
\EE(\FF_{r},\FF_{r-1})\otimes X_1 \to \EE(\FF_{r},\FF_{r-1})\otimes X_2
\]
is an equivalence for $1\leq i\leq n$. Thus, we may assume that $\FF',\FF$ are adjacent, say   $\FF'\minus \FF = \{(H)\}$
and $H$ acts trivially on $S$.

By \aref{prop:EEFFcoloc} and \aref{lem:EErew}, in order to show that the map displayed above is an equivalence, it  suffices to  show that 
\[
\stable_{\spt^G_{S^1}(S) }(W_+,\wt{\EE}\FF\otimes f)
\] 
is an equivalence for any $p:W\to S$ in $\Sm^G_S[\FF']$, with $W$ affine.
From the gluing sequence \cite[Proposition 5.2]{Hoyois:6}, we have the cofiber sequence in $\spt^G_{S^1}(S) $
\[
W(\FF)_+\to W_+ \to p_\#i_*(W^{\FF}_+)
\]
where $i:W^{\FF}\subseteq W$ is the inclusion (see \aref{not:useful}). 
By \aref{lem:EErew} we have 
\[
\stable_{\spt^G_{S^1}(S) }(-,\wt{\EE}\FF\otimes f)\simeq \stable_{\spt^G_{S^1}(S) }(\wt{\EE}\FF\otimes -, \wt{\EE}\FF\otimes f).
\]
We have $\wt{\EE}\FF\otimes W(\FF)_+\simeq \pt$, using  \aref{cor:EEFFequiv} and \aref{lem:EErew} and so we conclude that 
\[
\stable_{\spt^G_{S^1}(S) }(p_\#i_*(W^{\FF}_+),\wt{\EE}\FF\otimes f)\simeq \stable_{\spt^G_{S^1}(S) }(W_+,\wt{\EE}\FF\otimes f).
\]

The canonical map $G\times_{\NN(H)}W^{H}\to W^{\FF}$ is an isomorphism. Indeed, it suffices to check that $W^{H}\cap W^{gHg^{-1}} = \emptyset$ if $H\neq gHg^{-1}$. Now, if $w\in W^{H}\cap W^{gHg^{-1}}$, then $\stab(w)$ contains both of these subgroups. But since $\stab(w)\in \FF'$ and $(H)$ is a maximal element of this poset, we have that $H=\stab(w)=gHg^{-1}$. In particular, $W^{\FF}$ is smooth over $S$, since $W^{H}\to S$ is smooth. 
 It follows from purity \cite[Proposition 5.7]{Hoyois:6} that 
$p_\#i_*W^{\FF} \simeq \Th(\NNN_{i})$.
We show that if $V\to W^\FF$ is an equivariant vector bundle, then $\stable_{\spt^G_{S^1}(S)}(\Th(V),\wt{\EE}\FF\otimes f)$ is an equivalence. For this, we may assume that $H$ is normal and in particular that $W^\FF=W^H$. Indeed, we have $\Th(V)\simeq G_+\ltimes_{\NN H}\Th(V|_{W^H})$, so via the induction-restriction adjunction, we can replace $G$ by $\NN H$ and $\FF, \FF'$ by $\FF|_{\NN H}$, $\FF'|_{\NN H}$, if necessary.  
Since $G$ is linearly reductive, we can write $V\cong V'\oplus V^H$. Since $(V')^H=0$, all stabilizers of $V'\minus \{0\}$ are in $\FF$ and so  
$\wt\EE\FF\otimes V'\minus \{0\}_+\simeq \ast$ which implies that 
\[
 \wt\EE\FF\otimes \Th(V')\simeq \wt\EE\FF\otimes W^\FF_+. 
\]
Now it follows that
\begin{align*}
\stable_{\spt^G_{S^1}(S)}(\Th(V),\wt{\EE}\FF\otimes f)& \simeq 
\stable_{\spt^G_{S^1}(S)}(\Th(V')\otimes \Th(V^H),\wt{\EE}\FF\otimes f) \\
&\simeq \stable_{\spt^{G}_{S^1}(S)}(W^{\FF}_+\otimes \Th(V^H), \wt{\EE}\FF\otimes f)\\
&\simeq \stable_{\spt^{G,\E}_{S^1}(S)}(W^{\FF}_+\otimes \Th(V^H),u^*(\wt{\EE}\FF\otimes f))\\
&\simeq \stable_{\spt^{G,\E}_{S^1}(S)}(W^{\FF}_+\otimes \Th(V^H),u^*( f)),
\end{align*}
which is an equivalence, as needed.

Next we consider the case when $\FF=\FF'\cap \FF[N]$, where $N$ is a normal subgroup of $G$ which acts trivially on $S$. Again we consider $p:W\to S$ in $\Sm^G_S[\FF']$
 and the cofiber sequence in $\spt^G_{S^1}(S) $
\[
W(\FF[N])_+\to W_+ \to p_\#i_*(W^{N}_+)
\]
where now $i:W^{N}\subseteq W$ is the inclusion. Since $W(\FF[N])_+\in \Sm^G_S[\FF]$
we have that $\wt{\EE}\FF\otimes W(\FF[N])_+\simeq \pt$ and  we conclude that 
\[
\stable_{\spt^G_{S^1}(S) }(p_\#i_*(W^{N}_+),\wt{\EE}\FF\otimes f)\simeq \stable_{\spt^G_{S^1}(S) }(W_+,\wt{\EE}\FF\otimes f).
\]
Since $N$ acts trivially on $S$, $W^N\to S$ is again smooth and so we have $p_\#i_*W^{N} \simeq \Th(\NNN_{i})$ and a similar argument as in the previous paragraph shows that 
\[
\stable_{\spt^G_{S^1}(S)}(\Th(\NN_{i}),\wt{\EE}\FF\otimes f)
\]
 is an equivalence.

Now we consider the general case. 
 Let $N$ be the normal subgroup as in \aref{a:pair} and 
consider the inclusions $\FF\subseteq \FF''\subseteq \FF'$, where we
write $\FF'' = \FF'\cap\FF[N]$.   
From the cofiber sequence
\[
\EE(\FF'',\FF) \to \EE(\FF',\FF) \to \EE(\FF',\FF'')
\]
we see that it suffices to show that $\EE(\FF'',\FF)\otimes f$ and $\EE(\FF',\FF'')\otimes f$
are both equivalences. The case $\FF''\subseteq \FF'$ is covered by the previous paragraph and since all elements of  $\FF''\minus \FF$ act trivially on $S$, this case is covered by the first.

\end{proof}

\begin{corollary}\label{cor:u*u_*}
Let $\FF\subseteq \FF'$ be a subfamily which satisfies \aref{a:pair}. Write $\E=\FF'\minus\FF$, and $u:\Sm^G_S[\E]\subseteq \Sm^G_S$ the inclusion. 
Then the map
$$
\EE(\FF',\FF)\otimes u_!u^*X \to \EE(\FF',\FF)\otimes X
$$
is an equivalence for any $X\in \spt_{S^1}^G(S)$.
\end{corollary}
\begin{proof}
	By \aref{prop:wtEEFF1FF2property}, it suffices to show that $u^*\epsilon :u^*(u_!u^*X) \to u^*(X)$ is an equivalence. This follows from the triangle identity for the unit and counit since 
the unit $\eta:\id\wkeq u^*u_!$ is an equivalence, as $u_!$ is full and faithful.  
\end{proof}

\subsection{Localization at a cofamily}

\begin{proposition}\label{cor:j!j*}
Let $\FF$ be a family of subgroups which satisfies \aref{a:pair} 
and write $j:\Sm^G_S[\co(\FF)]\subseteq \Sm^G_S$ for the inclusion. Then the natural equivalence $j^*(\wt\EE\FF\otimes -)\simeq j^*$ induces an equivalence of localization endofunctors
 $$
  \wt\EE\FF\otimes - \wkeq j_*j^*: \spt_{S^1}^G(S)\to  \spt_{S^1}^G(S).
$$
  
\end{proposition}
\begin{proof}
 We have natural equivalences
 \begin{align*}
  \stable(W, j_*j^* X) & 
  \wkeq \stable(j^*W, j^*X) 
  \\
  & \wkeq \stable(j_!j^*W, \wt\EE\FF\otimes X)
\\
& \wkeq \stable(\wt\EE\FF\otimes j_!j^*W, \wt\EE\FF\otimes X)
\\
& \wkeq \stable(\wt\EE\FF\otimes W, \wt\EE\FF\otimes  X)
\\
& \wkeq \stable(W, \wt\EE\FF\otimes  X),
 \end{align*}
where the second follows by adjunction and the equivalence $j^*X\wkeq j^*(\wt\EE\FF\otimes X)$, the fourth from \aref{cor:u*u_*}, and the third and fifth from \aref{cor:wtEEloc}. 
\end{proof}

Write $\L_{\co(\FF)} \SH^G_{\TT}(S)$ for the essential image of 
$\wt\EE\FF\wedge -:\SH^G_{\TT}(S)\to \SH^G_{\TT}(S)$.
Since $\EE\FF_+\wedge -$ is a colocalization endofunctor, 
 $\wt\EE\FF\wedge -$ is a localization endofunctor. 
 In particular
  $\wt\EE\FF$ is an idempotent object of $\SH^G_\TT(S)$, see \cite[Proposition 4.8.2.4]{HA} and so 
$\SH^G_{\TT}(S) \to \L_{\co(\FF)}\SH^G_{\TT}(S)$ is a symmetric monoidal localization.

We will abuse notation and terminology slightly by saying that a stabilizing set of spheres
 $\{T^{\E}\in \HH^G_\bb(S)\}$ is in $\HH^{G,\co(\FF)}_\bb(S)$ if it is in the essential image of $j_!$. Suppose that $\TT\subseteq \Sph^G_B$ is a set of spheres such that $p^*\TT$ is in $\HH^{G,\co(\FF)}_\bb(S)$. In this case,  we can stabilize $\HH^{G,\co(\FF)}_\bb(S)$ with respect to $\TT$ and we define
\[
\SH^{G,\co(\FF)}_{\TT}(S):= \HH_\bb^{G,\co(\FF)}(S)[(p^*\TT)^{-1}].
\]
Equivalently, $\SH^{G,\co(\FF)}_{\TT}(S)\simeq \spt_{S^1}^{G,\co(\FF)}(S)[(p^*\TT)^{-1}]$

\begin{proposition}\label{prop:uswtEE}
Let $\FF$ be a family which satisifies \aref{a:pair} and write $j:\Sm^G_S[\co(\FF)]\subseteq \Sm^G_S$ for the inclusion. 
\begin{enumerate}
\item The functor $j_!:\spt^{G,\co(\FF)}_{S^1}(S)\to \spt^G_{S^1}(S)$ is symmetric monoidal.
\item There is an equivalence $j_*\simeq \wt\EE\FF\wedge j_!$. Moreover, $j_*$ induces
 a symmetric monoidal equivalence $\overline{j}_*:\spt^{G,\co(\FF)}_{S^1}(S)\to \L_{\co(\FF)}\spt^G_{S^1}(S)$. 
\end{enumerate}

\end{proposition}
\begin{proof}
Note that $S\in \Sm^{G}_S[\co(\FF)]$, so that $j_!$ preserves terminal objects. That it is symmetric monoidal then follows from the fact that $j$ preserves products. 
Since  $j_!$ is a fully faithful left adjoint,   the unit map $\id\to j^*j_!$ is an equivalence.
Precomposing $j_!$ with the equivalence $\wt\EE\FF\land(-)\simeq j_*j^*$ of \aref{cor:j!j*}, we obtain equivalences $\wt\EE\FF\land j_!(-)\simeq j_*j^*j_!\simeq j_*$ and consequently a commutative diagram
\[
\begin{tikzcd}
\spt_{S^1}^{G,\co(\FF)}(S)\ar[r, "j_!"]\ar[rd, "\overline{j}_*"'] & \spt_{S^1}^G(S)\ar[d,"{\wt\EE\FF\land -}"]\\
& \L_{\co(\FF)}\spt^G_{S^1}(S).
\end{tikzcd}
\]
In particular, the functor $\overline{j}_*:\spt^{G,\co(\FF)}_{S^1}(S)\to \L_{\co(\FF)}\spt^G_{S^1}(S)$ is a composite of symmetric monoidal functors, hence symmetric monoidal itself.
It is fully faithful, since the composite functor  $\spt^{G,\co(\FF)}_{S^1}(S)\to \L_{\co(\FF)}\spt^G_{S^1}(S)\subseteq\spt^G_{S^1}(S)$ (also denoted $j_*$) is fully faithful by \aref{prop:i*ff}.
To see that $j_*$ is also essential surjective, suppose given $X\in\spt_{S^1}^G(S)$, and consider $j^*X\in\spt_{S^1}^{G,\co(\FF)}(S)$.
By the previous equivalence $\wt\EE\FF\land j_!(-)\simeq j_*$ and  \aref{cor:j!j*}, we see that
\[
\wt\EE\FF\land j_!j^*X\simeq\wt\EE\FF\land X
\]
since both are equivalent to $j_*j^*X$.
Thus $\overline{j}_*:\spt^{G}_{S^1}(S)\to\spt^{G,\co(\FF)}_{S^1}(S)$ is an equivalence of symmetric monoidal $\infty$-categories.
\end{proof}

\begin{proposition}\label{prop:wtEEFFSH}
Let $\FF$ be a family which  satisifies \aref{a:pair} and write $j:\Sm^G_S[\co(\FF)]\subseteq \Sm^G_S$ for the inclusion. Suppose that  $p^*\TT$ is in  $\HH^{G,\co(\FF)}_\bb(S)$.
\begin{enumerate}
\item  $j_!$ induces a symmetric monoidal functor $ j_!:\SH^{G,\co(\FF)}_{\TT}(S)\to \SH^G_{\TT}(S)$.
\item $\wt\EE\FF\wedge j_!$ induces
 a symmetric monoidal equivalence
 $$
 \wt\EE\FF\wedge j_!:\SH_\TT^{G,\co(\FF)}(S)\to
 \L_{\co(\FF)} \SH^G_{\TT}(S).
 $$
\end{enumerate}

\end{proposition} 
\begin{proof}
Since $j_!$ is a symmetric monoidal left adjoint, it induces a symmetric monoidal functor on   $p^*\TT$-stabilizations.

By \aref{prop:uswtEE}
the induced map
\[
\SH_\TT^{G,\co(\FF)}(S)\simeq\spt_{S^1}^{G,\co(\FF)}(S)[p^*\TT^{-1}]\to (\L_{\co(\FF)}\spt_{S^1}^G(S))[p^*\TT^{-1}]
\]
is an equivalence in $\CAlg(\LPrx)$. 
The result now follows from the equivalence from \aref{lem:Lcomm}
\[
  (\L_{\co(\FF)}\spt^G_{\S^1}(S)))[p^*\TT^{-1}]\simeq \L_{\co(\FF)}(\spt^G_{\S^1}(S))[p^*\TT^{-1}]).
\]

\end{proof}

\subsection{Adjacent pairs}

\begin{lemma}\label{lem:Xconc}
	Suppose that $\FF\subseteq \FF'$ is $N$-adjacent at $H$ and and that $H$ acts trivially on $S$.  Let $X\in \Sm^G_S[\FF'\minus \FF]$.  Then the canonical map
	\[
	f:G\times_{N_GH}X^H\to X
	\]
	is an isomorphism.
\end{lemma}
\begin{proof}
	The map $f$ is identified with the map  
	$\coprod_{[g]\in G/N_G(H)} X^{gHg^{-1}}\to X$ induced by the inclusions $X^{gHg^{-1}}\subseteq X$. 
	We show that the induced map $f_s$ on the fiber over any $s\in S$,  is an isomorphism. By \cite[17.9.5]{EGA4} this implies that $f$ is an isomorphism as it is a map between
	 smooth (in particular flat) finitely presented $S$-schemes.
	The stabilizer $\stab(x)$ of any $x\in X_s$ contains a subgroup conjugate to $H$ which implies that $\coprod X_s^{gHg^{-1}}\to X_s$ is surjective. 
	On the other hand the closed subschemes $X_s^{gHg^{-1}}\subseteq X_s$ are pairwise disjoint  since all stabilizers of $X_s$ are in 
	$\{K\leq G \mid (K\cap N ) = (H)\}$. Indeed, if $\stab(x)$ contains both $H$ and $gHg^{-1}$, $\stab(x)\cap N$ contains both of these subgroups, which means that $H=gHg^{-1}$. 
	Thus $\coprod X_s^{gHg^{-1}}\to X_s$ is bijective, closed immersion. Since these are smooth over $\spec(k(s))$, in particular reduced, the map $f_s$ is an isomorphism. 
\end{proof}

Let $H\leq N$ be a subgroup. Then $\W_NH$ is a normal subgroup of $\W_GH$ and we will often write $\WW H = \W_G H/\W_N H$ for the quotient of Weyl groups. We will often write
 $\EE_{\W_NH}(\W_GH) $ instead of $\EE\FF(\W_NH)$ for the universal 
  $\W_NH$-free motivic $\W_GH$-motivic space, in order to emphasize the ambient group, where 
  $\FF(\W_NH)$ is the family of subgroups $\{K\leq \W_GH \mid K\cap \W_NH = \{e\}\}$. 

\begin{proposition}\label{prop:EENadj}
	Suppose that $\FF\subseteq \FF'$ is $N$-adjacent at $H$. Then 
	\begin{enumerate}
		\item $(G\times_{\NN_GH} \EE_{\W_NH}(\W_GH))_+\wedge \EE(\FF',\FF)
		\xrightarrow{\sim} \EE(\FF',\FF)$ is an equivalence in $\spt_{S^1}^G(S)$
		\item $\EE(\FF', \FF)|_{\NN_GH}\xrightarrow{\sim} \wt\EE\FF[H]\wedge \EE(\FF',\FF)|_{\NN_GH}$ in $\spt_{S^1}^{\NN_GH}(S)$.
	\end{enumerate}
\end{proposition}
\begin{proof}
It suffices to prove these equivalences in the case $S=B$; the general case follows by applying $f^*$ where for $f:S\to B$ is the structure map.	
For the first item, by \aref{prop:wtEEFF1FF2property},
	it suffices to show that the projection 
	\[
	p:(G\times_{\NN_GH} \EE_{\W_NH}(\W_GH))_+ \to  S^0
	\]
	induces equivalences $\map_{\HH^G(B)}(X,p)$, for any $X\in \Sm^G_B[\FF'\minus \FF]$. But for such an $X$ we have by   \aref{lem:Xconc} that $X\cong G\times_{\NN_GH}X^H$, so we have 
	\begin{align*}
	\stable_{\spt_{S^1}^G(B)}(X_+, & (G\times_{\NN_GH} \EE_{\W_NH}(\W_GH))_+)  \\ & \wkeq 
	\stable_{\spt_{S^1}^{\NN_GH}(B)}(X^{H}_+, (G\times_{\NN_GH} \EE_{\W_NH}(\W_GH))_+) \\
	& \wkeq \stable_{\spt_{S^1}^{\W_GH}(B)}(X^{H}_+, (G\times_{\NN_GH} \EE_{\W_NH}(\W_GH))^H_+)\\
	& \wkeq  \stable_{\spt_{S^1}^{\W_GH}(B)}(X^{H}_+, (S^0)^H)\\
	& \wkeq \stable_{\spt_{S^1}^{\NN_GH}(B)}(X^{H}_+, S^0) \\
	& \wkeq \stable_{\spt_{S^1}^{G}(B)}((G\times_{\NN_GH}X^{H})_+, S^0).
\end{align*}

	For the second item, we have that 
	$\FF_{\NN_GH}[H]\cap \FF'|_{\NN_GH} = \FF_{\NN_GH}[H]\cap \FF|_{\NN_GH}$ and so this follows from 
	\aref{prop:usefulF}.
\end{proof}

\section{Fixed point functors}\label{sec:fp}

We define fixed point functors on motivic spaces and spectra. Throughout this section, $N\unlhd G$ is a normal subgroup and we write $\pi:G\to G/N$ for the quotient homomorphism. Unless noted otherwise, we assume that $N$ acts trivially on $S$.

\subsection{Fixed point motivic spaces}\label{fps}

We begin  by extending the adjunction
\[
\pi^{-1}:\Sm_S^{G/N}\rightleftarrows \Sm^{G}_{S}:(-)^N
\]
to motivic $G$-spaces. 
Note that $\pi^{-1}$ is full and faithful and induces an equivalence  
\begin{equation}\label{eqn:pieq}
\Sm_S^{G/N}\wkeq \Sm^{G}_S[\co(\FF[N])]\subseteq \Sm^G_S
\end{equation}
with the subcategory of smooth $G$-schemes over $S$ on which $N$ acts trivially.

The functor $\pi_*$ in the following propositions is the $N$-fixed point functor on motivic $G$-spaces. In keeping with standard notation, we  sometimes write 
\[
(-)^N := \pi_*. 
\]

\begin{proposition}\label{prop:usfpadj}
The functor $\pi^{-1}:\Sm^{G/N}_S\to \Sm^G_S$ induces  adjoint pairs  of functors
\begin{align*}
\pi^{*}:\HH^{G/N}(S)  & \rightleftarrows \HH^G(S) : \pi_*,\textrm{ and} \\
\pi^{*}:\HH^{G/N}_\bb(S) &  \rightleftarrows \HH^G_\bb(S) : \pi_* 
\end{align*}
such that the diagrams commute

\[
\begin{tikzcd}
\Sm^{G/N}_S \ar[r, "\pi^{-1}"] \ar[d] &   \Sm^{G}_S\ar[d] \\
\HH^{G/N}(S)\ar[r, "\pi^*"]  \ar[d, "(-)_+"'] & \HH^G(S)   \ar[d, "(-)_+"] \\
\HH^{G/N}_\bb(S) \ar[r, "\pi^*"] &  \HH_\bb^{G}(S)
\end{tikzcd}
\,\,\textrm{ and } \,\,\,
\begin{tikzcd}
\Sm^G_S \ar[r, "(-)^N "] \ar[d] &  \Sm^{G/N}_S  \ar[d] \\
\HH^G(S)\ar[r,  "\pi_*"] \ar[d, "(-)_+"'] & \HH^{G/N}(S)   \ar[d,  "(-)_+"'] \\
\HH^G_\bb(S) \ar[r,  "\pi_*"] &  \HH^{G/N}_\bb(S).
\end{tikzcd}
\]
Moreover, 
the $\pi^*$ and $\pi_*$ are symmetric monoidal and $\pi_*$ preserve colimits.

\end{proposition}
\begin{proof}
Nisnevich topologies correspond under the equivalence \eqref{eqn:pieq}, so the adjunction can be obtained as a special case of \aref{cor:sev}, where we set $\pi^*:=(\pi^{-1})_!$ and $\pi_*:=(\pi^{-1})_*$. The remaining claims about the $\pi^*$ and $\pi_*$ are straightforward. 
\end{proof}

\subsection{Fixed point spectra}

We view $G/N$-equivariant vector bundles on $S$ as $G$-equivariant vector bundles via the quotient homomorphism $\pi:G\to G/N$.
Given a stabilizing subset $\TT\subseteq \Sph^{{G/N}}_{B}$,
 we have the stabilizing subset 
 $\pi^*\TT = \{T^{\pi^*\E} \mid T^\E\in \TT\} \subseteq \Sph^{G}_B$ which for  simplicity we usually write again $\TT$. We call a $G$-sphere of the form $\pi^*(T^\E)$ an \emph{$N$-trivial $G$-sphere}.  
 Recall also that we write $\Ntriv = \pi^*\big(\Sph^{G/N}_B\big)$. 

\begin{proposition}
Let $\TT\subseteq \Sph^{{G/N}}_{B}$ be a stabilizing set of spheres.
There is an adjoint pair of symmetric monoidal functors 
$$
\pi^*:\SH_{\TT}^{{G/N}}(S) \rightleftarrows \SH_{\TT}^G(S):\pi_*
$$
such that 
$\pi_*(\Sigma^{\infty}_{\TT}Y)\simeq \Sigma^{\infty}_{\TT}(Y^N)$ for $Y\in \HH_\bb^G(S)$.
\end{proposition}

\begin{proof}
	The adjoint pair is the stabilization of the adjoint pair in \aref{prop:usfpadj}, using that 
 $\pi^*(X\otimes T^\E)\simeq \pi^*(X)\otimes T^{\pi^*\E}$ and $(Y\otimes T^{\pi^* \E})^N\simeq Y^N\otimes T^\E$.
\end{proof}

That the fixed point functor in the previous proposition commutes with stabilization is a consequence of the fact that we have only stabilized with respect to a set  of $N$-trivial spheres.
In general, fixed point functors do not commute with stabilization; rather, this is a key feature of the geometric fixed points functor, defined later.

\begin{lemma}\label{lem:fppi}
Let $\TT\subseteq \Sph^{G/N}_B$. 
The equivalence \eqref{eqn:pieq} induces inverse equivalences
\[
\tilde{\pi}^*:\SH^{{G/N}}_{\TT}(S)\simeq \SH^{G,\co(\FF[N])}_{\TT}(S):\tilde{\pi}_*
\]
which fit into commutative diagrams
\[
\begin{tikzcd}
\SH^{{G/N}}_{\TT}(S)\ar[r, "\tilde{\pi}^*"]\ar[dr, "\pi^*"'] & \SH^{G,\co(\FF[N])}_{\TT}(S)\ar[d, hookrightarrow, "j_!"] \\
& \SH^{G}_\TT(S)
\end{tikzcd}
\,\,\textrm{ and } \,\,\,
\begin{tikzcd}
 \SH^{G,\co(\FF[N])}_{\TT}(S) \ar[r, "\tilde{\pi}_*"], \ar[d, hookrightarrow, "j_!"]& \SH^{{G/N}}_{\TT}(S) \\
\SH^G(S)\ar[ur, "\pi_*"'] & 
\end{tikzcd}
\]
\end{lemma}
\begin{proof}
	Nisnevich topologies correspond under the equivalence \eqref{eqn:pieq}, so that
	$\pi^*:\HH_\bb^{{G/N}}(S)\simeq \HH_\bb^{G,\co(\FF[N])}(S)$ is an equivalence of symmetric monoidal $\i$-categories with inverse $\pi_*$. The first statement follows. The commutativity of the displayed diagrams is straightforward to check.

\end{proof}

Let $\TT\subseteq \Sph^G_B$  be a stabilizing set of  $G$-spheres  and $\TT^N = \{T^{\E^N}\mid T^\E\in \TT\} $ the associated  stabilizing set of $G/N$-spheres. 
Continuing to overload notation, we simply write $\pi^*:\SH_{\TT^{N}}^{{G/N}}(S) \to \SH_{\TT}^G(S)$ again for composite
\[
\SH^{{G/N}}_{\TT^N}(S) \xrightarrow{\iota^*} \SH^{{G/N}}_{\TT}(S) \xrightarrow{\pi^*} \SH^{G}_{\TT}(S).
\]
Since $\pi^*$ preserves colimits, we obtain the fixed-points  adjunction
\[
\pi^*:\SH_{\TT^{N}}^{{G/N}}(S) \rightleftarrows \SH_{\TT}^G(S):\pi_*.
\]

The stabilizing set of spheres $\TT$ does not appear in the notation for the fixed points functor $\pi_*$. 
We usually consider  stabilization with respect to all spheres
and will always make the domain of $\pi_*$ explicit in other cases. 
When $\TT = \Sph^G_B$, in keeping with standard notation, we will  sometimes write 
\[
(-)^N:= \pi_{*}.
\]

	We have also not made reference to the base scheme in the notation. We show in the next section, after proving the Adams isomorphism, that the fixed points functor is compatible with the various change-of-base functors in motivic homotopy.

\subsection{Geometric fixed points}\label{sub:gfp}

Recall that
$\FF[N] :=\{H\leq G \mid N \not\subseteq H \}$ and that we write
\[
\L_{\co(\FF[N])} \SH^{G}_{\TT}(S)\subseteq \SH^G_{\TT}(S).
\] 
for the essential image of the endofunctor
 $\wt\EE\FF[N]\otimes -:\SH^G_{\TT}(S)\to \SH^G_{\TT}(S)$.
Write  $W = \rho_{G}/\rho_{G/N}$. 
By  \aref{ex:key} we have
\begin{equation}\label{eqn:TFN}
\wt{\EE}\FF[N] \wkeq \colim_{n} T^{n W}.
\end{equation}
\begin{lemma}\label{lem:TFN}
	The map $S^{0}\to T^{W}$ induces an equivalence in $\HH^G_\bb(S)$,
	\[
	\wt{\EE}\FF[N]\wkeq T^{W}\otimes\wt{\EE}\FF[N].
	\]
\end{lemma}
\begin{proof}
	This follows from \eqref{eqn:TFN} together with the fact that the cyclic permutation acts as the identity on  $T^W\wedge T^W\wedge T^W$, see 
 \cite[Lemma 6.3]{Hoyois:6}.
\end{proof}

\begin{proposition}\label{prop:UV}
 The stabilization functor $\SH^{G}_{\Ntriv}(S)\to \SH^{G}(S)$ induces an equivalence
$$
\L_{\co(\FF[N])} \SH^{G}_{\Ntriv}(S)\xrightarrow{\sim} \L_{\co(\FF[N])} \SH^{G}(S)
$$
of symmetric monoidal stable $\infty$-categories.
\end{proposition}
\begin{proof}
Write $\L = \L_{\co(\FF[N])}$. We make use of the equivalences 
$\SH^{G}_{ T^{\rho_{G/N}}}(S)\wkeq \SH^{G}_{\Ntriv }(S)$ and   
$\SH^{G}_{T^{\rho_G}}(S)\wkeq \SH^{G}(S)$ and the equivalence \eqref{eqn:TFN}.
Since $T^{\rho_{G}} = T^{\rho_{G/N}}\otimes T^W$, we have an equivalence 
\[
\SH^G(S) \wkeq \SH^G_{\Ntriv}(S)[(T^W)^{-1}]
\]
 and under this equivalence
$\L \SH^{G}(S)\wkeq 
(\L \SH^{G}_{\Ntriv}(S))[(T^W)^{-1}]$ by \aref{lem:Lcomm}.
But it follows from \aref{lem:TFN} that
$\Sigma^W$ is an autoequivalence of $\L\SH^G_{\Ntriv}(S)$, and therefore the stabilization induces an equivalence.
\[
\L \SH^{G}_{\Ntriv}(S) \wkeq (\L\SH^G_{\Ntriv}(S))[(T^W)^{-1}].
\]
The result follows since $(\L\SH^G_{\Ntriv}(S))[(T^W)^{-1}] \wkeq \L(\SH^G_{\Ntriv}(S)[(T^W)^{-1}])$.
\end{proof}

Write $\psi$ for the composite
\[
\SH^{G/N}(S)\xrightarrow{\pi^*} \SH^G_{\Ntriv}(S)\to \SH^G(S)\to \L_{\co(\FF[N])}\SH^G(S).
\]

\begin{proposition}
The functor $\psi: \SH^{{G/N}}(S) \to\L_{\co(\FF[N])}\SH^G(S)$
 is an equivalence of symmetric monoidal stable $\infty$-categories.
\end{proposition}
\begin{proof}
 
 By \aref{prop:UV} this composite is  equivalent to 
 \[
 \SH^{G/N}(S)\xrightarrow{\pi^*} \SH^G_{\Ntriv}(S) \xrightarrow{\wt\EE\FF[N]\otimes -} \L_{\co(\FF[N])} \SH^G_{\Ntriv}(S),
 \]
 which is an equivalence by \aref{lem:fppi} and \aref{prop:wtEEFFSH}.

 \end{proof}

\begin{definition}
Let $N\trianglelefteq G$ be a normal subgroup. The \emph{motivic geometric fixed points functor} 
\[
(-)^{\mgf N}:\SH^G(S) \to \SH^{{G/N}}(S)
\]
is the composite $\SH^G(S)\to \L_{\co(\FF[N])}\SH^G(S)\xrightarrow{\psi^{-1}} \SH^{G/N}(S)$
\end{definition}

\begin{proposition}\label{prop:gfp}
The functor $(-)^{\mgf N}:\SH^G(S) \to \SH^{{G/N}}(S)$ satisfies the following properties.

\begin{enumerate}
\item It is a symmetric monoidal left adjoint. Moreover, its right adjoint  $\wt\EE\FF[N]\otimes \pi^*$ is full and faithful. 
\item  There is a natural equivalence $( \Sigma^{\infty}X)^{\mgf N}\simeq \Sigma^{\infty}(X^{N})$ for $X\in \HH^G_\bb(S)$. 
\item $(-)^{\mgf N}\wkeq 
(\wt\EE\FF[N]\otimes -)^{N}$.
\item There is a natural transformation $(-)^N \to (-)^{\mgf N}$.
\item There is a natural equivalence $\wt\EE\FF[N]\otimes Y\simeq \wt\EE\FF[N]\otimes Y^{\mgf N}$ for $Y\in \SH^G(S)$. 

\end{enumerate}
\end{proposition}
\begin{proof}
Easy consequence of the results above.
\end{proof}

In \aref{sub:base} and \aref{sub:appl} below, we show that fixed points commutes with arbitrary base-change. We note here the easier fact that geometric fixed points commutes with arbitrary base-change. 

\begin{proposition}\label{prop:bcgfp}
Let $p:T\to S$ be a map in $\Sch^{G/N}_B$. There is a natural equivalence
$p^*(Y^{\mgf N})\simeq (p^*Y)^{\mgf N}$.  
\end{proposition}
\begin{proof}
Since $p^*\wt\EE\FF[N]_S\simeq \wt\EE\FF[N]_T$, we have a commutative diagram
\[
\begin{tikzcd}
\SH^{G/N}(S)\ar[r, "\pi^*"]\ar[d, "p^*"] &  \SH^G(S) \ar[r]\ar[d, "p^*"] & \L_{\co(\FF[N])}\SH^G(S)\ar[d, "p^*"] \\
\SH^{G/N}(T)\ar[r, "\pi^*"] &  \SH^G(T) \ar[r] & \L_{\co(\FF[N])}\SH^G(T),
\end{tikzcd}
\]
so that $p^*\psi\simeq \psi p^*$, which implies the result. 
\end{proof}

Recall that for a subgroup $H\leq N$ we write $\WW H = \W_G H/\W_N H$ for the quotient of Weyl groups. 
We also write $\EE_{\W_NH}(\W_GH) $ for the universal $\W_NH$-free $\W_GH$-motivic space (also denoted $\EE\FF(\W_NH)$). See \aref{sec:filtadj} for a recollection of $N$-adjacency.

\begin{proposition}\label{prop:Phiadj}
	Suppose that $\FF\subseteq \FF'$ is $N$-adjacent at the subgroup $H\leq N$. Then, for $X\in \SH^{G}(S)$, there is a natural equivalence
	\[
	(\EE(\FF',\FF)\otimes X)^{N} \simeq G/N_+\ltimes_{\WW H} \left(\EE_{\W_NH}(\W_GH)_+\otimes X^{\mgf H}\right)^{\W_NH} 
	\]
	
\end{proposition}
\begin{proof}
Consider the commutative diagram of group homomorphisms
\[
\begin{tikzcd}
& \NN_GH \ar[r, hookrightarrow, "\lambda"]\ar[ld, twoheadrightarrow, "\pi'"' ]\ar[d, twoheadrightarrow,] & G\ar[d, twoheadrightarrow, "\pi"] \\
\W_GH\ar[r, twoheadrightarrow,  "\pi''"] & \ar[r, hookrightarrow, "\overline{\lambda}"]\WW & G/N.
\end{tikzcd}
\]	
There is a natural equivalence
$\pi_*\lambda_! \simeq \overline{\lambda}_{!}\pi''_*\pi'_*$ since both are seen to be right adjoint to the restriction along $\NN_GH \to G/N$.	
	By  \aref{prop:EENadj}, the canonical maps  
\[
	\begin{tikzcd}
\lambda_! \left(\EE_{\W_NH}(\W_GH)_+\otimes \lambda^*(\EE(\FF',\FF)\otimes X)\right)
\ar[r,"\sim"]\ar[d,"\sim"]& \EE(\FF',\FF)\otimes X \\
\lambda_!\left(\EE_{\W_NH}(\W_GH)_+\otimes \wt\EE\FF[H]\otimes \lambda^*(\EE(\FF',\FF) \otimes X)\right)	&  \\
\end{tikzcd}
\]
are equivalences. 
Applying $\pi_*$, the equivalence displayed above and that $(-)^{\mgf}$ is symmetric monoidal,  yields the result.

\end{proof}

\subsection{Homotopy fixed points and the Tate construction}\label{sub:Tate}

\begin{definition}
	Let $X\in \SH^{G}(S)$ and $N\leq G$ a normal subgroup.
	\begin{enumerate}
		\item The \emph{motivic homotopy fixed point spectrum} of $X$ is 
		\[
		X^{\hh N}:= \pi_*i_*i^*(X) \simeq F(\EE\FF(N)_+, X)^{N}
	\]
	\item the {\emph motivic Tate spectrum} of $X$ is 
	\[
	X^{\tte N}:=\wt\EE\FF(N)\wedge F(\EE\FF(N)_+, X)^{N}.
	\]
	\end{enumerate}
\end{definition}

\begin{corollary}[Motivic Tate diagram]\label{prop:motivicTate}
	Suppose that $p$ is invertible on $B$. Let 
	$X\in \SH^{C_{p}}(B)$. There is a natural  
	pushout square in $\SH(B)$
	\[
	\begin{tikzcd}
	X^{C_{p}} \ar[r] \ar[d] & X^{\mgf C_{p}} \ar[d] \\
	X^{\hh C_{p}} \ar[r] & X^{\tte C_{p}}. 
	\end{tikzcd}
	\]	
\end{corollary}
\begin{proof}
	We have a commutative diagram with exact rows
	\[
	\begin{tikzcd}
	(\EE C_{p\,+}\wedge X)^{C_{p}}\ar[r]\ar[d,"\sim"] & X^{C_{p}} \ar[r]\ar[d] & (\wt\EE C_{p} \wedge X)^{C_{p}} \ar[d] \\
	(\EE C_{p\,+}\wedge F(\EE C_{p\,+}, X))^{C_{p}}\ar[r]  & F(\EE C_{p\,+}, X)^{C_{p}} \ar[r] & (\wt\EE C_{p} \wedge F(\EE C_{p\,+}, X))^{C_{p}}. 
	\end{tikzcd}
	\]
\end{proof}

\section{Quotient spectra}\label{sec:quotient}
As in the previous section, we let $N\unlhd G$ be a normal subgroup and $\pi:G\to G/N$ the quotient homomorphism.  It turns out (as in classical equivariant homotopy) that the functor
$\pi^*:\SH^{G/N}(S)\to \SH^G(S)$ does not have have a left adjoint, except in the trivial case that $G=\{e\}$. It does however have a partial left adjoint, constructed in this section, defined on the full subcategory of $N$-free $G$-spectra.

\subsection{Stabilization of free objects}
We write $\SH_{\TT}^{G,\Nfree}(S) := \SH_{\TT}^{G,\FF(N)}(S)$. In this section, we show that the stabilization 
\[
\lambda^*:\SH_{\Ntriv}^{G,\Nfree}(S)\to \SH^{G,\Nfree}(S)
\]
is an equivalence of stable $\i$-categories. We are grateful to Tom Bachmann for suggestions which streamlined the argument given here for this equivalence. 

Recall that  the quotient functor induces an equivalence of categories $(-)/N:\Sm^G_S\simeq \Sm^{G/N}_{S/N}$ if $N$ acts freely on $S$. Under this equivalence, Nisnevich topologies correspond, so we obtain an equivalence of symmetric monoidal $\i$-categories
\[
\HH^{G}_\bb(S) \xrightarrow{\sim} \HH^{G/N}_\bb(S/N).
\]
Lastly, stabilizing with respect to $\Sph^{G/N}_{S/N}$, we obtain an equivalence of symmetric monoidal $\i$-categories
\[
\pi_!:\SH^G_{\Ntriv}(S)\xrightarrow{\sim} \SH^{G/N}(S/N).
\]

\begin{lemma}\label{lem:easycase}
If $N$ acts freely on $S$, then 
$\lambda^*:\SH_{\Ntriv}^{G,\Nfree}(S)\to \SH^{G,\Nfree}(S)$ is an equivalence.
\end{lemma}
\begin{proof}
It suffices to see that if $E\to S$ is an equivariant vector bundle on $S$ then $T^{E}$ is invertible in $\SH_{\Ntriv}^{G,\Nfree}(S)\simeq \SH_{\Ntriv}^{G}(S)$. Under the equivalence
$\pi_!:\SH^{G}_{\Ntriv}(S)\simeq \SH^{G/N}(S/N)$ of symmetric monoidal $\i$-categories, we have that $\pi_!(T^E)\simeq T^{E/N}$, where $E/N\to S/N$ is the quotient, which is a $G/N$-equivariant vector bundle. In particular, $T^{E/N}$ is invertible, so the result follows.
\end{proof}

\begin{theorem}\label{cor:stabequiv}
	Let $N\trianglelefteq G$ be a normal subgroup.
	The stabilization functor
	\[
	\lambda^*:\SH_{\Ntriv}^{G,\Nfree}(S)\to\SH^{G,\Nfree}(S)
	\]
	is an equivalence of stable $\infty$-categories.
\end{theorem}
\begin{proof}

Write $\lambda_*$ for the right adjoint of $\lambda^*$. We first show that 
$\id \to \lambda_*\lambda^*$ is an equivalence, i.e., $\lambda^*$ is fully faithful. By \aref{prop:summary} it suffices to check that $p^*\to p^*\lambda_*\lambda^*$ is an equivalence for any $p:X\to S$ in $\Sm^{G,\Nfree}_S$ (indeed the family of such $p^*$ is conservative). 
But 
 $\lambda^*$ commutes with $p^*$, and since $p$ is smooth, so does $\lambda_*$.  Therefore
 this transformation is identified with 
$p^*\to \lambda_*\lambda^*p^*$ and  since $X$ is $N$-free, $\lambda_*\lambda^*p^*\simeq p^*$ by \aref{lem:easycase} as required.

To see that $\lambda^*$ is essentially surjective, by \aref{prop:summary}, it suffices to show that
$T^{-\VV}\otimes X_+$ is in the image of $\lambda^*$, where  $T^\VV\in  \Sph_B^G$ and $p:X\to S$ is in $\Sm^{G, \Nfree}_S$. But $T^{-\VV}\otimes X_+\simeq p_\#p^*(T^{-\VV})$ and by \aref{lem:easycase}, we have $p^*(T^{-\VV})$ is in the essential image of $\lambda^*$ and since $\lambda^*$ and $p_\#$ commute, we are done.   
\end{proof}

\begin{remark}
It follows that when $S$ has $N$-free action, there is an induced equivalence of symmetric monoidal $\i$-categories
\[
\pi_!:\SH^G(S)\simeq \SH^{G/N}(S/N).
\]	
\end{remark}

Let $p:X\to S$ be a map. The exceptional pushforward on $N$-free $G$-spectra  $p_!:\SH^{G,\Nfree}(X)\to \SH^{G,\Nfree}(S)$ can be defined as the composite
\[
i^*p_!i'_!:\SH^{G,\Nfree}(X)\hookrightarrow \SH^G(X)\to \SH^G(S) \to \SH^{G,\Nfree}(S).
\]
If $p$ is smooth, then $p_!\simeq p_\#\Sigma^{-\Omega_f}$, where $\Omega_f$ is the sheaf of differentials of $X$ over $S$. 

\begin{corollary}\label{cor:Nfreegen}
	The stable $\i$-category $\SH^{G,\Nfree}(S)$
	is generated under colimits by any of the following sets.
\begin{enumerate}
	\item $\Sigma^{-k\rho_{G/N}}p_{\#}\SS_X$ where $k\geq 0$ and $p:X\to S$ is in $\Sm^{G,\Nfree}_S$ with $X$  affine.
	\item $\Sigma^{-k\rho_{G/N}}p_{!}\SS_X$
	 where $k\geq 0$ and $p:X\to S$ is in $\Sm^{G,\Nfree}_S$ with $X$  affine.
	\item $\Sigma^{-k\rho_{G/N}}i^*q_*\SS_X$ where $k\geq 0$ and $q:X\to S$ is in  $\Sch^{G}_S$ with $q$  projective.
	\item $\Sigma^{-k\rho_{G/N}}q_*\SS_X$ where $k\geq 0$  and $q:X\to S$ is in  $\Sch^{G,\Nfree}_S$.
\end{enumerate} 
\end{corollary}
\begin{proof}
The objects in (1) are generators of $\SH^{G,\Nfree}_{\Ntriv}(S)$ by \aref{prop:summary}.

If $X$ is affine then there is a surjection
$p^*(n\rho_{G})\to \Omega_p$ for some $n> 0$.   Let $\E$ be the kernel of this map. Then in $\SH^{G,\Nfree}(X)$ we have an equivalence $T^{\Omega_p}\wedge T^{\E}\simeq T^{n\rho_{G/N}}$. Let $r:\V(\E)\to X$ be the projection. We then have 
\[
\Sigma^{n\rho_{G/N}}(p\circ r)_!(\SS_{\V(\E)})\simeq p_{\#}\SS_{X}.
\]
Therefore the generators in (1) are contained in the category generated under colimits by the objects in (2).

Next, we check that the set in (2) is contained in the  the category  generated under  colimits by the objects
$\Sigma^{-k\rho_{G/N}}i^*q_*\SS_X$ of (3).

Let $p:Y\to S$ be in $\Sm^{G,\Nfree}_S$. Since $p$ is $G$-quasi-projective,
there is an equivariant compactification
\[
\begin{tikzcd}
Y \ar[r, hookrightarrow, "u" ]\ar[rd, "p"']& \overline{Y}\ar[d, "f"] \ar[r, hookleftarrow, "t"]& Z \ar[dl, "g" ] \\
& S & 
\end{tikzcd}
\]
where $f$ is an equivariant projective morphism, $u$ is an invariant open, and $t$ is an invariant closed complement. From the gluing sequence, using the fact that $f,g$ are proper,    we obtain an exact sequence in $\SH^{G}(S)$ of the form
\[
\Sigma^{-k\rho_{G/N}}p_!\SS_{Y}\to \Sigma^{-k\rho_{G/N}}f_{*}\SS_{\overline{Y}}\to \Sigma^{-k\rho_{G/N}}g_{*}\SS_Z.
\]
Applying $i^*$ finishes the third case.

Write $\EE\FF(N)\simeq \colim_n U_n$ where $U_n\in \Sm_S^{G,\Nfree}$. Now,  case (4) follows by noting that if $q:X\to S$ is projective, then $i_!i^*q_*(\SS_X)$ is the colimit of $q_{n*}(\SS_{X\times U_n})$ where $q_n:X\times U_n\to S$.  
 
\end{proof}

\subsection{Quotient functor}

\begin{proposition}\label{prop:hqe}
	Let  $N\trianglelefteq G$ be a normal subgroup which acts trivially  on $S$.
\begin{enumerate}
\item 	There are colimit preserving functors
	\begin{align*}
	(-)/N:\HH^{G,\Nfree}(S) &\to \HH^{G/N}(S/N) \textrm{  and}\\
	(-)/N:\HH^{G,\Nfree}_{\bb}(S) &\to \HH^{G/N}_{\bb}(S/N)
	\end{align*}
	such that the diagram commutes
	\[
	\begin{tikzcd}
		\Sm^{G,\Nfree}_S \ar[r, "(-)/N"]\ar[d] & \Sm^{G/N}_{S/N} \ar[d] \\	
		\HH^{G,\Nfree}(S) \ar[r, "(-)/N"]\ar[d, "(-)_+"' ] & \HH^{G/N}(S/N)\ar[d, "(-)_+"]\\
		\HH^{G,\Nfree}_{\bb}(S) \ar[r, "(-)/N"] & \HH^{G/N}_{\bb}(S/N).
	\end{tikzcd}
	\]
\item The functor $(-)/N$ satisfies a projection formula:
$(X\times \pi^*Y)/N\simeq (X/N)\times Y$
 for $X\in \HH^{G,\Nfree}(S)$ and $Y\in \HH^{G/N}(S/N)$. Similarly, if $A,B$ are based then $(A\wedge \pi^*B)/N\simeq (A/N)\wedge B$.

	\end{enumerate}
\end{proposition}
\begin{proof}
	Write $q:\Sm^G_S \to \Sm^{G/N}_{S/N}$ for the quotient functor, $q(W) =W/N$.
	Since  $q$ sends  Nisnevich squares in $\Sm^{G,\Nfree}_S$ to Nisnevich squares in $\Sm^{G/N}_S$ and
 \[
 q(W\times \A^1)\cong q(W)\times \A^1,
 \]
 the functor $q^*:\Pre(\Sm^{G/N}_{S/N})\to\Pre(\Sm^G_S)$ defined by precomposition restricts to a functor 
	$q^*:\HH^{G/N}(S/N)\to \HH^{G,\Nfree}(S)$. Since $q_*$ preserves limits, it admits a left adjoint
	\[
	(-)/N:\HH^{G,\Nfree}(S) \to \HH^{G/N}(S/N)
	\]
	 with the stated properties. Similarly, since $q_*$ preserves the terminal object, it  induces a limit preserving functor  $q_*:\HH^{G/N}_{\bb}(S/N)\to \HH^{G,\Nfree}_{\bb}(S)$ on based spaces and therefore admits a left adjoint $(-)/N:\HH^{G,\Nfree}_{\bb}(S) \to \HH^{G/N}_{\bb}(S/N) $
	The last statement follows from \aref{prop:qe}.
\end{proof}

If $X\in \Sm^G_S$ doesn't have free action, the scheme $X/G$ need not be smooth. It is still possible  to define a quotient functor on motivic $G$-spaces. However, this functor does not stabilize to give a quotient functor on all of $\SH^G(S)$. 

\begin{proposition}\label{prop:sqe}
	Let  $N\trianglelefteq G$ be a normal subgroup 
	which acts  trivially  on $S$.

	\begin{enumerate}
	\item There is a
	colimit preserving
	functor 
	\[
	\pi_!:\SH^{G,\Nfree}(S) \to\SH^{G/N}(S/N)
	\] 
	such that the diagram commutes
	\[
	\begin{tikzcd}
		\Sm^{G,\Nfree}_S \ar[r, "\pi_!"]\ar[d] & \Sm^{G/N}_{S/N} \ar[d] \\	
		\SH^{G,\Nfree}(S) \ar[r, "\pi_!"] & \SH^{G/N}(S/N).	
	\end{tikzcd}
	\]
	The right adjoint of $\pi_!$ is the composite $i^*\pi^*$ where $i:\Sm^{G,\Nfree}_S\subseteq \Sm^G_S$ is the inclusion. 
\item The functor $\pi_!$ satisfies a projection formula. That is, if $X\in \SH^{G,\Nfree}(S)$ and $Y\in \SH^{G/N}(S/N)$ then $\pi_!(X\otimes \pi^*Y)\simeq \pi_!(X)\otimes Y$.
\end{enumerate}
\end{proposition}
\begin{proof}
If $V\to S$ is $G/N$-equivariant vector bundle, we have $(\Sigma^VX)/N\simeq \Sigma^VX/N$ by \aref{prop:hqe}. 
	It follows that $(-)/N$ extends to a colimit preserving functor 
	\[
	\pi_!:\SH_{\Ntriv}^{G,\Nfree}(S) \to \SH^{G/N}(S).
	\]
The first statement then follows from \aref{cor:stabequiv}. The second statement follows from \aref{prop:hqe}.
\end{proof}

\section{The motivic Adams isomorphism}\label{sec:adams}

In classical homotopy, the Adams isomorphism 
 identifies the $N$-fixed points of an $N$-free $G$-spectrum $X$ with the quotient of $X$ by the $N$-action. 
   This equivalence was established by Adams for $N=G$ in \cite[Theorem 5.3]{Adams:equivariant} and generalized in
\cite[Theorem II.7.1]{LMS}. A recent modern take, in terms of orthogonal spectra, on this appears in \cite{ReichVarisco}. In this section, we establish a version for $N$-free motivic $G$-spectra.

\subsection{The Adams transformation}
\label{sub:atrans}

For the remainder of this section we suppose that $S$ has trivial $N$-action and we let $\pi:G\to {G/N}$ denote the quotient homomorphism.
The Adams isomorphism is a comparison of the two functors 
 \[
 \pi_!, \pi_*i_!:\SH^{G,\Nfree}(S) \to \SH^{G/N}(S).
\]

We first construct a comparison transformation $\tau:\pi_{!}\to \pi_{*}i_!$.
Consider the cartesian square of surjective homomorphisms
\begin{equation}\label{eqn:cartd}
\begin{tikzcd}
	G\times_{G/N}G \ar[r, "\pr_1"]\ar[d, "\pr_2"'] & G \ar[d, "\pi"] \\
	G \ar[r, "\pi"] & G/N.
\end{tikzcd}
\end{equation}

Write $G'=G\times_{G/N}G$ and $N'=\ker(\pr_2)$. Observe that if $X\in\Sm^{G,\Nfree}_S$ then $\pr^*_1X$ is in $\Sm^{G',N'\textrm{-free}}_S$.
We have a commutative diagram
\[
\begin{tikzcd}
	\SH^{G,\Nfree}(S)\ar[r, "i_!"]\ar[d, "j^*\pr_1^*i_!"'] & \SH^{G}(S) \ar[d, "\pr_1^*"]\\
	\SH^{G',N'\textrm{-free}}(S)\ar[r, "j_!"] & \SH^{G'}(S),
\end{tikzcd}
\]
of colimit preserving functors, where $j:\Sm^{G',N'\textrm{-free}}_S\subseteq \Sm^{G'}_S$ is the inclusion.
\begin{proposition}\label{prop:prpieq}
With notation as above,
	the diagram commutes
\[
\begin{tikzcd}
\SH^{G,\Nfree}(S)\ar[d,"\pi_!"]\ar[r, "j^*\pr_1^*i_!"] & \SH^{G',N'\textrm{-free}}(S)\ar[d,"\pr_{2!}"]\\
\SH^{G/N}(S)\ar[r, "\pi^*"] & \SH^G(S) .
\end{tikzcd}
\]
\end{proposition}
\begin{proof}
	We have a transformation, defined as the composite
	\[
	{\pr_{2!}j^*\pr_1^*i_!} \to
	\pr_{2!}j^*\pr_1^*i_!i^*\pi^*\pi_!\to
	\pr_{2!}j^*\pr_2^*\pi^*\pi_! \to
	\pi^*\pi_!,
	\]
	where the first arrow is induced by the unit of the adjunction $(\pi_!, i^*\pi^*)$ and the last by the counit of $(\pr_{2!}, j^*\pr_2^*)$.
	To check this is an equivalence, by \aref{cor:Nfreegen} it suffices to check this is an equivalence on $\Sigma^{-\rho_{G/N}}X_+$ where $X\in \Sm^{G,\Nfree}_S$. But this is clear.
\end{proof}

\begin{remark}\label{rem:conv}
	It is sometimes convenient to write again $\pr_1^*$ for $j^*\pr_1^*i_!$ so that 
	the equivalence
$\pi^*\pi_!\simeq \pr_{2!}j^*\pr_1^*i_!$ can be expressed 
 compactly as
	\[
\pi^*\pi_!\simeq \pr_{2!}\pr^*_1.
\]
 Since $\pr_1^*$ takes $N$-free spectra to $\ker(\pr_2)$-free spectra, this is only a minor overloading of notation and no confusion should arise.  
\end{remark}

We obtain a transformation $\hat{\tau}:\pi^*\pi_! \to i_!$ via
\[
\pi^*\pi_! \simeq \pr_{2!}j^*\pr^*_1 i_!\to \pr_{2!}j^*\Delta_!\Delta^*\pr^*_1i_!\simeq i_!,
\]
where $\Delta:G\hookrightarrow G\times_{G/N}G$ is the diagonal and we use that $\Delta_!\simeq \Delta_*$ by the Wirthmueller isomorphism \aref{prop:wirthmueller}.
\begin{definition}
	The {\it Adams transformation}
	\[
	\tau:\pi_! \to \pi_*i_!
	\]
	is the  transformation induced by adjunction from 
	$\hat{\tau}:\pi^*\pi_! \to i_!$ constructed above.
\end{definition}

In a certain sense, the Adams transformation is ``smashing",
as made precise in the proposition below.
First note that there is a canonical transformation
\begin{equation}\label{eqn:pformfree}
\pi_*i_!i^*(\SS_S)\otimes \id\to \pi_*i_!i^*\pi^*
\end{equation}
between endofunctors of $\SH^{G/N}(S)$
obtained as the adjoint of 
\[
\pi^*\pi_*i_!i^*(\SS_S)\otimes \pi^*\to i_!i^*(\SS_S)\otimes \pi^*\pi_!\simeq i_!i^*\pi^*.
\]
Since $i_!i^*\simeq i_!i^*(\SS_S)\otimes\id \simeq \SS_{\EE\FF(N)}\otimes\id$ by \aref{prop:EEimage}, the transformation \eqref{eqn:pformfree} can equivalently be written as 
$\pi_*(\SS_{\EE\FF(N)})\otimes\id \to \pi_*(\SS_{\EE\FF(N)}\otimes\pi^*)$.

\begin{proposition}\label{prop:untwist}
Let $X\in\SH^{G,\Nfree}(S)$.	
The diagram
	\[
	\begin{tikzcd}
	\pi_!X \ar[r, "\tau"] \ar[d]& \pi_*i_!X \ar[r] & \pi_*i_!i^*\pi^*\pi_!X \\
	\pi_!i^*(\SS_S)\otimes \pi_!X\ar[rr, "\tau\otimes \id"'] & & \pi_*i_!i^*(\SS_S)\otimes \pi_!X
	\ar[u ] 
	\end{tikzcd}
	\]
	commutes, where the left vertical map is obtained from the op-lax monoidality of $\pi_!$, the right one is \eqref{eqn:pformfree}, and the top right horizontal arrow comes from the unit of the adjunction $(\pi_!, i^*\pi^*)$.
\end{proposition}
\begin{proof}
	Consider the following diagram
	\[
	\begin{tikzcd}
	\pi^*\pi_!X \ar[r] \ar[ddd, bend right=80, "\hat{\tau}"']& \pi^*\pi_!i^*(\SS_S)\otimes \pi^*\pi_!X \ar[dddd, bend left=85, "\hat{\tau}_{i^*\SS_S}\otimes \id"] \\
	\pr_{2!}j^*\pr_1^*i_!X\ar[u, "\sim"'] \ar[r]\ar[d] & \pr_2!j^*\pr_1^*i_!i^*(\SS_S)\otimes \pr_{2!}j^*\pr_1^*i_!X\ar[u, "\sim"]  \ar[d] \\
	\pr_{2!}j^*\Delta_!\Delta^*\pr^*_1i_!X \ar[r] \ar[d, "\sim"']& 	\pr_{2!}j^*\Delta_!\Delta^*\pr^*_1i_!i^*(\SS_S)\otimes 	\pr_{2!}j^*\pr^*_1i_!X \ar[d, "\simeq"]\\
	i_!X \ar[r]\ar[d] & i_!i^*(\SS_S) \otimes \pr_{2!}j^*\pr^*_1i_!X \ar[d, "\sim" ]\\
	i_!i^*\pi^*\pi_!X\ar[r,"\sim"]&  i_!i^*(\SS_S) \otimes \pi^*\pi_!X.
	\end{tikzcd}
	\]
	
	The functors $\pi^*, \pr_1^*, j^*$ are symmetric monoidal, $i_!$ is non-unital symmetric monoidal, and $i^*(\SS_S)=\SS_{\EE\FF(N)}$ is the unit of $\SH^{G,\Nfree}(S)$. It is straightforward to check that the top square commutes. Using the natural equivalence $\Delta_!\Delta^*\simeq \Delta_!\Delta^*(\SS_S)\otimes \id$, it is straightforward to check that the remaining squares commute. This implies the result by adjointness.   
\end{proof}

\subsection{Changing the base}\label{sub:base}
Our next goal is to verify that the Adams transformation $\tau$ is compatible with the various change of base functors in motivic homotopy. 
First, however, we recall some basic facts about manipulating natural transformations and adjunctions used below.

Let  
\[
\begin{tikzcd}
\AA \ar[r,"f^*"]\ar[d, "g^*"'] & \BB \ar[d, "k^*"]  
\ar[dl, shorten <= 5pt, shorten >= 5pt, Rightarrow, "\phi"'] \\
\CC \ar[r, "h^*"']& \DD
\end{tikzcd}
\] 
be a diagram of $\infty$-categories, where $\phi:k^*f^*\to h^*g^*$ is a natural transformation.  Suppose that $f^*$, $h^*$ admit respective left adjoints $f_!$ and $h_!$. 
The \emph{left mate} of $\phi$ is a natural transformation 

\[
\begin{tikzcd}
\AA \ar[d, "g^*"'] & \BB \ar[d, "k^*"] \ar[l, "f_!"']  \\
\CC & \DD \ar[l, "h_!"] \ar[ul, shorten <= 5pt, shorten >= 5pt, Rightarrow, "\phi_L"' ].
\end{tikzcd}
\]
 Explicitly, 
 $\phi_L:h_!k^*\to g^*f_!$ is defined to be the composite
\[
h_!k^* \to h_!k^*f^*f_! \xrightarrow{\phi} h_!h^*g^*f_! \to g^*f_!.
\]

Similarly, if $g^*,k^*$ admit respective right adjoints $g_*$ and $k_*$ then its
\emph{right mate} is a transformation 
\[
\begin{tikzcd}
\AA \ar[r, "f^*"]  \ar[dr, shorten <= 5pt, shorten >= 5pt, Rightarrow, "\phi_R" ] & \BB \\
\CC\ar[u,"g_*"]\ar[r, "h^*"']  & \DD \ar[u, "k_*"'].
\end{tikzcd}
\]
Explicitly, $\phi_R:f^*g_*\to k_*h^*$ is defined to be the composite
\[
f^*g_* \to k_*k^*f^*g_* \xrightarrow{\phi} k_*h^*g^*g_*\to k_*h^*.
\]

If $\psi$ and $\phi$ are natural transformations, then $\psi\simeq \phi_L$ if and only if $\phi\simeq \psi_R$.

\begin{lemma}\label{lem:mateseq}
	Suppose we are given a diagram of $\infty$-categories. 
	\[
	\begin{tikzcd}
	\AA \ar[r,"f^*"]\ar[d, "g^*"']  & \BB \ar[d, "k^*"] 
	\ar[dl, shorten <= 5pt, shorten >= 5pt,  Rightarrow, "\phi"']\\
	\CC \ar[r, "h^*"']& \DD.
	\end{tikzcd}
	\]
	where $f^*,h^*$ have left adjoints and $g^*,k^*$ have right adjoints. Then $\phi_L$ is an equivalence if and only if $\phi_R$ is an equivalence.
\end{lemma}
\begin{proof}
	This is a straightforward check. 
\end{proof}

\begin{remark}
	It often happens that $\phi$ is  invertible. Care should be taken to not confuse the mates of $\phi$ with those of $\phi^{-1}$ (assuming all requisite adjoints exist). For example, it is not the case the mates of $\phi$ are equivalences exactly when the mates of $\phi^{-1}$ are equivalences.
\end{remark}

We will need to know that units and counits of adjunctions are compatible across equivalences induced by mates.

\begin{lemma}\label{lem:LR}
	Let $L:\CC\rightleftarrows\DD:R$ and $L':\CC'\rightleftarrows\DD':R'$ be two adjoint pairs and $F:\CC\to \CC'$ and $G:\DD\to \DD'$ functors. Let $\phi: FR \to R'G$ and  $\psi:L'F\to GL$ be mates. 
	Write $\eta$, $\epsilon$ for the unit and counit of $(L,R)$ and $\eta',\epsilon'$ for the unit and counit of $(L',R')$.
	Then for $X\in \DD$, $Y\in \CC$, the diagrams 
	\[
	\begin{tikzcd}
	L'R'GX \ar[rd, "\epsilon'  G"'] & L'FRX \ar[r, "\psi R" ]
	\ar[l,  "L'\phi"'] & GLRX \ar[dl, "G \epsilon"] \\
	& GX & 
	\end{tikzcd} 
	\]
	and
	\[
	\begin{tikzcd} 
	& FY\ar[dr, "\eta' F"] \ar[dl, "F \eta"' ] &  \\
	FRLY \ar[r, "\phi L"] & R'GLY 
	& R'L'FY \ar[l,  "R' \psi"'] . 
	\end{tikzcd} 
	\]
	commute.
\end{lemma}
\begin{proof}
	The first claim follows from the commutativity of the diagram
	\[
	\begin{tikzcd}
	L'FRX \ar[r, " L'F\eta R"]\ar[dr, " \id"'] & L'FRLRX \ar[r, " L'\phi LR"] \ar[d, "L'FR\epsilon "] & L'R'GLRX \ar[d, "L'R'G\epsilon "] \ar[r, " \epsilon'GLR"] & GLRX \ar[d, "G\epsilon"] \\
	& L'FRX \ar[r,"L'\phi" ] & L'R'GX \ar[r, "\epsilon' G" ] & GX
	\end{tikzcd}
	\]
	since the top composite is $\psi R$.  The second claim follows from the commutativity of the diagram
	\[
	\begin{tikzcd}
	R'L'FY \ar[r," R'LF\eta"] & R'L'FRLY \ar[r, " R'L'\phi L"] &  R'L'R'GLY \ar[r, " \epsilon GL"] & R'GLY \\
	FY \ar[u, "\eta' F "] \ar[r, "F\eta "]& FRLY\ar[u, " \eta'FRL"] \ar[r, " \phi L"]  & R'GLY,  \ar[ur, "\id"'] \ar[u, "\eta'R'GL"]& 	
	\end{tikzcd}
	\] 
	where the top composite is $R'\psi$.
\end{proof}

Let $p:T\to S$ be a map in $\Sch^{G}_{B}$, on which $N$ acts trivially. Write $\pi:G\to G/N$  for the quotient. We will use $i$ both to denote the inclusion $\Sm_S^{G,\Nfree}\subseteq \Sm_S^G$ as well as the inclusion $\Sm_T^{G,\Nfree}\subseteq \Sm_T^G$.
Fix equivalences 
\[
\alpha:\pi^*p^*\xrightarrow{\sim} p^*\pi^*
\] 
and 
\[
\gamma:   i_!p^* \simeq p^*i_!. 
\]

The right mate of $\alpha$ is a transformation 
$\alpha_R:p^*\pi_*\to \pi_*p^*$.
We write $\nu$ for the right mate of $\alpha_R$,
\[
\nu=(\alpha_R)_R:\pi_*p_*\xrightarrow{\sim} p_*\pi_*,
\]
which is an equivalence $\alpha\simeq (\alpha_R)_L$ by \aref{lem:mateseq}, since $\alpha\simeq (\alpha_R)_L$ is an equivalence.
Write $\nu' =\nu (\gamma^{-1})_R$ which is an equivalence
\[
\nu':\pi_*i_!p_*\xrightarrow{\sim} p_*\pi_*i_!.
\]

We have an equivalence 
\[
\alpha^{-1}\gamma_R : p^*i^*\pi^*\simeq i^*\pi^*p^*. 
\]
Write $\beta= (\alpha^{-1}\gamma_R)_L $ for the left mate of $\alpha^{-1}\gamma_R$. Then $\beta$ is an equivalence
\[
\beta:\pi_!p^*\simeq  p^*\pi_!
\]  
and write $\phi= (\beta^{-1})_R$, which is a transformation
\[
\phi:\pi_!p_*\to p_*\pi_!.
\]

If $p$ is smooth, then $\alpha$ has a left mate
$\alpha_{L}:p_\#\pi^*\xrightarrow{\sim} \pi^*p_\#$ which is an equivalence.  It follows that $\alpha_R$ is an equivalence  by \aref{lem:mateseq}. In \aref{cor:properbc} we see that $\alpha_R$ is more generally an equivalence even when $p$ is not smooth.
Write  
\[
\overline{\alpha}:p_\#\pi_*\to \pi_*p_\#
\]
for the left mate of $\alpha_R^{-1}:\pi_*p^*\simeq p^*\pi_*$ and 
\[
\overline{\alpha\gamma}:p_\#\pi_*i_!\to \pi_*i_!p_\#
\]
for the left mate of $\alpha_R^{-1}\gamma$.

For the next lemmas it is convenient to fix some further exchange transformations. Let $\kappa_i:\pr_i^*p^*\simeq p^*\pr_i^*$ and $\lambda:j^*p^*\simeq p^*j^*$. Set $\nu=((\lambda\kappa_2)^{-1})_L:\pr_{2!}p^*\simeq p^*\pr_{2!}$.

We use the following basic consequence of the fact that $\SH^{G,\FF}(S)$ is the value of a 
functor  $\eSch_B[\cd]^{\op}$. Let $f:T\to S$ be a scheme map and $\phi:G\to K$ a homomorphism. The exchange $\phi^*f^*\simeq f^*\phi^*$ expresses the fact that $(\id, f)(\phi,\id)$ and 
$(\phi,\id)(\id, f)$ are both equal to  $(\phi,f)$. 
In particular, these exchanges can be chosen compatibly. That is, if $\phi:G\to K$ and $\psi:K\to H$ are homomorphisms, then the diagram
\[
\begin{tikzcd}
\phi^*\psi^*f^*\ar[r, "\sim"]\ar[d, "\sim"'] & \phi^*f^*\psi^* \ar[r, "\sim"]& f^*\phi^*\psi^* \ar[d, "\sim"] \\
(\psi\phi)^*f^* \ar[rr, "\sim"] && f^*(\psi\phi)^*
\end{tikzcd}
\]
commutes, and similarly for a composite of scheme maps.

\begin{lemma}\label{lem:taucom}
		Let $X\in \SH^{G,\Nfree}(S)$. The following diagram commutes:
	\[
	\begin{tikzcd}
	\pi^*\pi_!p^*X \ar[r, " \alpha\beta", "\sim"']\ar[d, "\hat{\tau}p^* "' ] & p^*\pi^*\pi_!X \ar[d, " p^*\hat{\tau}"]\\
	i_!p^*X \ar[r, " \sim", "\gamma"']  & p^*i_!X . 
	\end{tikzcd}
	\]
\end{lemma}
\begin{proof}
We have to see that each rectangle of the diagram below commutes
\[
\begin{tikzcd}
\pi^*\pi_!p^*X \ar[r, "\beta"] & \pi^*p^*\pi_!X \ar[r, "\alpha"] & p^*\pi^*\pi_!X \\
\pr_{2!}j^*\pr_1^*i_!p^*X \ar[r, "\kappa_1\gamma"]\ar[u,"\sim"]\ar[d] & \pr_{2!}j^*p^*\pr_1^*i_!X \ar[r, "\nu\lambda"] & p^*\pr_{2!}j^*\pr_1^*i_!X \ar[u, "\sim"']\ar[d] \\
\pr_{2!}j^*\Delta_!\Delta^*\pr_1^*i_!p^*X \ar[r]\ar[d, "\simeq"] & \pr_{2!}j^*\Delta_!\Delta^*p^*\pr_1^*i_!X \ar[r] & p^*\pr_{2!}j^*\Delta_!\Delta^*p^*\pr_1^*i_!X \ar[d, "\simeq"]\\
i_!p^* \ar[rr, "\gamma"] & & p^*i_!.
\end{tikzcd}
\]

To see the top diagram commutes,  consider the following diagram
	\[
	\begin{tikzcd}
	{\pr_{2!}j^*\pr_1^*i_!p^*X}\ar[r] \ar[d]\ar[dr]  &[-9pt] {\pr_{2!}j^*\pr_1^*p^*i_!X} \ar[dr] \ar[r] &  {p^*\pr_{2!}j^*\pr_1^*i_!X} \ar[dd, bend left=80] \\
	\pr_{2!}j^*\pr_1^*i_!i^*\pi^*\pi_!p^*X \ar[d]\ar[r]	& \pr_{2!}j^*\pr_1^*i_!p^*i^*\pi^*\pi_!X \ar[r] & \pr_{2!}j^*\pr_1^*p^*i_!i^*\pi^*\pi_!X \ar[d] \\
	\pr_{2!}j^*\pr_2^*\pi^*\pi_!p^*X \ar[d]\ar[r]	& \pr_{2!}j^*\pr_2^*p^*\pi^*\pi_!X \ar[r] & p^*\pr_{2!}j^*\pr_2^*\pi^*\pi_!X \ar[d] \\
	\pi^*\pi_!p^*X\ar[rr]	& & p^*\pi^*\pi_!X.
	\end{tikzcd}
	\]
	The  outer composites of this diagram yield the diagram of the lemma and a straightforward inspection suffices to see that most of the pieces of this diagram commute. The remaining  pieces involve either moving $p^*$ across the unit for the adjunction $(\pi_!, i^*\pi^*)$ or across the counit for the adunction $(\pr_{2!}, j^*\pr_2^*)$. In either case, that this results in a commutative diagram follows from \aref{lem:LR}, since the pairs of exchange equivalences $\pi_!p^*\simeq p^*\pi_!$, $p^*i^*\pi^*\simeq i^*\pi^*p^*$ and $\pr_{2!}p^*\simeq p^*\pr_{2!}$, $p^*j^*\pr_2^*\simeq j^*\pr_2^*p^*$ are mates. 

The argument for commutativity of the remaining squares is similar. 
\end{proof}

\begin{proposition}\label{prop:puptau}
	Let $X\in \SH^{G,\Nfree}(S)$. The diagram
	\[
	\begin{tikzcd}
	\pi_!p^*X \ar[r, "\beta"]\ar[d, "\tau p^*"'] & p^*\pi_!X \ar[d, "p^*\tau"] \\
	\pi_*i_!p^*X  & p^*\pi_*i_!X \ar[l, "\gamma^{-1}\alpha_R"]. 
	\end{tikzcd}
	\]
	commutes.
\end{proposition}
\begin{proof}
	To see this, consider the diagram
	\[
	\begin{tikzcd}[column sep = large]
	\pi^*\pi_!p^*X \ar[r, "\pi^*\beta"]\ar[d, "\pi^*\tau p^*"'] & \pi^*p^*\pi_!X \ar[d, "\pi^*p^*\tau"]\ar[r, "\alpha", "\sim"'] & p^*\pi^*\pi_!X \ar[d, "p^*\pi^*\tau"] \\
	\pi^*\pi_*i_!p^*X  \ar[d] & \pi^*p^*\pi_*i_!X 
\ar[r, "\alpha", "\sim"']\ar[l, "\pi^*\gamma^{-1}\alpha_R"'] & p^*\pi^*\pi_*i_!X. \ar[d] \\
	i_!p^*X & & p^*i_!X. \ar[ll, "\sim"'] 
	\end{tikzcd}
	\]
	By adjointness, it suffices to see that the top left square commutes. The right square commutes and the lower rectangle commutes by \aref{lem:LR}, so it suffices to see the outer diagram commutes. This follows from \aref{lem:taucom}.
\end{proof}

\begin{proposition}
 Let $Y\in \SH^{G,\Nfree}(T)$. 	The following diagram commutes:
	\[
	\begin{tikzcd}
	p_\#\pi_!Y \ar[d, "p_\#\tau"']\ar[r, "\beta_L"] &  \pi_!p_\#Y \ar[d, "\tau p_\#"]   \\
	p_\#\pi_*i_!Y \ar[r, "\overline{\gamma\alpha}"]  & \pi_*i_!p_\#Y.
	\end{tikzcd}
	\]

\end{proposition}
\begin{proof}

	The diagram of the lemma is the composite of the squares
	\[
	\begin{tikzcd}
	p_\#\pi_!Y  \ar[r]\ar[d, "p_\#\tau"'] & p_\#\pi_!p^*p_\#Y  \ar[d,"p_\#\tau p^*p_\#"]\ar[r, "\beta"] & p_\#p^*\pi_!p_\#Y \ar[r] \ar[d, " p_\#p^*\tau p_\#"] & \pi_!p_\#Y \ar[d, "\tau p_\# " ]\\
	p_\#\pi_*i_!Y \ar[r]  & p_\#\pi_*i_!p^*p_\#Y \ar[r, "(\gamma\alpha_R)^{-1}"]  &  p_\#p^*\pi_*i_!p_\#Y \ar[r] & \pi_*i_!p_\#Y.
	\end{tikzcd}
	\]
	The first and the third square commute by functoriality. The second square commutes by \aref{prop:puptau}.

\end{proof}

\begin{proposition}\label{lem:comtaup*}
	Let $Y\in \SH^{G,\Nfree}(T)$. 	The following diagram commutes:
	\[
	\begin{tikzcd}
	\pi_!p_*Y \ar[d, "\tau p_*" ]\ar[r, "\phi", "\sim"' ] &  p_*\pi_!Y \ar[d, "p_*\tau" ]   \\
	\pi_*i_!p_*Y \ar[r, "\nu'", "\sim"' ]  & p_*\pi_*i_!Y.
	\end{tikzcd}
	\]
\end{proposition}
\begin{proof}
	Consider the following diagram. By adjointness, we need to see that the combined top rectangles commute. 
\[
\begin{tikzcd}
\pi_!p^*p_*Y\ar[r, "\beta", "\sim"'] \ar[rdd, bend right = 10, "\tau p^*p_*"' ] & p^*\pi_!p_*Y \ar[r, "p^*\phi"] \ar[d, "p^*\tau p_*"]& p^*p_*\pi_!Y\ar[r, "\epsilon \pi_!"] \ar[d, "p^*p_*\tau"] & \pi_!Y  \ar[d, "\tau"] \\
&p^*\pi_*i_!p_*Y \ar[d, "\gamma^{-1}\alpha_R"]\ar[r, "\nu'", "\sim"'] & p^*p_*\pi_*i_!Y \ar[r, "\epsilon\pi_*i_!" ] & \pi_*i_!Y \\
& \pi_*i_!p^*p_*Y .\ar[urr, bend right=10, "\pi_*i_!\epsilon"'] & & 
\end{tikzcd}
\]
Here $\epsilon$ is the counit of the adjunction $(p^*,p_*)$. The left triangle commutes by \aref{prop:puptau}, the bottom triangle commutes by \aref{lem:LR}, and the top composite is equivalent to $\pi_!\epsilon$, also by \aref{lem:LR}. It follows that the outer diagram commutes, and since $\beta$ is an equivalence the combined rectangles commute as desired. 
\end{proof}

\begin{lemma}\label{lem:f*i!}
	Let $f:T\to S$ be a map in $\Sch_B^G$. 
	The diagram
	\[
	\begin{tikzcd}
	\SH^{G,\Nfree}(T)\ar[d, "f_*"' ] \ar[r, "i_!"] &  \SH^{G}(T) \ar[d, "f_*"] \\
	\SH^{G,\Nfree}(S)\ar[r, "i_!"] &   \SH^{G}(S)
	\end{tikzcd}
	\]
	commutes.
\end{lemma}
\begin{proof}
	The functor $f_*:\SH^{G,\Nfree}(T)\to \SH^{G,\Nfree}(S)$ may be computed as the composite 
\[
\SH^{G,\Nfree}(T) \xrightarrow{i_!} \SH^G(T) \xrightarrow{f_*} \SH^G(S)\xrightarrow{i^*}\SH^{G,\Nfree}(S).
\]
We have $i_!i^*\simeq i_!i^*(\SS)\otimes -$ and $i_!i^*(\SS)\simeq \colim_n U_n$, where $U_n$ is dualizable over $S$, by \aref{cor:EEdual}. Therefore 
\begin{align*}
i_!i^*f_*i_!\simeq \colim_n F(D(U_n), f_*i_!) & \simeq f_*F(f^*D(U_n),i_!) \\
&\simeq \colim_n f_*F(Df^*U_n, i_!) \\
& \simeq f_* \colim_n f^*U_n\otimes i_! \\
& \simeq f_*(i_!i^*(\SS)\otimes i_!) \\
& \simeq f_*i_!,
\end{align*}
where we use that $f^*i_!i^*(\SS)\simeq i_!i^*(\SS)$ by \aref{prop:bcE} since $i_!i^*(\SS)\simeq \EE \FF(N)$. 
\end{proof}

\subsection{ \texorpdfstring{$\tau$}{tau} is an equivalence}

In this subsection, we show that the Adams transformation $\tau$ is an equivalence. By \aref{cor:Nfreegen}, it suffices to show that $\tau$ is an equivalence on 
$\Sigma^{-k\rho_{G/N}}q_\#\SS_X$  where $q:X\to S$ is in $\Sm^{G,\Nfree}_S$, which in turn would follow by showing that $\tau$ is an equivalence on all $q_\#\SS_X$. Unfortunately, we do not know how to show this directly. Instead our strategy is to 
 first show that it is an equivalence on those $q_\#\SS_X$ which are {\em dualizable}. This immediately implies that $\tau$ is an equivalence on the subcategory of $\SH^{G,\Nfree}(S)$ generated under colimits by ($N$-trivial desuspensions of) such $q_\#\SS_X$. Of course,  this is only a proper subcategory of $\SH^{G,\Nfree}(S)$, unless $S$ is the spectrum of a field of characteristic zero. However, since  $\SS_{\EE\FF(N)}$ has dualizable skeleta by \aref{cor:EEdual}, we can at least conclude that 
\[
\tau:\pi_!\SS_{\EE\FF(N)}\to \pi_*\SS_{\EE\FF(N)}
\]
 is an equivalence. 
This now allows us to 
define an inverse $\tau^{-1}$ using \aref{prop:untwist} and conclude that $\tau$ is an equivalence in general.

\begin{lemma}\label{lem:freedual}
	Let $\E\in \SH^{G,\Nfree}(S)$ and suppose that $i_!\E$ is dualizable. Then $\E$ is also dualizable and $i_!F_{\SH^{G,\Nfree}(S)}(\E, \DD)\simeq F_{\SH^{G}(S)}(i_!\E, i_!\DD)$.

 In particular,  the dual of $\E$, regarded as an object of $\SH^{G,\Nfree}(S)$, agrees via $i_!$, with the dual of $\E$, regarded as an object of $\SH^G(S)$.
	\end{lemma} 
	
\begin{proof}
	This first claim follows formally from the fact that $\E\simeq i^*i_!\E$ and that $i^*$ is a symmetric monoidal functor.
The second claim follows by applying $i_!$ to the  equivalence
	\[
	i^*F_{\SH^G(S)}(i_!\E,i_!\DD)\simeq F_{\SH^{G,\Nfree}(S)}(\E,\DD)
	\]
and that $i_!i^*\simeq i_!i^*(\SS_S)\otimes \id$, see \aref{prop:EEimage}, together with the commutative square
	\[
	\begin{tikzcd}
		i_!i^*\SS_S\otimes D_S(i_!\E)\otimes i_!\DD
		\ar[r, "\sim"]\ar[d, "\sim"'] & i_!i^*\SS_S\otimes F(i_!\E,i_!\DD)\ar[d]\\
		D_S(i_!\E)\otimes i_!\DD\ar[r, "\sim"] & F(i_!\E,i_!\DD)
		\end{tikzcd}
	\]
where the vertical maps are induced by the projection $i_!i^*(\SS_S)\to \SS_S$ and the left hand vertical map is an equivalence since $i_!i^*\SS_S\otimes i_!\DD\to i_!\DD$ is an equivalence.

The last statement follows since $ F_{\SH^{G}(S)}(i_!\E, \SS_S)\simeq F_{\SH^{G}(S)}(i_!\E, i_!(i^*\SS_S))$.
\end{proof}

\begin{lemma}\label{lem:dcomm}
	Let $p:Y\to S$ be a smooth equivariant map. For any 
	$\E \in \SH^{G}(Y)$, the diagram
	\[
	\begin{tikzcd}
		D_S(p_\#\E)\otimes_{S} p_\#\E \ar[r, "\ev"]  &  \SS_{S} \\
		p_\#(p^*D_S(p_{\#}\E)\otimes_{Y} \E)   \ar[ur, "\ev'"', bend right=20]  \ar[u,"\sim"]	&
	\end{tikzcd}
	\]
	commutes, where the vertical equivalence is  the projection formula and the diagonal map is the composite	
	\begin{align*}
	p_\#(p^*D_S(p_\#\E)\otimes_{Y} \E) & \to p_\#(D_Y(p^*p_\#\E)\otimes_{Y} \E)  \to p_\#(D_Y(\E)\otimes_{Y} \E)\\
	&   \xrightarrow{p_\#\ev} p_\#\SS_Y\simeq p_\#p^*\SS_S\to \SS_S,
	\end{align*}
	of the canonical map followed by maps induced by the unit and counit of the adjunction $(p_\#,p^*)$, respectively. 
\end{lemma}
\begin{proof}
	The composite of the vertical and horizontal map is adjoint to  the composite map in the diagram
	\[
	\begin{tikzcd}
		p^* D_S(p_\#\E)\otimes_Y\E\ar[r]\ar[dr] & p^* (D_S(p_\#\E)\otimes_Y p_\#\E) \ar[r]  & p^*\SS_S\simeq\SS_Y,\\
		& p^*D_Y(p_\#\E)\otimes_Y p^*p_\#\E\ar[r]\ar[u, "\sim"]& D_Y(p^*p_\#\E)\otimes_Y p^*p_\#\E\ar[u,"\ev"], 
		\end{tikzcd}
	\]
where the left diagonal is induced by the unit of the adjunction $(p_\#,p^*)$ and the square commutes by \aref{lem:LR} (with $F=p^*=G$, $R=F(p_\#\E,-)$, and $R'=F(p^*p_\#\E,0)$).
 	Applying $p_\#$ everywhere, we obtain the commutative diagram
	\[
	\begin{tikzcd}
		p_\#(p^*D_S(p_\#\E)\otimes_Y\E) \ar[r]\ar[d] & p_\#(p^*D_S(p_\#\E)\otimes_Y p^*p_\#\E)
		\ar[dr, bend left =20]\ar[d] &  & \\
		p_\#(D_Y(p^*p_\#\E)\otimes_Y\E)\ar[r]\ar[rd, bend right=18] & p_\#(D_Y(p^*p_\#\E)\otimes_Y p^*p_\#\E)\ar[r] &p_\#p^*\SS_S \ar[r]& \SS_S.\\
		& p_\#(D_Y(\E)\otimes\E)\ar[ur, bend right =16] &
		\end{tikzcd}
	\]
	where the lower piece of the diagram commutes for formal reasons: given a map $\varphi:M\to N$ in any symmetric monoidal category, the induced diagram
	\[ 
	\begin{tikzcd}
		D(N)\otimes M\ar[rr, "D(\varphi)\otimes\id_M"]\ar[d, "\id_{D(N)}\otimes\varphi"'] & & D(M)\otimes M\ar[d] \\
		D(N)\otimes N\ar[rr] & &\SS
		\end{tikzcd}
	\]
	commutes.
	This proves the claim.
\end{proof}

Let $q:X\to S$ be an object in $\Sm_S^{G,\Nfree}$
and write $\oX =X/N$, $f:X\to\oX$ for the quotient. Since $N$ acts trivially on $S$, the structure map 
factors through the quotient and we have the diagram in $\Sm_S^G$
\begin{equation}\label{eqn:tri}
\begin{tikzcd}
	X\ar[r, "{f}"]\ar[dr, "q"'] & \oX \ar[d, "{p}"] \\
	& S,
\end{tikzcd}
\end{equation}
where $f$ is finite \'etale. Our first goal is to establish \aref{thm:pidual}, which says that when $q_\#\SS_X$ is dualizable, its dual is computed as $\pi_!D_S(q_\#\SS_X)$.

We define a candidate evaluation map 
\begin{equation}\label{eqn:ev}
\epsilon:\pi_!D_S(p_\#f_{\#}\SS_X) \otimes_S p_\#\SS_{\oX} \to \SS_{S}
\end{equation}
as follows:
\begin{align*}
\pi_!D_S(p_\#f_{\#}\SS_X) \otimes_S p_\#\SS_{\oX} & \xrightarrow{\sim} p_\#p^*\pi_!D_S(p_\#f_{\#}\SS_X)  \\ 
& \to p_\#\pi_!D_{\oX}(p^*p_\#f_{\#}\SS_X)\\
&  \to  p_\#\pi_!f_\#\SS_X\simeq p_\#\SS_{\oX} \\
& \to \SS_S .
\end{align*}
Here the first equivalence is the projection formula, the second arrow is the equivalence 
$ p^*\pi_!\simeq \pi_!p^* $ together with the exchange $p^*D_S\to D_{\oX} p^*$, the
third arrow  is induced by the unit of the adjunction $(p_\#, p^*)$ together with the equivalence $D_{\oX}(f_{\#}\SS_X)\simeq f_{\#}\SS_X$, and the last arrow is induced by the counit of the adjunction $(p_\#, p^*)$. 

\begin{remark}\label{rem:altepsilon}
	The adjoint of $\epsilon$ is the map
	\begin{equation}\label{eqn:e}
	e:\pi_!D_S(q_\#\SS_X)\to D_S(p_\#\SS_{\oX}).
	\end{equation}
	The map $e$ can also be described as the adjoint of the map
	\[
	e':D_S(q_\#\SS_X)\to \pi^*D_S(p_\#\SS_{\oX}),
	\]
	which is the dual of the composite
	$	p_\#\SS_{\oX}\to p_\#f_*f^*\SS_{\oX}\simeq p_\#f_\#\SS_X$.
\end{remark}

Note that the evaluation map for $q_{\#}\SS_X$ factors canonically as
\[
D_S(q_{\#}\SS_X) \otimes_{S} q_\#\SS_{X} \to i_!\SS_{\EE\FF(N)}\to \SS_{S},
\]
where $\SS_{\EE\FF(N)}\in \SH^{G,\Nfree}(S)$ is the unit.
\begin{lemma}\label{lem:evalcomp}
	The following diagram commutes:
	\[
	\begin{tikzcd}
		\pi_!(D_S(q_{\#}\SS_X) \otimes_{S} q_\#\SS_{X}) \ar[d]\ar[r]	& \pi_!\SS_{\EE\FF(N)}\ar[d]\\
		\pi_!(D_S(q_{\#}\SS_X)) \otimes_{S} p_\#\SS_{\oX} \ar[r, "{\epsilon}"] & \SS_{S}. 
	\end{tikzcd}
	\]	
\end{lemma}
\begin{proof}
	The left vertical map comes from the op-lax monoidality of $\pi_!$. It may also be described as the map induced by $p_\#f_\#\SS_X\to p_\#\SS_{\oX}$ together with the equivalence
	\[
	\pi_!(D_S(p_{\#}f_\#\SS_X) \otimes_{S} p_\#\SS_{\oX})\simeq \pi_!D_S(p_{\#}f_\#\SS_X) \otimes_{S} p_\#\SS_{\oX}.
	\]
Consider the following diagram
	\[
	\begin{tikzcd}
		p_\#\pi_!\big(p^*D(q_{\#}\SS_X)\otimes_{\oX} f_\#\SS_X\big) \ar[r]\ar[d] &[-8pt] p_\#\pi_!\big(D(f_\#\SS_X) \otimes_{\oX} f_\#\SS_X\big) \ar[r]\ar[d] &[-5pt] p_\#\SS_{\oX} \ar[r] &[-9pt] \SS_S  \\
		p_\#\pi_!\big(p^*D(q_{\#}\SS_X)\otimes_{\oX} \SS_{\oX}\big) \ar[r]& p_\#\pi_!\big(D(f_\#\SS_X)\otimes_{\oX} \SS_{\oX}\big),\ar[ur]  & &
	\end{tikzcd}
	\]
	where the top and bottom horizontal arrows of the square are the composite of the exchange $p^*D_S\to D_{\oX}p^*$ with the map induced by the unit of the adjunction $(p_\#,p^*)$.
	Via \aref{lem:dcomm} and the equivalences  $\pi_!p_\#\simeq p_\#\pi_!$, 
	we see that the top row is identified with the composite of the top horizontal and right vertical arrows in the diagram of the lemma. The composite around the lower part of this diagram is identified with the composite of the left vertical and bottom horizontal arrows of the diagram in the lemma. That this diagram commutes follows from the next lemma together with the equivalence
	$D_{\oX}(f_\#\SS_X)\simeq f_{\#}\SS_X$.
\end{proof}

\begin{lemma}
	The following diagram commutes
	\[
	\begin{tikzcd}
		f_{\#}\SS_X\underset{\overline{X}}{\otimes}f_{\#}\SS_X \ar[rr]\ar[rd] & & f_\#\SS_{X}\\
		& f_\#\SS_X\underset{\overline{X}}{\otimes}\SS_{\overline{X}}\ar[ur, "\sim"'] &,
		\end{tikzcd}
	\]
	where the horizontal map is the projection formula $f_{\#}\SS_X\underset{\overline{X}}{\otimes}f_{\#}\SS_X\simeq f_{\#}f^*f_\#\SS_X$ followed by the counit (using the ambidexterity equivalence $f_\#\simeq f_*$).
\end{lemma}

\begin{proof}
	Let $p_i:X\times_{\oX} X\to X$ be the projection to the $i$th factor and $\Delta:X\to X\times_{\oX} X$ the diagonal. Under the ambidexterity equivalence $f_\#\simeq f_*$, the counit of the adjunction $(f^*,f_*)$ is the arrow $f^*f_\#\to \id$, defined as the composite 
	\[
	f^*f_\# \simeq p_{2\#}p_1^*\to p_{2\#}\Delta_*\Delta^*p^*_1\simeq p_{2\#}\Delta_*\simeq \id, 
	\]
	see \cite[Theorem 6.9]{Hoyois:6}. In particular, together with the equivalence $f_\#p_{2\#}\simeq f_\#p_{1\#}$, we see the morphism $f_\#f^*f_\#\SS_X\to f_\#\SS_X$ in $\SH^G_\bb(\oX)$ 
	can be represented by  the projection 
	\[
	f_{\#}(X\times_{\overline{X}}X)_+\to f_{\#}\frac{(X\times_{\overline{X}}X)_+ }{(X\times_{\overline{X}}X\smallsetminus \Delta(X))_+}\simeq f_{\#}X_+,
	\]
	where $X\times_{\overline{X}}X$ is an $X$-scheme via the \emph{first} coordinate.
	The lemma follows from the commutativity of the diagram where the rows are the cofiber sequences in $\HH^G_\bb(\oX)$ associated to the closed immersions $\Delta(X)\subseteq X\times_{\overline{X}} X$ and $\Delta(f)\subseteq X\times_{\overline{X}}\overline{X}$, the diagonal and the graph of $f$, respectively:
	\[
	\begin{tikzcd}
		(X\times_{\overline{X}} X \smallsetminus \Delta(X))_+\ar[r]\ar[d] & (X\times_{\overline{X}} X)_+\ar[r, "p_1"]\ar[d] & X_+\ar[d]\\
		(X\times_{\overline{X}}\overline{X} \smallsetminus \Delta(f))_+\ar[r] & (X\times_{\oX}\oX)_+\ar[r, "\sim"] & X_+ .
		\end{tikzcd}
	\]
	
\end{proof}

Now we suppose that  that $q_\#\SS_X$ is dualizable in $\SH^{G,\Nfree}$ and that $G$ is isomorphic to a semi-direct product of $N$ and $G/N$. We define a ``coevaluation'' 
\begin{equation}\label{eqn:coev}
\eta:\SS_{S}\to    p_\#\SS_{\oX}  \otimes_S  \pi_!D_S(p_\#f_{\#}\SS_X)
\end{equation}
as follows. 
Since $G$ is isomorphic to a semi-direct product of $G/N$ and $N$, there is a map $\SS_S\to \pi_!\SS_{\EE\FF(N)}$ which splits the canonical map $\pi_!\SS_{\EE \FF(N)}\to \SS_{S}$, obtained  by taking the $N$-quotient of the (unstable) map $N'_+\to \SS_{\EE\FF(N)}$, where $N'$ is the $G$-set $G/(G/N)$. 
Now, since $q_\#\SS_{X}$ is dualizable in $\SH^{G}(S)$ it is dualizable in $\SH^{G,\Nfree}(S)$ by \aref{lem:freedual} and the duals in both categories agree under the standard inclusion. This means we have a coevaluation map 
\[
\coev:\SS_{\EE\FF(N)} \to   p_\#f_\#\SS_{X}  \otimes_S  D_S(p_\#f_{\#}\SS_X).
\]
 Now,
$\eta$ is defined as the composite
\begin{align*}
\SS_{S}\to \pi_!\SS_{\EE\FF(N)}& \xrightarrow{\pi_{!}\coev}  \pi_{!}( p_\#\SS_{X}  \otimes_S  D_S(p_\#f_{\#}\SS_X)) \\
& \to  p_\#\SS_{\oX}  \otimes_S  \pi_!D_S(p_\#f_{\#}\SS_X).
\end{align*}

\begin{theorem}\label{thm:pidual}
	Let $q:X\to S$ be an object of $\Sm^{G,\Nfree}_S$ and $f:X\to \oX$, $p:\oX\to S$  as in \eqref{eqn:tri}. Suppose that $q_{\#}\SS_X\in\SH^G(S)$ is dualizable. Then 
	$p_{\#}\SS_{\oX}$ is dualizable in $\SH^{G/N}(S)$  and
	\eqref{eqn:e}
	is an equivalence 
	\[
	\pi_!D_S(q_\#\SS_X)\xrightarrow{\sim} D_S(\pi_!q_\#\SS_X).
	\] 	
\end{theorem}
\begin{proof}
	First we explain why it suffices to assume that $G$ is a semi-direct product of $N$ and $G/N$. Suppose that the result is true in this case.  Given an arbitrary $G$ and $N$ and let $X$ be a dualizable $N$-free smooth $G$-scheme over $S$, consider the cartesian square \eqref{eqn:cartd}. Since $X$ is dualizable over $S$, as a smooth $G$-scheme, the $\ker(\pr_2)$-free $G\times_{G/N}G$-scheme $\pr_1^*X$ is dualizable over $S$.  Since 
	$G\times_{G/N}G$ is a semi-direct product of $G$ and $\ker(\pr_2)$, our hypothesis implies that the composite is an equivalence
	\[
	\pr_{2!}\pr_{1}^*D_S(q_\#\SS_X)\xrightarrow{\sim} \pr_{2!}D_S(\pr_{1}^*q_\#\SS_X)\xrightarrow{e} D_S(\pr_{2!}\pr_{1}^*q_\#\SS_X),
	\] 	
	where we again write $e$ for an instance of \eqref{eqn:e} for the group $G\times_{G/N}G$ and we use that $\pr_{1}^*D_S(q_\#\SS_X)\simeq D_S(\pr^*_1q_\#\SS_X)$ 
	since $q_\#\SS_X$ is dualizable and $\pr^*_1$ is symmetric monoidal.
	We claim that this implies that $p_{\#}\SS_{\oX}$ is dualizable in $\SH^{G/N}(S)$. Indeed, $\pr_{2!}\pr_{1}^*q_\#X\simeq \pi^*\pi_!q_\#\SS_{X}$ is dualizable in
	$\SH^{G}(S)$ and therefore  
	$\Phi^N(\pi^*\pi_!q_\#\SS_X)\simeq\pi_!q_\#\SS_{X}\simeq p_{\#}\SS_{\oX}$ is dualizable  in $\SH^{G/N}(S)$ as $\Phi^N$ is symmetric monoidal.
		In this case  $\epsilon$ and $\eta$ defined above are evaluation and coevaluation maps for the duality pairing.

The diagram
\begin{equation}\label{eqn:five}
\begin{tikzcd}
\pi^*\pi_!D_S(q_\#\SS_X)\ar[r, "\pi^*e"] & \pi^*D_S(\pi_!q_\#\SS_X) \ar[r, "\sim"] & D_S(\pi^*\pi_!q_\#\SS_X)\ar[d, "\sim"] \\
\pr_{2!}\pr_1^*D_S(q_\#\SS_X) \ar[u, "\sim"'] \ar[r,"\sim"] & \pr_{2!}D_S(\pr_1^*q_\#\SS_X) \ar[r, "e", "\sim"'] & D_S(\pr_{2!}\pr_1^*q_\#\SS_X)
\end{tikzcd}
\end{equation}
commutes. This follows by adjointness from the commutativity of the diagram
\[
\begin{tikzcd}
\pi^*\pi_!D_S(q_\#\SS_X) \otimes \pi^*\pi_!q_\#\SS_X \ar[r] & \SS_S \\
\pr_{2!}\pr_1^*D_S(q_\#\SS_X) \otimes \pr_{2!}\pr_1^*q_\#\SS_S, \ar[u, "\sim"]\ar[ur] &
\end{tikzcd}
\]
where the the horizontal and diagonal arrows  come from \eqref{eqn:ev}. We thus find that
	\[ 
	\pi^*\pi_!D_S(q_\#\SS_X)\xrightarrow{\pi^*e} \pi^*D_S(\pi_!q_\#\SS_X)
	\]	
	is an equivalence. Since $\pi^*$ is conservative, it follows that $e$ is an equivalence.

	It remains to establish the result under that assumption that $G$ is a semi-direct product. We show that in this case the evaluation and coevaluation maps \eqref{eqn:ev} and \eqref{eqn:coev} satisfy the triangle identities.
	To compactify notation, we write $X=p_{\#}f_{\#}\SS_{X}$, $\oX=p_\#\SS_{\oX}$, $\SS=\SS_S$, and $\SS_{\EE\FF}=\SS_{\EE\FF(N)}$.
	We have to check that the two composites below are the identity
	
	\begin{align*}
	{\oX} \otimes \SS & \xrightarrow{\id\otimes\eta} {\oX}\otimes \pi_!D(X)\otimes {\oX} \xrightarrow{\epsilon\otimes \id} \SS\otimes {\oX}\\
	\SS \otimes \pi_!D({X})   & \xrightarrow{\eta\otimes \id}  \pi_!D(X)\otimes {\oX} \otimes \pi_!D({X})\xrightarrow{\id\otimes\epsilon} \pi_!D({X})\otimes \SS.
	\end{align*}
	To establish the first identity we observe that we have a commutative diagram
	\[
	\begin{tikzcd}[column sep =small] 
		\oX\otimes\SS \ar[rrr, "\id\otimes \eta"]\ar[dr]&[-40pt] &[-15pt]  & \oX\otimes \pi_!D(X)\otimes \oX \ar[rr, "{\epsilon\otimes\id}"] &[-13pt] &[-8pt] \SS\otimes \oX \\
		&  \oX\otimes \pi_!\SS_{\EE \FF} \ar[rr]\ar[urr, "{\id\otimes \pi_{!}(\coev)}"]&  & \oX\otimes \pi_!(D(X)\otimes X) \ar[u] &\pi_!(X\otimes D(X))\otimes \oX \ar[r] \ar[lu] & \pi_!\SS_{\EE\FF}\otimes \oX \ar[u] \\
		\pi_!(X\otimes \SS_{\EE\FF}) \ar[ur]\ar[rrr, "{\pi_!(\id \otimes \coev)}"] \ar[uu, "{\id}"]& & &\pi_!(X\otimes D(X)\otimes X)\ar[rr, "{\pi_!(\ev\otimes \id)}"]\ar[ur]\ar[u] &   & \pi_!(\SS_{\EE\FF}\otimes X), \ar[u]	\ar[uu, bend right=70, "{\id}"']
	\end{tikzcd} 
	\]
	in which $X\simeq X\otimes\SS_{\EE\FF}$ and the upper righthand square commutes by Lemma \ref{lem:evalcomp}.
	The second identity is established by a similar diagram. 
\end{proof}

Write $h:D_S\pi_!\xrightarrow{\sim}\pi_*D_S$ for the canonical equivalence. 
\begin{proposition}
	Let $q:X\to S$ be an object of $\Sm^{G,\Nfree}_S$ such that  that $q_{\#}\SS_X\in\SH^G(S)$ is dualizable.	The diagram
	\[
	\begin{tikzcd}
	D_SD_S(\pi_!q_\#\SS_X)\ar[r, "D(e)", "{\sim}"']&  D_S(\pi_!D_S(q_\#\SS_X)) \ar[r, "\sim"', "{h}"] &  \pi_*D_SD_S(q_\#\SS_X)\\
	\pi_!q_\#\SS_X \ar[rr, "\tau"]\ar[u, "\sim"',"{\can}"]&   &  \pi_*q_\#\SS_X \ar[u,"\sim", "\pi_*\can"']
	\end{tikzcd}
	\]
	commutes, where $e$ is the map from \aref{rem:altepsilon}.

\end{proposition}
\begin{proof}
Write $D = D_S$ and $X=q_\#\SS_X$. 	
The commutativity of the diagram of the proposition is equivalent, by adjointness, to that of
	\[
	\begin{tikzcd}
	\pi^*DD(\pi_!X)\ar[r, "\pi^*hD(e)"]&  \pi^*\pi_*DD(X)  \ar[r, "\epsilon"]&  DD(X)\\
	\pi^*\pi_!X \ar[rr, "\hat{\tau}"]\ar[u, "\pi^*\can"]&   &  X \ar[u, "\can"'],
	\end{tikzcd}
	\]
where $\epsilon:\pi^*\pi_*\to \id$ is the counit.	
The outer composite of this diagram agrees with the outer composite of the following diagram below, which we will show to be commutative:
	\[
	\begin{tikzcd}[column sep=small]
	&[-5pt]	\pi^*DD(\pi_!X) \ar[r, "De"] \ar[d,  "\pi^*D\to D\pi^*"] &[-5pt]  \pi^*D(\pi_!D(X)) \ar[r, "\epsilon h"] \ar[d, "\pi^*D\to D\pi^*"'] \ar[d, phantom, yshift=.6ex, shift left=1.3ex, bend left=50, "\circled{\bf 4}"] & DD(X)  \\ 
	\pi^*\pi_!X \ar[ur, "\pi^*\can"]\ar[dr, "\can"']	
	& D(\pi^*D(X))\ar[r,"De"]& D(\pi^*\pi_!D(X))  \ar[dr, phantom, yshift=1.6ex, "\circled{\bf 3}"]\ar[ur,   "D(\id \to \pi^*\pi_!)"'] & \\
	& DD(\pi^*\pi_!X)\ar[u,"D(\pi^*D\to D\pi^*)"'] &  D(\pr_{2!}\pr_1^*D(X))\ar[r, "D\beta"] \ar[u]  & D(\pr_{2!}\Delta_!\Delta^* \pr_1^*D(X)) \ar[uu ]\\
		& & & D(\pr_{2!}\Delta_!\Delta^* D(\pr_1^*X))\ar[u, "D(\pr_1^*D\to D\pr_1^*)"']\\
	&  DD(\pr_{2!}\pr_1^*X) \ar[uu]\ar[r,"De"]\ar[dr, "\phi"'] &  D(\pr_{2!}D(\pr_1^*X)) \ar[uu, " D(\pr_1^*D\to D\pr_1^*)"]\ar[r,"\phi"] \ar[ur, "D\beta"]\ar[luuu, phantom, "\circled{\bf 5}"] & D(\pr_{2!}D(\Delta_!\Delta^* \pr_1^*X))  \ar[u , "D\gamma"']\ar[u, phantom, shift left=1, yshift=-.6ex, bend left=60, "\circled{\bf 2}"] \\
	& &DD(\pr_{2!}\Delta_!\Delta^*\pr_1^*X) \ar[ur, "De"' ] & \\
	& \pr_{2!}\pr^*_1X\ar[uu, "\can"] \ar[r,"\phi"] \ar[uuuuul, bend left, "\sim"] &\pr_{2!}\Delta_!\Delta^*\pr^*_1X \ar[u, "\can"]\ar[r, "\sim" ] & X. 
	\ar[uuuuuu, bend right=70, "\sim"'] \ar[uu, phantom, "\circled{\bf 1}"] \\
	\end{tikzcd} 
	\]
	Here,   $\pr_i$ and $\Delta$ are as in \ref{sub:atrans},
$\phi:\id\to \Delta_!\Delta^*$ arises from the Wirthmueller isomorphism $\Delta_!\simeq \Delta_*$ together with the unit $\id \to \Delta_*\Delta^*$, $\beta:\Delta_!\Delta^*\to \id$ is the counit, and we identify $\pi_*\pi_!X\simeq\pr_{2!}\pi_1^*X$ as in \aref{rem:conv}. The map $\gamma$ is induced by the composite of exchanges
\[
\Delta_!D \xrightarrow{\psi} D\Delta_*\;\;\;\textrm{ and }\;\;\; D\Delta^* \xrightarrow{\psi_R} \Delta^*D
\]
as follows:
First, the exchange $\psi$ is the equivalence fitting in the commutative diagram
\[
\begin{tikzcd}
\Delta_!D\ar[d, "\sim"']\ar[r,"\psi"] & D\Delta_* \ar[d, "\sim"] \\
\Delta_*D & D\Delta_! \ar[l,"\sim"'],
\end{tikzcd}
\]
where the vertical equivalences are the Wirthmueller isomorphism and the bottom one is the standard equivalence. In particular $\psi$ is an equivalence. The exchange $\psi_{R}$ is its right mate
 and $\psi_R$ is an equivalence on dualizable objects, so in particular $\psi_R$ applied to $X$ is invertible. 
 By \aref{lem:LR}, the diagram
 \begin{equation}\label{eqn:uco} 
 \begin{tikzcd}
 D\Delta_!\Delta^*q_\#\SS_X \ar[rrr, phantom, bend left = 17, "\gamma"'] & D\Delta_*\Delta^*q_\#\SS_X \ar[l,"\sim"']\ar[rd, "D({\rm unit})"'] 
 & \Delta_!D\Delta^*q_\#\SS_X \ar[l,"\psi"']\ar[r, "\psi_R"] & \Delta_!\Delta^*D q_\#\SS_X \ar[dl, "{\rm counit}_D"]
	\ar[lll, rounded corners, 	"\gamma"',
 to path={ -- ([yshift= 3.0ex]\tikztostart.north) - |([yshift= 1.5ex]\tikztotarget.north) -- (\tikztotarget) }]
\\
 & & Dq_\#\SS_X & 
 \end{tikzcd}
 \end{equation}
 commutes, where $\gamma$ is by definition the upper horizontal composite. This implies that the subdiagram marked $\circled{\bf 2}$ above commutes.

Next to see that $\circled{\bf 1}$ commutes, we need to see that the diagram
\[
\begin{tikzcd}
\pr_{2!}D(\Delta_!\Delta^*\pr_2^*q_\#\SS_X)\ar[r,"e"] & D(\pr_{2!}\Delta_!\Delta^*\pr_2^*q_\#\SS_X)\\
\pr_{2!}\Delta_!\Delta^*\pr_2^*D(q_\#\SS_X)\ar[r,"\sim"]\ar[u,"\gamma'"] & D(q_\#\SS_X)
\ar[u, "\sim"]
\end{tikzcd}
\] 
commutes, where $\gamma'$ is the composite of $\gamma$ with the exchange $\pr_1^*D\to D\pr_1^*$. 
Using \aref{rem:altepsilon} and that $\Delta_!\Delta^*\pr_1^*q_\#\SS_X\simeq m_\#\SS_{G'/G\times X}$, where $m:G'/G\times X\to S$ is the $G'$-equivariant structure map, we see that the arrow $e:\pr_{2!}D(\Delta_!\Delta^*\pr_1^*q_\#\SS_X)\to 
D(\pr_{2!}\Delta_!\Delta^*\pr_1^*q_\#\SS_X)$ is obtained as the adjoint of
the map 
\[
D(\Delta_!\Delta^*\pr_1^*q_\#\SS_X)\simeq D(\Delta_*\Delta^*\pr_2^*q_\#\SS_X)\xrightarrow{D({\rm unit})} D(\pr_2^*q_\#\SS_X)\simeq 
\pr_2^*D(q_\#\SS_X).
\]
It now follows, using \eqref{eqn:uco},
that this diagram, and hence $\circled{\bf 1}$, commutes.

  That $\circled{\bf 3}$ commutes follows from the 
 	commutativity of the diagram for $W\in \SH^{G,\Nfree}(S)$
 	\[
 	\begin{tikzcd}
 	W \ar[r] \ar[d, "\sim"']& \pi^*\pi_!W \\
 	\pr_{2!}\Delta_!\Delta^*\pr_1^*W \ar[r, "\beta"] & \pr_{2!}\pr_1^*W. \ar[u, "\sim"'] 
 	\end{tikzcd}
 	\]
 	This commutativity can be checked for $W$ the suspension spectrum of a smooth $N$-free $G$-scheme over $S$, which is straightforward to verify.  
	Using \aref{lem:LR}, we see that the subdiagram $\circled{\bf 4}$ commutes. That subdiagram $\circled{\bf 5}$ follows by applying $D$ to the diagram \eqref{eqn:five}.
The remaining subdiagrams are easily seen to commute.
\end{proof}

Next, we verify that $\EE\FF[N]_+$ is a colimit of   dualizable spectra.  
Let $\FF$ be a family. 
If $X\in \Sm^{G}_B$, recall that we write
\[
X^\FF = \cup_{H\in \co(\FF)}X^{H}\;\textrm{and}\; X(\FF) = X\minus X^{\FF}.
\]

Write $f:X\to B$ and  $g:X(\FF)\to B$ for the structure maps. 

\begin{lemma}
	Suppose that $f_{\#}(\SS_X)$ is 
	dualizable in $\SH^{G}(B)$. Then $g_{\#}(\SS_{X(\FF)})$ is dualizable. 
\end{lemma}
\begin{proof}
	Filter the inclusion $\FF\subseteq \FF_{\all}$ by adjacent families
	$\FF=\FF_{0}\subseteq\FF_1\subseteq \FF_{n}=\FF_{\all}$. 
	Write $\alpha_{i}:X(\FF_i)\subseteq X$ and $\beta_{i}:X(\FF_{i})^{\FF_{i-1}}\subseteq X(\FF_{i})$ for the inclusions and write $f_{i}:X(\FF_i)\to B$ for the induced structure map. 
	From the  gluing sequence, we obtain exact sequences
	\[
	f_{i\#}(\SS_{X(\FF_{i})}) \to f_{i+1\#}(\SS_{X(\FF_{i+1})}) \to
	f_{i+1\#}\beta_{i+1 *}(\SS_{X(\FF_{i+1})^{\FF_{i}}}).
	\]
	Let $H_i\leq G$ be a subgroup such that $\FF_{i+1}\minus \FF_{i}=\{(H_i)\}$. Then since $X(\FF_{i+1})^{\FF_{i}}$ is concentrated at $(H_i)$, we have that 
	$X(\FF_{i+1})^{\FF_{i}}\cong G\times_{\NN H_i}X(\FF_{i+1})^{H_i}$ by \aref{lem:Xconc}.
	We then have that 
	\[
	f_{i+1\#}(\beta_{i+1})_*(\SS_{X(\FF_{i+1})^{\FF_{i}}})
	\simeq G_+\wedge_{\NN H_i}(f'_{H_i})_\#(\SS_{X(\FF_{i+1})^{H_{i}}}),
	\]
	where $f'_{H_i}:X(\FF_{i+1})^{H_{i}}\to B$
	is the structure map.

	Now suppose that $f_{i+1\#}(\SS_{X(\FF_{i+1})})$ is dualizable. It follows that 
	$(f'_{H_i})_\#(\SS_{X(\FF_{i+1})^{H_{i}}})$ is also dualizable, since this is obtained by applying the geometric fixed points functor.
	Therefore 
	$f_{i+1\#}(\beta_{i+1})_*(\SS_{X(\FF_{i+1})^{\FF_{i}}})$ is dualizable and we conclude that 
	$f_{i\#}(\SS_{X(\FF_{i})})$ is dualizable as well.
	Since we have assumed that $f_{\#}(\SS_X)=f_{n\#}(\SS_{X(\FF_{n})})$ is dualizable, the result follows by (finite) induction.

\end{proof}

\begin{corollary}\label{cor:EEdual}
	Let $S\in \Sch^G_B$.	Given a family $\FF$, $\SS_{\EE\FF}\in \SH^{G}(S)$ can be expressed as 
	\[
	\SS_{\EE\FF}\simeq \colim_{n\in \NN}q_{n\#}\SS_{U_n}
	\]
	where $q_{n}:U_n\to S$ is in $\Sm^{G}_S[\FF]$ and  $q_{n\#}\SS_{U_n}$ is dualizable in $\SH^{G}(S)$.
\end{corollary}
\begin{proof}
	It suffices to show this when $S=B$. Use the previous lemma together with the presentation of \aref{ex:bc}.
\end{proof}

\begin{corollary}\label{cor:adual}
	Suppose that $q_{\#}\SS_X\in\SH^G(S)$ is dualizable
	where $q:X\to S$ be an object of $\Sm^{G,\Nfree}_S$. Then the Adams transformation is an equivalence
	\[
	\tau:\pi_{!}q_{\#}\SS_{X} \xrightarrow{\sim} \pi_*i_!q_{\#}\SS_{X}.
	\]
		In particular, 
	\[
	\tau:\pi_{!}\SS_{\EE\FF(N)} \xrightarrow{\sim} \pi_*i_!\SS_{\EE\FF(N)}
	\]
\end{corollary}
\begin{proof}
	The first statement is immediate from the previous proposition. The second statement then follows, using  
	\aref{cor:EEdual}.
\end{proof}

Next we will define a transformation 
\[
\nu: \pi_*i_!\to \pi_!, 
\]
which we will show is inverse to $\tau$. We begin with the following observation. Recall that there is a transformation $\pi_*i_!i^*(\SS_S)\otimes \id \to \pi_*i_!i^*\pi^*$, see \eqref{eqn:pformfree}.
\begin{lemma}
	The map \eqref{eqn:pformfree} is an equivalence.
\end{lemma}
\begin{proof}
Let $Y\in \SH^{G/N}(S)$ and write $\EE\FF(N) \simeq \colim_n U_n$ with $q_{n\#}\SS_{U_n}$  dualizable over $S$, see \aref{cor:EEdual}. Then 
\eqref{eqn:pformfree} evaluated on $Y$ is the colimit of the transformations
\[
\pi_*(q_{n\#}\SS_{U_n})\otimes Y \to \pi_*(q_{n\#}\SS_{U_n}\otimes \pi^*Y).
\] 
We see that each of these transformations is an equivalence since each fits into the commutative diagram
\[ 
\begin{tikzcd}
\pi_*(DD(q_{n\#}\SS_{U_n})\otimes \pi^*Y) & \pi_*F(D(q_{n\#}\SS_{U_n}), \pi^*Y)) \ar[l, "\sim"']  & F(\pi_!D(q_{n\#}\SS_{U_n}), Y) \ar[l, "\sim"'] \\
\pi_*DD(q_{n\#}\SS_{U_n})\otimes Y \ar[u]& \pi_!q_{n\#}\SS_{U_n}\otimes Y \ar[l,"\sim"'] & F(D(\pi_!q_{n\#}\SS_{U_n}),Y). \ar[l, "\sim"']\ar[u, "\sim"']
\end{tikzcd}
\]	
	
\end{proof}

Now define $\upsilon:\pi_*i_!\to \pi_!$ as the composite
\begin{align*}
\pi_*i_!& \xrightarrow{{\rm unit}} \pi_*i_!i^*\pi^*\pi_!\simeq \pi_*(i_!i^*(\SS_S)\otimes \pi^*\pi_!)
\\
& \xrightarrow{\sim}\pi_*(i_!i^*(\SS_S))\otimes \pi_! \xrightarrow{\tau^{-1}}
\pi_!i^*(\SS_S)\otimes \pi_! \to \pi_!.
\end{align*}

\begin{lemma}\label{lem:tauinv}
	The composite $\pi_!\xrightarrow{\tau}\pi_*i_!\xrightarrow{\upsilon} \pi_!$ is an equivalence.
\end{lemma}
\begin{proof}
	The composite $\upsilon\tau$ agrees with the composite around the following diagram
	\[
	\begin{tikzcd}
	\pi_! \ar[r, "\tau"] \ar[d]& \pi_*i_! \ar[r] & \pi_*i_!i^*\pi^*\pi_! \ar[d,"\sim"] \ar[r] & \pi_!\\
	\pi_!i^*(\SS_S)\otimes \pi_!\ar[rr, "\tau\otimes \id"',"\sim"] & & \pi_*i_!i^*(\SS_S)\otimes \pi_!\ar[r, "\tau^{-1}\otimes \id"', "\sim" ]
	& \pi_!i^*(\SS_S)\otimes \pi_!. \ar[u]
	\end{tikzcd}
	\]
\end{proof}

\begin{theorem}[Adams Isomorphism]\label{thm:adams}
	The Adams transformation $\tau:\pi_!\to \pi_*i_!$ is an equivalence.	
\end{theorem} 
\begin{proof}
	By \aref{cor:Nfreegen}, it suffices to show that $\tau$ is an equivalence on $\Sigma^{-\VV}q_*\SS_X$ where $\VV$ is an $N$-trivial representation and $q:X\to S$ is an $N$-free (not necessarily smooth) $G$-scheme over $S$. Write $f:X\to \overline{X}$ for the quotient and $p:\oX\to S$ for the induced map. Consider the diagram
	\[
	\begin{tikzcd}
	\pi_!p_*f_*\SS_X\ar[r, "\tau"] \ar[d] & \pi_*i_!p_*f_*\SS_X \ar[r, "\upsilon"]\ar[d, "\sim", "\upsilon' "'] & \pi_!p_*f_*\SS_X \ar[d] \\
	p_*\pi_!f_*\SS_X \ar[r, "\tau", "\sim"'] & p_*\pi_*i_!f_*\SS_X \ar[r, "\upsilon", "\sim"'] & p_*\pi_!f_*\SS_X.
	\end{tikzcd} 
	\]
	The left-hand square commutes by \aref{lem:comtaup*} and \aref{lem:f*i!} and the right-hand square commutes by  
a similar argument. Both  horizontal composites are equivalences by \aref{lem:tauinv}, and therefore the outer vertical arrows are equivalences as the middle one is. Finally, 	$\tau:\pi_!f_*\SS_X \to \pi_*i_!f_*\SS_X$ is an equivalence by \aref{cor:adual} (since $f_\#\SS_X\simeq f_*\SS_X$, by ambidexterity \cite{Hoyois:6}, as $f$ is finite \'etale). 
\end{proof}

\subsection{Applications}\label{sub:appl}
In this section, we present a few applications of the Adams isomorphism.

	Let ${\rm B} G$ denote the classifying stack of the group $G$ and and $f:{\rm B} G\to {\rm B}(G/N)$ the resulting proper map of stacks. We have an equivalence $\SH({\rm B} G)\wkeq \SH^G(S)$ and from this perspective the fixed point functor $\pi_*$ becomes identified with the pushforward functor $f_*$. 
	The following base change results are an instance of the proper and the smooth-proper base change formula, but with the curious feature that $f$ is not a representable morphism. 	
These don't follow immediately from the six functor formalism in \cite{Hoyois:6}, precisely because $f$ is not representable.

We use the names for exchange morphisms given in the first part of the previous subsection.

\begin{corollary}[Proper base change]\label{cor:properbc}
Let $p:T\to S$ be a morphism in $\Sch^{G/N}_B$. The exchange
	\[
	\alpha_R:p^*\pi_*\to \pi_*p^*
	\]
	 is an equivalence. 
 	
\end{corollary}
\begin{proof}
	Choose a filtration $\emptyset=\FF_{-1}=\FF_0\subseteq \FF_1\subseteq \cdots \subseteq \FF_n= \FF_{{\rm all}}$ such that each pair $\FF_i\subseteq \FF_{i+1}$ is $N$-adjacent, see \aref{sec:filtadj}.
	This gives rise to the filtration of $X\in \SH^{G}(S)$,
	\[
	\ast\simeq \EE\FF_{-1+}\otimes X \to \EE\FF_{0+}\otimes X  \to \cdots \to \EE\FF_{n-1+}\otimes X \to \EE\FF_{n+}\otimes X \simeq X.
	\]
	It thus suffices to check that $\alpha_R$ is an equivalence on each filtration quotient $\EE(\FF_{i+1}, \FF_{i})\otimes X$. Suppose that $\FF\subseteq \FF'$ is $N$-adjacent at $H\leq N$. 
	By \aref{prop:Phiadj}, we find that 
	\begin{align*}
	p^*\pi_*(\EE(\FF',\FF)\otimes X) &\simeq p^*(G/N_+\ltimes_{\WW} (\EE\FF(\W_NH)_+\otimes X^{\mgf H})^{\W_NH} ) \\ 
	& \simeq	(G/N_+\ltimes_{\WW} p^*((\EE\FF(\W_NH)_+\otimes X^{\mgf H})^{\W_NH} )\\
	& \simeq	(G/N_+\ltimes_{\WW}(p^*(\EE\FF(\W_NH)_+\otimes X^{\mgf H}))^{\W_NH} \\
	& \simeq (G/N_+\ltimes_{\WW}(\EE\FF(\W_NH)_+\otimes (p^*X)^{\mgf H})^{\W_NH} \\
	& \simeq 	\pi_*p^*(\EE(\FF',\FF)\otimes X).
	\end{align*}
	Here, the third equivalence follows from \aref{thm:adams} and \aref{prop:puptau}, since
$\EE\FF(\W_NH)_+\otimes X^{\mgf H}$ is $\W_NH$-free. The fourth follows from \aref{prop:bcgfp}.

\end{proof}
\begin{corollary}[Smooth-proper base change]
	Let $p:T\to S$ be a smooth morphism in $\Sch^{G/N}_B$. The exchange
	\[
	\overline{\alpha}:p_\#\pi_*\to \pi_*p_\#
	\]
	 is an equivalence. 	
\end{corollary}
\begin{proof}
	Similar to the proof of proper base change.
\end{proof}

\begin{corollary}[Projection formula]
Let $S\in \Sch_B^{G/N}$.	There is a canonical equivalence 
\[\pi_*(X)\otimes Y\simeq\pi_*(X\otimes \pi^*(Y))
\]
for $X\in \SH^{G}(S)$ and $Y\in \SH^{G/N}(S)$. 
\end{corollary}
\begin{proof}
Consider the map $\pi_*(X)\otimes Y\to\pi_*(X\otimes\pi^*(Y))$ adjoint to the map
\[
\pi^*(\pi_*(X)\otimes Y)\simeq\pi^*\pi_*(X)\otimes\pi^*Y\to X\otimes\pi^*Y,
\]
where the equivalence follows from the symmetric monoidality of $\pi^*$ and the map is given by tensoring the the counit $\pi^*\pi_*X\to X$ with $\pi^*Y$.

Let $\FF\subseteq \FF'$ be an $N$-adjacent pair of families, say at $H\leq N$. 
 $\EE(\FF',\FF)\otimes X$. Then by \aref{prop:Phiadj}
\[
(\EE(\FF',\FF)\otimes X)^N\simeq G/N_+\ltimes_{\WW} (\EE\FF(\W_NH)_+\otimes X^{\mgf H})^{\W_NH} ).
\]
Under this equivalence, the transformation is identified with $G/N_+\ltimes_{\WW}-$ applied to the transformation
\[
(\EE\FF(\W_NH)_+\otimes X^{\mgf})^{\W_NH}\otimes Y \to  (\EE\FF(\W_NH)_+\otimes \Phi^{H}X\otimes \pi^*Y)^{\W_NH}.
\]
That this is an equivalence follows from the commutativity of the diagram
\[
\begin{tikzcd}
\pi_!W\otimes V \ar[r]\ar[d, "\tau\otimes \id "'] & \pi_!(W\otimes \pi^*V) \ar[d, "\tau "]\\
\pi_*W\otimes V \ar[r] & \pi_*(W\otimes \pi^*V) 
\end{tikzcd}
\] 
where $W\in \SH^{G,\Nfree}(S)$, $V\in \SH^{G/N}(S)$; the commutativity can be checked by an argument similar to ones above, e.g., that in \aref{prop:untwist}.

The  general case follows by choosing a filtration 
$\emptyset=\FF_{-1}=\FF_0\subseteq \FF_1\subseteq \cdots \subseteq \FF_n= \FF_{{\rm all}}$ such that each pair $\FF_i\subseteq \FF_{i+1}$ is $N$-adjacent and considering the induced filtration $\EE\FF_i\otimes X$ on $X$.

\end{proof}

\section{Splitting motivic  \texorpdfstring{$G$}{G}-spectra \texorpdfstring{\`a}{a} la tom Dieck}\label{sec:mtds}

Throughout this section, $N\trianglelefteq G$ is a normal subgroup and we assume that $N$ acts trivially on $S$.

For a subgroup $H\leq N$ we write $\WW H = \W_G H/\W_N H$ for the quotient of Weil groups. 
Let $\FF(\W_NH)$ be the family of subgroups $\{K\leq \W_GH \mid K\cap \W_NH = \{e\}\}$. 
As before, we write $\EE_{\W_NH}(\W_GH)=\EE\FF(\W_NH)  $ for the universal $\W_NH$-free $\W_GH$-motivic space, to emphasize the ambient group.

\begin{definition}
Let $X$ be a motivic $G$-spectrum and $H\leq G$ a subgroup.
A \emph{splitting of $X$ at  $H$} is a map $f_H$ in $\SH^{\W_GH}(S)$, 
\[
\begin{tikzcd}
X^H \ar[r] & X^{\mgf H} \ar[l, bend right, "\ctext{  $f_H$}"']
\end{tikzcd}
\]
 which splits the canonical map   $X^H\to X^{\mgf H}$. An  \emph{$N$-splitting} of $X$ is a choice of splitting of $X$ at each subgroup $H\leq N$.  
\end{definition}

\begin{example}
\begin{enumerate}
\item Let $Y\in \HH_\bullet^G(S)$. The suspension spectrum $\Sigma^{\infty}Y$ has a canonical splitting at any subgroup $H\leq G$, defined as follows. Write $\pi:\NN_GH\to \W_GH$ for the quotient. The counit of the adjunction $\pi^*\dashv (-)^H$ on based motivic $\NN_GH$-spaces yields the map $\pi^*(Y^{H}) \to Y$ of spaces and thus a map of $\NN_GH$-spectra
$\pi^*\Sigma^{\infty}(Y^H)\to \Sigma^{\infty}Y$. Its adjoint is the map
\[
\Sigma^{\infty}(Y^H)\to (\Sigma^{\infty}Y)^{H}.
\]
By  \aref{prop:gfp}, this induces the desired splitting.  

\item If $X\in \SH^{G/N}(S)$ then $\phi^*(X)$ is split. 

\item If $X$ and $Y$ are split at $H$ then $X\otimes Y$ is canonically split via the composition $(X\otimes Y)^{\mgf H}\wkeq X^{\mgf H}\otimes Y^{\mgf H} \to X^H\otimes Y^H \to (X\otimes Y)^H$.

\end{enumerate}
\end{example}

Write $i:\Sm^{G,\Nfree}_S\subseteq \Sm^G_S$ 
for the inclusion.

\begin{definition}
The \emph{motivic homotopy orbit point spectrum} of $X$ is 
\[
X_{\hh N}:= \pi_*i_!i^*(X) \simeq (\EE\FF(N)_+\wedge X)/N
\]
\end{definition}
Let $X$ be an $N$-split motivic $G$-spectrum. Let $H\leq N$ be a subgroup and  consider the composition, where for notational brevity, we write simply 
$\EE_H = \EE_{\W_NH}(\W_GH)$.
Define the map $\Theta_{X,H}$ as the following composite, where the maps are explained below, 
\begin{align*}\label{eqn:theta}
\begin{split}
{G/N}_+\ltimes_{\WW H}\left(X^{\mgf H} \right)_{\hh\W_NH}  &\wkeq  G/N_+\ltimes_{\WW H}(\EE_{H+}\otimes X^{\mgf H})^{W_NH} \\
& \wkeq  G/N_+\ltimes_{\WW H} ((\EE_{H+}\otimes X)^{\mgf H})^{W_NH} \\
&\xrightarrow{(f_H)_*}  G/N_+\ltimes_{\WW H} ((\EE_{H+}\otimes X)^H)^{W_NH} \\
& \wkeq (( G_+\ltimes_{\NN_GH}\EE_{H+})\otimes X)^{N} \\
&  \longrightarrow X^N.
\end{split}
\end{align*}
The map $f_H$ is the splitting of $X$ at $H$ and the last map is induced by the projection 
\[
G_+\ltimes_{\NN_GH}\EE_{H+}\wkeq (G\times_{\NN_HG}\EE_H)_+ \to S^0.
\] 
The first equivalence comes from the Adams isomorphism. The second comes from the monoidality of geometric fixed points and that $H$ acts trivially on 
$\EE\FF(\W_NH)_+$. The fourth map, which is an equivalence, comes from a canonical exchange of functors (see proof of \aref{prop:Phiadj}) together with the projection formula for induction-restriction.
Note that this map only depends on the $G$-conjugacy class of the subgroup $H$.

Now define the map of motivic ${G/N}$-spectra
$$
\Theta_X:\bigoplus_{(H)} 
{G/N}_+\ltimes_{\WW H}\left(X^{\mgf H} \right)_{\hh\W_NH}
\to X^N, 
$$ 
where the index is over the set of $G$-conjugacy classes of subgroups of $N$,  
to be the sum over the maps $\Theta_{X,H}$ defined above.  The map $\Theta_X$ 
is natural with respect to maps of $N$-split spectra which are compatible with splitting.

\begin{theorem}[Motivic tom Dieck splitting]\label{thm:mtd}
Let $X\in \SH^G(S)$ be an $N$-split motivic $G$-spectrum. The map
$$
\Theta_X:\bigoplus_{(H)} 
{G/N}_+\ltimes_{\WW H}\left(X^{\mgf H} \right)_{\hh\W_NH}
\to X^N
$$ 
is an equivalence of motivic ${G/N}$-spectra. 
 \end{theorem}
\begin{proof}
Since $G$ is finite there is a sequence of families  
\[
\emptyset = \FF_{-1} \subseteq \FF_{0}\subseteq \FF_{1}\subseteq \cdots \subseteq \FF_{n}=\FF_{all}
\]
such that each pair $\FF_i\subseteq \FF_{i+1}$ is $N$-adjacent, see \aref{sec:filtadj}. 
This gives rise to the filtration of the identity functor
\[
\ast\simeq \EE\FF_{-1+}\otimes - \to \EE\FF_{0+}\otimes-  \to \cdots \to \EE\FF_{n-1+}\otimes \to \EE\FF_{n+}\otimes - \simeq \id.
\]
It this suffices to show that to show that 
$\Theta_{X\otimes \EE(\FF',\FF)}$ is an equivalence whenever $\FF\subseteq\FF'$ is an $N$-adjacent pair.

But if $\FF\subseteq \FF'$ is $N$-adjacent at $H\leq N$,  then  all summands of the domain of $\Theta$  vanish except the summand corresponding to the conjugacy class $(H)$ and $\Theta_{X\otimes\EE\FF(\FF',\FF)}$ is an equivalence by \aref{prop:Phiadj}.
 \end{proof}

\begin{corollary}
Let $Y$ be a based motivic $G$-space over $B$. Then 
\[
\pi_{a,b}^G(\SS_B)\iso \bigoplus_{(H)}\pi_{a,b}(\mathbf{B}\W H_+).
\]
\end{corollary}

%

\vspace{20pt}
\scriptsize
\noindent
David Gepner\\
School of Mathematics and Statistics\\
The University of Melbourne\\
Parkville, VIC, 3010\\
Australia\\
\texttt{david.gepner@unimelb.edu.au}

\vspace{10pt}
\noindent
Jeremiah Heller\\
University of Illinois at Urbana-Champaign\\
Department of Mathematics\\
1409 W. Green Street, Urbana, IL 61801\\
United States\\
\texttt{jbheller@illinois.edu}
\end{document}